\documentclass[11pt]{article}
\usepackage[english]{babel}
\usepackage[T1]{fontenc}
\usepackage[utf8]{inputenc}
\usepackage{amsmath}
\usepackage{amssymb}
\usepackage{amsfonts}
\usepackage{amsthm}
\usepackage{bm}
\usepackage{mathtools}
\usepackage{upgreek}

\newcommand{\eps}{\varepsilon}
\newcommand{\C}{\mathbb{C}}
\newcommand{\N}{\mathbb{N}}

\newcommand{\R}{\mathbb{R}}
\renewcommand{\S}{\mathbb{S}}

\newcommand{\boA}{\mathcal{A}}

\newcommand{\boC}{\mathcal{C}}
\newcommand{\boD}{\mathcal{D}}
\newcommand{\boE}{\mathcal{E}}
\newcommand{\boF}{\mathcal{F}}
\newcommand{\boG}{\mathcal{G}}
\newcommand{\boH}{\mathcal{H}}
\newcommand{\boI}{\mathcal{I}}
\newcommand{\boJ}{\mathcal{J}}

\newcommand{\boL}{\mathcal{L}}

\newcommand{\boN}{\mathcal{N}}
\newcommand{\boO}{\mathcal{O}}

\newcommand{\boR}{\mathcal{R}}

\newcommand{\boT}{\mathcal{T}}
\newcommand{\boU}{\mathcal{U}}
\newcommand{\boV}{\mathcal{V}}

\newcommand{\ga}{\mathfrak{a}}
\newcommand{\gb}{\mathfrak{b}}
\newcommand{\gc}{\mathfrak{c}}

\newcommand{\gv}{\mathfrak{v}}

\DeclareMathOperator{\Dom}{{\rm Dom}}
\renewcommand{\Im}{\mathop{{\rm Im}}\nolimits}
\DeclareMathOperator{\Ker}{{\rm Ker}}
\renewcommand{\Re}{\mathop{{\rm Re}}\nolimits}

\DeclareMathOperator{\Span}{{\rm Span}}

\renewcommand{\th}{{\rm th}}

\newtheorem*{claim*}{Claim}
\newtheorem{claim}{Claim}
\newtheorem{cor}{Corollary}
\newtheorem{lem}{Lemma}
\newtheorem{prop}{Proposition}
\newtheorem{step}{Step}
\newtheorem{thm}{Theorem}

\theoremstyle{definition}

\newtheorem*{merci}{Acknowledgments}
\newtheorem*{remark}{Remark}
\newtheorem*{remarks}{Remarks}
\theoremstyle{remark}

\begin{document}
\title{Asymptotic stability in the energy space for dark solitons of the Landau-Lifshitz equation }
\author{\renewcommand{\thefootnote}{\arabic{footnote}} Yakine Bahri
\footnotemark[1]}
\footnotetext[1]{Centre de Math\'ematiques Laurent Schwartz, \'Ecole polytechnique, 91128 Palaiseau Cedex,
France. E-mail: {\tt yakine.bahri@polytechnique.edu}}
\maketitle
\begin{abstract}
We prove the asymptotic stability in the energy space of non-zero speed solitons for the one-dimensional Landau-Lifshitz equation with an easy-plane anisotropy 
$$\partial_t m + m \times( \partial_{xx} m - m_3 e_3) = 0$$
for a map $m = (m_1, m_2, m_3) : \R \times \R \to \S^2$, where $e_3 = (0, 0, 1)$. More precisely, we show that any solution corresponding to an initial datum close to a soliton with non-zero speed, is weakly convergent in the energy space as time goes to infinity, to a soliton with a possible different non-zero speed, up to the invariances of the equation. Our analysis relies on the ideas developed by Martel and Merle for the generalized Korteweg-de Vries equations. We use the Madelung transform to study the problem in the hydrodynamical framework. In this framework, we rely on the orbital stability of the solitons and the weak continuity of the flow in order to construct a limit profile. We next derive a monotonicity formula for the momentum, which gives the localization of the limit profile. Its smoothness and exponential decay then follow from a smoothing result for the localized solutions of the Schrödinger equations. Finally, we prove a Liouville type theorem, which shows that only the solitons enjoy these properties in their neighbourhoods.
\end{abstract}
\numberwithin{cor}{section}
\numberwithin{lem}{section}
\numberwithin{prop}{section}
\numberwithin{thm}{section}
\section{Introduction}
We consider the one-dimensional Landau-Lifshitz equation
\renewcommand{\theequation}{LL}
\begin{equation}
\label{LL}
\partial_t m + m \times( \partial_{xx} m + \lambda m_3 e_3) = 0,
\end{equation}
for a map $m = (m_1, m_2, m_3) : \R \times \R \to \S^2$, where $e_3 = (0, 0, 1)$ and $\lambda \in \R$. This equation was introduced by Landau and Lifshitz in \cite{LandLif1}. It describes the dynamics of magnetization in a one-dimensional ferromagnetic material, for example in  CsNiF$_3$ or TMNC (see e.g. \cite{KosIvKo1, HubeSch0} and the references therein). The parameter $\lambda$ accounts for the anisotropy of the material. The choices $\lambda > 0$ and $\lambda < 0$ correspond respectively to an easy-axis and an easy-plane anisotropy. In the isotropic case $\lambda = 0$, the equation is exactly the one-dimensional Schrödinger map equation, which has been intensively studied (see e.g. \cite{GuoDing0,JerrSme1}). In this paper, we study the Landau-Lifshitz equation with an easy-plane anisotropy ($\lambda < 0)$. Performing, if necessary, a suitable scaling argument on the map $m$, we assume from now on that $\lambda = - 1$. Our main goal is to prove the asymptotic stability for the solitons of this equation (see Theorem~\ref{thm:stabasympt_m}  below).

The Landau-Lifshitz equation is Hamiltonian. Its Hamiltonian, the so-called Landau-Lifshitz energy, is given by
$$E(m) := \frac{1}{2} \int_\R \big( |\partial_x m|^2 + m_3^2 \big).$$
In the sequel, we restrict our attention to the Hamiltonian framework in which the solutions $m$ to \eqref{LL} have finite Landau-Lifshitz energy, i.e. belong to the energy space
$$\boE(\R) := \big\{ \upsilon : \R \to \S^2, \ {\rm s.t.} \ \upsilon' \in L^2(\R) \ {\rm and} \ \upsilon_3 \in L^2(\R) \big\}.$$
A soliton with speed $c$ is a travelling-wave solution of \eqref{LL} having the form
$$m(x, t) := u(x - c t).$$
Its profile $u$ is a solution to the ordinary differential equation
\renewcommand{\theequation}{TWE}
\begin{equation}
\label{Solc}
u''+ |u'|^2 u + u_3^2 u - u_3 e_3 + c u \times u' = 0.
\end{equation}
The solutions of this equation are explicit. When $|c| \geq 1$, the only solutions with finite Landau-Lifshitz energy are the constant vectors in $\S ^1 \times \{0\}$. In contrast, when $|c| < 1$, there exist non-constant solutions $u_c$ to \eqref{Solc}, which are given by the formulae
\begin{equation*}
[u_c]_1(x) = \frac{c}{\cosh \big( (1 - c^2)^\frac{1}{2} x \big)}, \quad [u_c]_2(x) = \tanh \big( (1 - c^2)^\frac{1}{2} x \big), \quad [u_c]_3(x) = \frac{(1 - c^2)^\frac{1}{2}}{\cosh \big( (1 - c^2)^\frac{1}{2} x \big)},
\end{equation*}
up to the invariances of the problem, i.e. translations, rotations around the axis $x_3$ and orthogonal symmetries with respect to the plane $x_3 = 0$ (see \cite{deLaire4} for more details).

Our goal is to study the asymptotic behaviour for solutions of \eqref{LL} which are initially close to a soliton in the energy space. We endow $\boE(\R)$ with the metric structure corresponding to the distance introduced by de Laire and Gravejat in \cite{DeLGr},
$$d_\boE (f,g):= |\check{f}(0)-\check{g}(0)| +\|f'-g'\|_{L^2(\R)} + \|f_3-g_3\|_{L^2(\R)},$$
where $f=(f_1,f_2,f_3)$ and $\check{f}=f_1+if_2$ (respectively for $g$). The Cauchy problem and the orbital stability of the travelling waves have been solved by de Laire and Gravejat in \cite{DeLGr}. We are concerned the asymptotic stability of travelling waves.
The following theorem is our main result.
\begin{thm}
\label{thm:stabasympt_m}
Let $\gc \in (- 1, 1) \setminus \{ 0 \}$. There exists a positive number $\delta_\gc$, depending only on $\gc$, such that, if
$$d_\boE \big( m^0, u_\gc \big) \leq \delta_\gc,$$
then there exist a number $\gc^* \in (- 1, 1) \setminus \{ 0 \}$, and two functions $b \in \boC^1(\R, \R)$ and $\theta \in \boC^1(\R, \R)$ such that
$$b'(t) \to \gc^*, \quad {\rm and} \quad \theta'(t) \to 0,$$
as $t \to + \infty$, and for which the map
$$ m_{\theta}:= \Big( \cos(\theta)m_1-\sin(\theta)m_2,\sin(\theta)m_1+\cos(\theta)m_2,m_3\Big),$$
satisfies the convergences
\begin{equation*}
\partial_x m_{\theta(t)} \big( \cdot + b(t), t \big) \rightharpoonup \partial_x u_{\gc^*} \quad {\rm in} \ L^2(\R), \quad m_{\theta(t)} \big( \cdot + b(t), t \big) \to u_{\gc^*} \quad {\rm in} \ L^\infty_{\rm loc}(\R),
\end{equation*}
and 
\begin{equation*}
m_3\big( \cdot + b(t), t \big) \rightharpoonup \big[ u_{\gc^*}\big]_3 \quad {\rm in} \ L^2(\R),
\end{equation*}
as $t \to + \infty$.
\end{thm}
\begin{remarks}
$(i)$ Note that the case $\gc=0$, that is black solitons, is excluded from the statement of Theorem \ref{thm:stabasympt_m}. In this case, the map $\check{u}_0$ vanishes and we cannot apply the Madelung transform and the subsequent arguments. Orbital and asymptotic stability remain open problems for this case. Note that, to our knowledge, there is currently no available proof of the local well-posedness of \eqref{LL} in the energy space, when $u_0$ vanishes and so the hydrodynamical framework can no longer be used.\\
$(ii)$ Here, we state a weak convergence result and not a local strong convergence one, like the results given by Martel and Merle for the Korteweg-de Vries equation \cite{MartMer6, MartMer5}. In their situation, they can use two monotonicity formulae for the $L^2$ norm and the energy. This heuristically originates in the property that dispersion has negative speed in the context of the Korteweg de Vries equation. In contrast, the possible group velocities for the dispersion of the Landau-Lifshitz equation are given by $v_g(k) = \pm \frac{1+2k^2}{\sqrt{1+k^2}}$, where $k$ is the wave number. Dispersion has both negative and positive speeds. A monotonicity formula remains for the momentum due to the existence of a gap in the possible group velocities, which satisfy the condition $|v_g(k)|\geq 1.$ However, there is no evidence that one can establish a monotonicity formula for the energy.
\end{remarks}

Similar results were stated by Soffer and Weinstein in \cite{SoffWei1, SoffWei2, SoffWei3}. They provided the asymptotic stability of ground states for the nonlinear Schr\"odinger equation with a potential in a regime for which the nonlinear ground-state is a close continuation of the linear one. They rely on dispersive estimates for the linearized equation around the ground state in suitable weighted spaces, and they apply a fixed point argument. This strategy was successful extended in particular by Buslaev, Perelman, C. Sulem and Cuccagna to the nonlinear Schr\"odinger equations without potential (see e.g. \cite{BuslPer1, BuslPer2, BuslSul1, Cuccagn1}) and with a potential (see e.g. \cite{GangSig1}). We refer to the detailed historical survey by Cuccagna \cite{Cuccagn2} for more details. In addition, asymptotic stability in spaces of exponentially localized perturbations was studied by Pego and Weinstein in \cite{PegoWei1} (see also \cite{Mizumac1}
for perturbations with algebraic decay).

Our strategy for establishing the asymptotic stability result in Theorem \ref{thm:stabasympt_m} is reminiscent from ideas developed by Martel and Merle for the Korteweg-de Vries equation \cite{Martel2, MartMer6, MartMer5}, and successfully adapted by Béthuel, Gravejat and Smets in \cite{BetGrSm1} for the Gross-Pitaevskii equation.

The main steps of the proof are similar to the ones for the Gross-Pitaevskii equation in \cite{BetGrSm2}. Indeed, the solitons of the Landau-Lifshitz equation share many properties with the solitons of the Gross-Pitaevskii equation. In fact, the stereographic variable $ \psi$ defined by
$$\psi=\frac{u_1+i u_2}{1+u_3},$$
verifies the following equation
\begin{equation*}
\partial_{xx} \psi +\frac{1-|\psi|^2}{1+|\psi|^2}\psi -ic \partial_x \psi=\frac{2\bar{\psi}}{1+|\psi|^2}(\partial_x \psi )^2,
\end{equation*}
which can be seen as a perturbation of the equation for the travelling waves of the Gross-Pitaevskii equation, namely
\begin{equation*}
\partial_{xx} \Psi +(1-|\Psi|^2)\Psi -ic\partial_x\Psi=0.
\end{equation*}

However, the analysis of the Landau-Lifshitz equation is much more difficult. Indeed, we rely on a Hasimoto like transform in order to relate the Landau-Lifshitz equation with a nonlinear Schrödinger equation. During so, we lose some regularity. We have to deal with a nonlinear equation at the $L^2$-level and not at the $H^1$-level as in the case of the Gross-Pitaevskii equation. This leads to important technical difficulties.

Coming back to the proof of Theorem \ref{thm:stabasympt_m}, we first translate the problem into the hydrodynamical formulation. Then, we prove the asymptotic stability in that framework. In fact, we begin by refining the orbital stability. Next, we construct a limit profile, which is smooth and localized. For the proof of the exponential decay of the limit profile, we cannot rely on the Sobolev embedding $H^1$ into $L^\infty$ as it was done in \cite{BetGrSm2}. We use instead the results of Kenig, Ponce and Vega in \cite{KenPoVe9}, and the Gagliardo-Niremberg inequality (see the proof of Proposition \ref{prop:smooth} for more details). We also have to deal with the weak continuity of the flow in order to construct the limit profile. For the Gross Pitaevskii equation, this property relies on the uniqueness in a weaker space (see \cite{BetGrSm2}). There is no similar result at the $L^2$-level. Instead, we use the Kato smoothing effect. The asymptotic stability in the hydrodynamical variables then follows from a Liouville type theorem. It shows that the only smooth and localized solutions in the neighbourhood of the solitons are the solitons. Finally, we deduce the asymptotic stability in the original setting from the result in the hydrodynamical framework.

In Section 2 below, we explain the main tools and different steps for the proof. First, we introduce the hydrodynamical framework. Then, we state the orbital stability of the solitons under a new orthogonality condition. Next, we sketch the proof of the asymptotic stability for the hydrodynamical system and we state the main propositions. We finally complete the proof of Theorem \ref{thm:stabasympt_m}.

In Section 3 to 5, we give the proofs of the results stated in Section 2. In Section 3, We deal with the orbital stability in the hydrodynamical framework. In Section 4, we prove the localization and the smoothness of the limit profile. In the last section, we prove a Liouville type theorem. In a separate appendix, we show some facts used in the proofs, in particular, the weak continuity of the \eqref{HLL} flow.  
\section{Main steps for the proof of Theorem \ref{thm:stabasympt_m}}
\subsection{The hydrodynamical framework}
We introduce the map $\check{m} := m_1 + i m_2.$ Since $m_3$ belongs to $H^1(\R)$, it follows from the Sobolev embedding theorem that
$$|\check{m}(x)| = (1 - m_3^2(x))^\frac{1}{2} \to 1,$$
as $x \to \pm \infty$. As a consequence, the Landau-Lifshitz equation shares many properties with the Gross-Pitaevskii equation (see e.g. \cite{BetGrSa2}).
One of these properties is the existence of an hydrodynamical framework for the Landau-Lifshitz equation. In terms of the maps $\check{m}$ and $m_3$, this equation may be written as
$$\Bigg\{ \begin{array}{ll}
i \partial_t \check{m} - m_3 \partial_{xx} \check{m} + \check{m} \partial_{xx} m_3 - \check{m} m_3 = 0,\\[5pt]
\partial_t m_3 + \partial_x \big\langle i \check{m}, \partial_x \check{m} \big\rangle_\C = 0.
\end{array}$$
When the map $\check{m}$ does not vanish, one can write it as $\check{m} = (1 - m_3^2)^{1/2} \exp i \varphi$. The hydrodynamical variables $v := m_3$ and $w := \partial_x \varphi$ verify the following system
\renewcommand{\theequation}{HLL}
\begin{equation}
\label{HLL}
\left\{ \begin{array}{ll}
\partial_t v = \partial_x \big( (v^2 - 1) w \big),\\[5pt]
\displaystyle \partial_t w = \partial_x \Big( \frac{\partial_{xx} v}{1 - v^2} + v \frac{(\partial_x v)^2}{(1 - v^2)^2} + v \big( w^2 - 1) \Big).
\end{array} \right.
\end{equation}
This system is similar to the hydrodynamical Gross-Pitaevskii equation (see e.g. \cite{BetGrSm2}).\footnote{The hydrodynamical terminology originates in the fact that the hydrodynamical Gross-Pitaevskii equation is similar to the Euler equation for an irrotational fluid (see e.g. \cite{BetGrSm1}).} We first study the asymptotic stability in the hydrodynamical framework.

In this framework, the Landau-Lifshitz energy is expressed as
\renewcommand{\theequation}{\arabic{equation}}

\numberwithin{equation}{section}
\setcounter{equation}{0}
\begin{equation}
\label{eq:hydro-E}
E(\gv) := \int_\R e(\gv) := \frac{1}{2} \int_\R \Big( \frac{(v')^2}{1 - v^2} + \big( 1 - v^2 \big) w^2 + v^2 \Big),
\end{equation}
where $\gv := (v, w)$ denotes the hydrodynamical pair. The momentum $P$, defined by
\begin{equation}
\label{def:P}
P(\gv) := \int_\R v w,
\end{equation}
is also conserved by the Landau-Lifshitz flow. The momentum $P$ and the Landau-Lifshitz energy $E$ play an important role in the study of the asymptotic stability of the solitons.
When $c \neq 0$, the function $\check{u}_c$ does not vanish. The hydrodynamical pair $Q_c := (v_c, w_c)$ is given by
\begin{equation}
\label{form:vc}
v_c(x) = \frac{(1 - c^2)^\frac{1}{2}}{\cosh \big( (1 - c^2)^\frac{1}{2} x \big)}, \quad {\rm and} \quad w_c(x) = \frac{c \, v_c(x)}{1 - v_c(x)^2} = \frac{c (1 - c^2)^\frac{1}{2} \cosh \big( (1 - c^2)^\frac{1}{2} x \big)}{\sinh \big( (1 - c^2)^\frac{1}{2} x \big)^2 + c^2}.
\end{equation}
The only invariances of \eqref{HLL} are translations and the opposite map $(v, w) \mapsto (- v, - w)$. We restrict our attention to the translation invariances. All the analysis developed below applies when the opposite map is also taken into account. For $a \in \R$, we denote 
$$Q_{c, a}(x) := Q_c(x - a) := \big(  v_c(x - a), w_c(x - a) \big),$$
a non-constant soliton with speed $c
$. 
We also set 
$$\boN\boV(\R) := \Big\{ \gv = (v, w) \in H^1(\R) \times L^2(\R), \ {\rm s.t.} \ \max_\R |v| < 1 \Big\}.$$
This non-vanishing space is endowed in the sequel with the metric structure provided by the norm
$$\| \gv \|_{H^1 \times L^2} := \Big( \| v \|_{H^1}^2 + \| w \|_{L^2}^2 \Big)^\frac{1}{2}.$$
\subsection{Orbital stability}
A perturbation of a soliton is provided by another soliton with a slightly different speed. This property follows from the existence of a continuum of solitons with different speeds. A solution corresponding to such a perturbation at initial time diverges from the soliton due to the different speeds of propagation, so that the standard notion of stability does not apply to solitons. The notion of orbital stability is tailored to deal with such situations. The orbital stability theorem below shows that a perturbation of a soliton at initial time remains a perturbation of the soliton, up to translations, for all time.
 
 The following theorem is a variant of the result by de Laire and Gravejat \cite{DeLGr} concerning sums of solitons. It is useful for the proof of the asymptotic stability.
\begin{thm}
\label{thm:orbistab}
Let $c \in (- 1, 1) \setminus \{ 0 \}$. There exists a positive number $\alpha_c$, depending only on $c$, with the following properties. Given any $(v_0,w_0) \in X(\R):=H^1(\R) \times L^2(\R)$ such that
\begin{equation}
\label{cond:alpha}
\alpha_0 := \big\| (v_0,w_0) - Q_{c, a} \big\|_{X(\R)} \leq \alpha_c,
\end{equation}
for some $a \in \R$, there exist a unique global solution $(v,w) \in \boC^0(\R, \boN\boV(\R))$ to \eqref{HLL} with initial datum $(v_0,w_0)$, and two maps $c \in \boC^1(\R, (- 1, 1) \setminus \{ 0 \})$ and $a \in \boC^1(\R, \R)$ such that the function $\eps$ defined by
\begin{equation}
\label{def:eps}
\eps(\cdot, t) := \big( v(\cdot + a(t), t), w(\cdot + a(t), t) \big) - Q_{c(t)},
\end{equation}
satisfies the orthogonality conditions
\begin{equation}
\label{eq:ortho}
\langle \eps(\cdot, t), \partial_x Q_{c(t)} \rangle_{L^2(\R)^2} = \langle \eps(\cdot, t), \chi_{c(t)} \rangle_{L^2(\R)^2} = 0,
\end{equation}
for any $t \in \R$. Moreover, there exist two positive numbers $\sigma_c$ and $A_c$, depending only and continuously on $c$, such that
\begin{equation}
\label{eq:max-v}
\max_{x \in \R} v(x, t) \leq 1 - \sigma_c,
\end{equation}
\begin{equation}
\label{eq:modul0}
\big\| \eps(\cdot, t) \big\|_{X(\R)} + \big| c(t) - c \big| \leq A_c \alpha^0,\\
\end{equation}
and
\begin{equation}
\label{eq:modul1}
\big| c'(t) \big| + \big| a'(t) - c(t) \big| \leq A_\gc \big\| \eps(\cdot, t) \big\|_{X(\R)},
\end{equation}
for any $t\in \R$.
\end{thm}
\begin{remark}
In this statement, the function $\chi_c$ is a normalized eigenfunction associated to the unique negative eigenvalue of the linear operator 
$$\boH_c := E''(Q_c) + c P''(Q_c).$$
The operator $\boH_c$ is self-adjoint on $L^2(\R) \times L^2(\R)$, with domain $\Dom(\boH_c) := H^2(\R) \times L^2(\R)$ (see \eqref{def:boHc} for its explicit formula). It has a unique negative simple eigenvalue $-\tilde{\lambda}_c$, and its kernel is given by
\begin{equation}
\label{eq:Ker-Hc}
\Ker(\boH_c) = \Span(\partial_x Q_c).
\end{equation}
\end{remark}

Our statement of orbital stability relies on a different decomposition from that proposed by Grillakis, Shatah and Strauss in \cite{GriShSt1}. This modification is related to the proof of asymptotic stability. A key ingredient in the proof is the coercivity of the quadratic form $G_c$, which is defined in \eqref{eq:loc-virial}, under a suitable orthogonality condition. In case we use the orthogonality conditions in \cite{GriShSt1}, the corresponding orthogonality condition for $G_c$ is provided by the function $v^{-1}_c S \partial_c Q_c$ (see \eqref{def:J} for the definition of $S$), which does not belong to $L^2(\R)$. In order to by-pass this difficulty, we use the second orthogonality condition in \eqref{eq:ortho} for which the corresponding orthogonality condition for $G_c$ is given by the function $v^{-1}_c S  \chi_c$, which does belong to $L^2(\R)$ (see the appendix for more details). This alternative decomposition is inspired from the one used by Martel and Merle in \cite{MartMer6}.

 Concerning the proof of Theorem \ref{thm:orbistab}, we first establish an orbital stability theorem with the classical decomposition of Grillakis, Shatah and Strauss \cite{GriShSt1}. This appears as a particular case of the orbital stability theorem in \cite{DeLGr} for sum of solitons. We next show that, if we have orbital stability for some decomposition and orthogonality conditions, then we also have it for different decomposition and orthogonality conditions (see Section 2 for the detailed proof of Theorem \ref{thm:orbistab}).

\subsection{Asymptotic stability for the hydrodynamical variables}
The following theorem shows the asymptotic stability result in the hydrodynamical framework.
\begin{thm}
\label{thm:stabasympt}
Let $\gc \in (- 1, 1) \setminus \{ 0 \}$. There exists a positive constant $\beta_\gc \leq \alpha_\gc$, depending only on $\gc$, with the following properties. Given any $(v_0, w_0) \in X(\R)$ such that
$$\big\| (v_0, w_0) - Q_{\gc, \ga} \big\|_{X(\R)} \leq \beta_\gc,$$
for some $\ga \in \R$, there exist a number $\gc^* \in (- 1, 1) \setminus \{ 0 \}$ and a map $b \in \boC^1(\R, \R)$ such that the unique global solution $(v, w) \in \boC^0(\R, \boN\boV(\R))$ to \eqref{HLL} with initial datum $(v_0, w_0)$ satisfies
\begin{equation}
\label{conv-stab-asymp}
\big( v(\cdot + b(t), t), w(\cdot + b(t), t) \big) \rightharpoonup Q_{\gc^*} \quad {\rm in} \ X(\R),
\end{equation}
and
$$b'(t) \to \gc^*,$$
as $t \to + \infty$.
\end{thm}
Theorem \ref{thm:stabasympt} establishes a convergence to some orbit of the soliton. This result is stronger than the one given by Theorem \ref{thm:orbistab} which only shows that the solution stays close to that orbit.

In the next subsections, we explain the main ideas of the proof, which follows the strategy developed by Martel and Merle for the Korteweg-de Vries equation \cite{MartMer6, MartMer5}.
\subsubsection{Construction of a limit profile}
Let $\gc \in (- 1, 1) \setminus \{ 0 \}$, and $(v_0, w_0) \in X(\R)$ be any pair satisfying the assumptions of Theorem \ref{thm:stabasympt}. Since $\beta_\gc \leq \alpha_\gc$ in the assumptions of Theorem \ref{thm:stabasympt}, we deduce from Theorem \ref{thm:orbistab} that the unique solution $(v, w)$ to \eqref{HLL} with initial datum $(v_0, w_0)$ is global.

We take an arbitrary sequence of times $(t_n)_{n \in \N}$ tending to $+ \infty$. In view of \eqref{eq:modul0} and \eqref{eq:modul1}, we may assume, up to a subsequence, that 
there exist a limit perturbation $\eps_0^* \in X(\R)$ and a limit speed $c_0^* \in [- 1, 1]$ such that
\begin{equation}
\label{eq:assump1}
\eps(\cdot, t_n) = \big( v(\cdot + a(t_n), t_n), w(\cdot + a(t_n), t_n) \big) - Q_{c(t_n)} \rightharpoonup \eps_0^* \quad {\rm in} \ X(\R),
\end{equation}
and
\begin{equation}
\label{eq:assump2}
c(t_n) \to c_0^*,
\end{equation}
as $n \to + \infty$. Our main goal is to show that
$$\eps_0^* \equiv 0,$$
(see Corollary \ref{cor:cayest}). For that, we establish smoothness and rigidity properties for the solution of \eqref{HLL} with the initial datum $Q_{c_0^*} + \eps_0^*$.

First, we impose the constant $\beta_\gc$ to be sufficiently small so that, when the number $\alpha^0$ which appears in Theorem \ref{thm:orbistab} satisfies $\alpha^0 \leq \beta_\gc$, then we infer from \eqref{eq:modul0} and \eqref{eq:modul1} that
\begin{equation}
\label{eq:lvmh1}
\min \big\{ c(t)^2, a'(t)^2 \big\} \geq \frac{\gc^2}{2}, \qquad \max \big\{ c(t)^2, a'(t)^2 \big\} \leq 1 + \frac{\gc^2}{2}, 
\end{equation}
and
\begin{equation}
\label{eq:lvmh2}
\big\| v_\gc(\cdot) - v(\cdot + a(t), t) \big\|_{L^\infty(\R)} \leq \min \Big\{ \frac{\gc^2}{4}, \frac{1 - \gc^2}{16} \Big\},
\end{equation}
for any $t\in \R$. This yields, in particular, that $c_0^* \in (- 1, 1) \setminus \{ 0 \}$, and then, that $Q_{c_0^*}$ is well-defined and different from the black soliton.

By \eqref{eq:modul0}, we also have 
\begin{equation}
\label{eq:encorezero}
\big| c_0^* - \gc \big| \leq A_\gc \beta_\gc,
\end{equation}
and, applying again \eqref{eq:modul0}, as well as \eqref{eq:assump1}, and the weak lower semi-continuity of the norm, we also know that the function
$$(v_0^*, w_0^*) := Q_{c_0^*} + \eps_0^*,$$
satisfies
\begin{equation}
\label{eq:encoreune}
\big\| (v_0^*, w_0^*) - Q_\gc \big\|_{X(\R)} \leq A_\gc \beta_\gc + \big\| Q_\gc - Q_{c_0^*} \big\|_{X(\R)}.
\end{equation}
We next impose a supplementary smallness assumption on $\beta_\gc$ so that 
\begin{equation}
\label{cond:alpha*}
\big\| (v_0^*, w_0^*) - Q_\gc \big\|_{X(\R)} \leq \alpha_\gc.
\end{equation}
By Theorem \ref{thm:orbistab}, there exists a unique global solution $(v^*, w^*) \in
\boC^0(\R, \boN\boV(\R))$ to \eqref{HLL} with initial datum $(v_0^*, w_0^*)$, and two maps $c^* \in \boC^1(\R, (- 1, 1) \setminus \{ 0 \})$ and $a^* \in \boC^1(\R, \R)$ such that the function $\eps^*$ defined by
\begin{equation}
\label{def:eps*}
\eps^*(\cdot, t) := \big( v^*(\cdot + a^*(t), t), w(\cdot + a^*(t), t) \big) - Q_{c^*(t)},
\end{equation}
satisfies the orthogonality conditions
\begin{equation}
\label{eq:orthobis}
\langle \eps^*(\cdot, t), \partial_x Q_{c^*(t)} \rangle_{L^2(\R)^2} = \langle \eps^*(\cdot, t), \chi_{c^*(t)} \rangle_{L^2(\R)^2} = 0,
\end{equation}
as well as the estimates
\begin{equation}
\label{eq:modul0bis}
\big\| \eps^*(\cdot, t) \big\|_{X(\R)} + \big| c^*(t) - \gc \big|+ \big| a{^*}'(t) - c^*(t) \big| \leq A_\gc \big\| (v_0^*, w_0^*) -Q_{\gc} \big\|_{X(\R)},
\end{equation}
for any $t\in \R$.

We may take $\beta_\gc$ small enough such that, combining \eqref{eq:encorezero} with \eqref{eq:encoreune} and \eqref{eq:modul0bis}, we obtain
\begin{equation}
\label{eq:lvmh1bis}
\min \big\{ c^*(t)^2, (a^*)'(t)^2 \big\} \geq \frac{\gc^2}{2}, \qquad \max \big\{ c^*(t)^2, (a^*)'(t)^2 \big\} \leq 1 + \frac{\gc^2}{2}, 
\end{equation}
and 
\begin{equation}
\label{eq:lvmh2bis}
\big\| v_\gc(\cdot) - v^*(\cdot + a^*(t), t) \big\|_{L^\infty(\R)} \leq \min \Big\{ \frac{\gc^2}{4}, \frac{1 - \gc^2}{16} \Big\},
\end{equation}
for any $t\in \R$.

Finally, we use the weak continuity of the flow map for the Landau-Lifshitz equation. The proof relies on Proposition \ref{prop:w-cont-Q} and follows the lines the proof of Proposition 1 in \cite{BetGrSm2}.

\begin{prop}
\label{prop:reprod}
Let $t \in \R$ be fixed. Then,
\begin{equation}
\label{sologne1}
\big( v(\cdot + a(t_n), t_n + t), w(\cdot + a(t_n), t_n + t) \big) \rightharpoonup \big( v^*(\cdot, t), w^*(\cdot, t) \big) \quad {\rm in} \ X(\R),
\end{equation}
while
\begin{equation}
\label{sologne2}
a(t_n + t) - a(t_n) \to a^*(t), \quad {\rm and} \quad c(t_n + t) \to c^*(t),
\end{equation}
as $n \to + \infty$. In particular, we have
\begin{equation}
\label{sologne3}
\eps(\cdot, t_n + t) \rightharpoonup \eps^*(\cdot, t) \quad {\rm in} \ X(\R),
\end{equation}
as $n \to + \infty$.
\end{prop}
\subsubsection{Localization and smoothness of the limit profile}

Our proof of the localization of the limit profile is based on a monotonicity formula.

Consider a pair $(v, w)$ which satisfies the conclusions of Theorem \ref{thm:orbistab} and suppose that \eqref{eq:lvmh1} and \eqref{eq:lvmh2} are true. Let $R$ and $t$ be two real numbers, and set
$$I_R(t) \equiv I_R^{(v, w)}(t) := \frac{1}{2} \int_\R \big[ v w \big](x + a(t), t) \Phi(x - R) \, dx,$$
where $\Phi$ is the function defined on $\R$ by 
\begin{equation}
\label{eq:defiphi}
\Phi(x) := \frac{1}{2} \Big( 1 + \th \big( \nu_\gc x \big) \Big),
\end{equation}
with $\nu_\gc := \sqrt{1 - \gc^2}/8$. We have

\begin{prop}
\label{prop:mono}
Let $R \in \R$, $t \in \R$, and $\sigma \in [- \sigma_\gc, \sigma_\gc]$, with 
$\sigma_\gc := \sqrt{1 - \gc^2}/4$. Under the above assumptions, there exists a positive number $B_\gc$, depending only on $\gc$, such that
\begin{equation}
\label{eq:mono}
\begin{split}
\frac{d}{dt} \big[ I_{R + \sigma t}(t) \big] \geq & \frac{1 - \gc^2}{8} \int_\R \big[ (\partial_x v )^2 + v^2 + w^2 \big](x + a(t), t) \Phi'(x - R - \sigma t) \, dx\\
& - B_\gc e^{- 2 \nu_\gc |R + \sigma t|}.
\end{split}
\end{equation}
In particular, we have
\begin{equation}
\label{eq:monobis}	
I_R(t_1) \geq I_R(t_0) - B_c e^{- 2 \nu_\gc |R|},
\end{equation}
for any real numbers $t_0 \leq t_1$.
\end{prop}

For the limit profile $(v^*, w^*)$, we set $I_R^*(t) := I_R^{(v^*, w^*)}(t)$ for any $R \in \R$ and any $t \in \R$. We claim

\begin{prop}[\cite{BetGrSm2}]
\label{prop:local0}
Given any positive number $\delta$, there exists a positive number $R_\delta$, depending only on $\delta$, such that we have
\begin{align*}
\big| I_R^*(t) \big| \leq \delta, & \quad \forall R \geq R_\delta,\\
\big| I_R^*(t) - P(v^*, w^*) \big| \leq \delta, & \quad \forall R \leq - R_\delta,
\end{align*}
for any $t \in \R$.
\end{prop}

The proof of Proposition \ref{prop:local0} is the same as the one of Proposition 3 in \cite{BetGrSm2}.

From Propositions \ref{prop:mono} and \ref{prop:local0}, we derive as in \cite{BetGrSm2} that
\begin{prop}[\cite{BetGrSm2}]
\label{prop:local}
Let $t\in \R$. There exists a positive constant $\boA_\gc$ such that
$$\int_t^{t + 1} \int_\R \big[ (\partial_x v^*)^2 + (v^*)^2 + (w^*)^2 \big](x + a^*(s), s) e^{2 \nu_\gc |x|} \, dx \, ds \leq \boA_\gc.$$
\end{prop}

We next consider the following map which was introduced by de Laire and Gravejat in \cite{DeLGr},
\begin{equation}
\label{def:Psi}
\Psi := \frac{1}{2} \Big( \frac{\partial_x v}{(1 - v^2)^\frac{1}{2}} + i (1 - v^2)^\frac{1}{2} w \Big) \exp i \theta,
\end{equation}
where
\begin{equation}
\label{def:theta}
\theta(x, t) := - \int_{- \infty}^x v(y, t) w(y, t) \, dy.
\end{equation}
The map $\Psi$ solves the nonlinear Schr\"odinger equation
\begin{equation}
\label{eq:Psi}
i \partial_t \Psi + \partial_{xx} \Psi + 2 |\Psi|^2 \Psi + \frac{1}{2} v^2 \Psi - \Re \Big( \Psi \big( 1 - 2 F(v, \overline{\Psi}) \big) \Big) \big( 1 - 2 F(v, \Psi) \big) = 0,
\end{equation}
with
\begin{equation}
\label{def:F1}
F(v, \Psi)(x, t) := \int_{- \infty}^x v(y, t) \Psi(y, t) \, dy,
\end{equation}
while the function $v$ satisfies the two equations
\begin{equation}
\label{eq:v-Psi}
\left\{ \begin{array}{l}
\partial_t v = 2 \partial_x \Im \Big( \Psi \big( 2 F(v, \overline{\Psi}) - 1 \big) \Big),\\[5pt]
\partial_x v = 2 \Re \Big( \Psi \big( 1 - 2 F(v, \overline{\Psi}) \big) \Big).
\end{array} \right.
\end{equation}
The local Cauchy problem for \eqref{eq:Psi}-\eqref{eq:v-Psi} was analyzed by de Laire and Gravejat in \cite{DeLGr}. We recall the following proposition which shows the continuous dependence with respect to the initial datum of the solutions to the system of equations \eqref{eq:Psi}-\eqref{eq:v-Psi} (see \cite{DeLGr} for the proof).
\begin{prop}[\cite{DeLGr}]
\label{prop:cont-Cauchy}
Let $(v^0, \Psi^0) \in H^1(\R) \times L^2(\R)$ and $(\tilde{v}^0, \tilde{\Psi}^0) \in H^1(\R) \times L^2(\R)$ be such that
$$\partial_x v^0 = 2 \Re \Big( \Psi^0 \big( 1 - 2 F(v^0, \overline{\Psi^0}) \big) \Big), \quad {\rm and} \quad \partial_x \tilde{v}^0 = 2 \Re \Big( \tilde{\Psi}^0 \Big( 1 - 2 F \big( \tilde{v}^0, \overline{\tilde{\Psi}^0} \big) \Big) \Big).$$
Given two solutions $(v, \Psi)$ and $(\tilde{v}, \tilde{\Psi})$ in $\boC^0([0, T_*], H^1(\R) \times L^2(\R))$, with $(\Psi, \tilde{\Psi}) \in L^4([0, T_*], \linebreak[0] L^\infty(\R))^2$, to \eqref{eq:Psi}-\eqref{eq:v-Psi} with initial datum $(v^0, \Psi^0)$, resp. $(\tilde{v}^0, \tilde{\Psi}^0)$, for some positive time $T_*$, there exist a positive number $\tau$, depending only on $\| v^0 \|_{L^2}$, $\| \tilde{v}^0 \|_{L^2}$, $\| \Psi^0 \|_{L^2}$ and $\| \tilde{\Psi}^0 \|_{L^2}$, and a universal constant $A$ such that we have
\begin{equation}
\label{mali}
\begin{split}
\big\| v - \tilde{v} \big\|_{\boC^0([0, T], L^2)} + \big\| \Psi - \tilde{\Psi} \big\|_{\boC^0([0, T], L^2)} + & \big\| \Psi - \tilde{\Psi} \big\|_{L^4([0, T], L^\infty)}\\
\leq & A \Big( \big\| v^0 - \tilde{v}^0 \big\|_{L^2} + \big\| \Psi^0 - \tilde{\Psi}^0 \big\|_{L^2} \big),
\end{split}
\end{equation}
for any $T \in [0, \min \{ \tau, T_* \}]$. In addition, there exists a positive number $B$, depending only on $\| v^0 \|_{L^2}$, $\| \tilde{v}^0 \|_{L^2}$, $\| \Psi^0 \|_{L^2}$ and $\| \tilde{\Psi}^0 \|_{L^2}$, such that
\begin{equation}
\label{burkina}
\begin{split}
\big\| \partial_x v - \partial_x \tilde{v} \big\|_{\boC^0([0, T], L^2)} \leq B \Big( \| v^0 - \tilde{v}^0 \big\|_{L^2} + & \big\| \Psi^0 - \tilde{\Psi}^0 \big\|_{L^2} \Big),
\end{split}
\end{equation}
for any $T \in [0, \min \{ \tau, T_* \}]$.
\end{prop}
This proposition plays an important role in the proof of not only the smoothing of the limit profile, but also the weak continuity of the hydrodynamical Landau-Lifshitz flow.

In order to prove the smoothness of the limit profile, we rely on the following smoothing type estimate for localized solutions of the linear Schr\"odinger equation (see \cite{BetGrSm2, EsKePoV5} for the proof of Proposition \ref{prop:smoothing}).
\begin{prop}[\cite{BetGrSm2, EsKePoV5}]
\label{prop:smoothing}
Let $\lambda \in \R$ and consider a solution $u \in \boC^0(\R, L^2(\R))$ to the linear \\ Schrödinger equation
\renewcommand{\theequation}{LS}
\begin{equation}
\label{eq:LS}
i \partial_t u + \partial_{xx} u = F,
\end{equation}
with $F \in L^2(\R, L^2(\R))$. Then, there exists a positive constant $K_\lambda$, depending only on $\lambda$, such that
\renewcommand{\theequation}{\arabic{equation}}
\numberwithin{equation}{section}
\setcounter{equation}{36}
\begin{equation}
\label{eq:smoothing}
 \lambda^2 \int_{- T}^T \int_\R |\partial_x u(x, t)|^2 e^{\lambda x} \, dx \, dt \leq K_\lambda \int_{- T - 1}^{T + 1} \int_\R \Big( |u(x, t)|^2 + |F(x, t)|^2 \Big) e^{\lambda x} \, dx \, dt,
\end{equation}
for any positive number $T$.
\end{prop}

We apply Proposition \ref{prop:smoothing} to $\Psi^*$ as well as all its derivatives, where $\Psi^*$ is the solution to \eqref{eq:Psi} associated to the solution $(v^*, w^*)$ of \eqref{HLL}, and then we express the result in terms of $(v^*, w^*)$ to obtain
\begin{prop}
\label{prop:smooth}
The pair $(v^*, w^*)$ is indefinitely smooth and exponentially decaying on $\R \times\R$. Moreover, given any $k \in \N$, there exists a positive constant $A_{k, \gc}$, depending only on $k$ and $\gc$, such that
\begin{equation}
\label{eq:smooth}
\int_\R \big[ (\partial_x^{k + 1} v^*)^2 + (\partial^k_x v^*)^2 + (\partial_x^k w^*)^2 \big](x + a^*(t), t) e^{ \nu_\gc |x|} \, dx \leq A_{k, \gc},
\end{equation}
for any $t\in \R$.
\end{prop}

\subsubsection{The Liouville type theorem}
\label{sub:rigidity}

We next establish a Liouville type theorem, which guarantees that the limit profile constructed above is exactly a soliton. In particular, we will show that $\eps_0^* \equiv 0$.

The pair $\eps^*$ satisfies the equation
\begin{equation}
\label{eq:pourepsbis}
\partial_t \eps^* = J \boH_{c^*(t)}(\eps^*) + J \boR_{c^*(t)} \eps^* + \big( {a^*}'(t) - c^*(t) \big) \big( \partial_x Q_{c^*(t)} + \partial_x \eps^* \big) - {c^*}'(t) \partial_c Q_{c^*(t)},
\end{equation}
where $J$ is the symplectic operator
\begin{equation}
\label{def:J}
J = - 2 S \partial_x := \begin{pmatrix} 0 & - 2 \partial_x \\ - 2 \partial_x & 0 \end{pmatrix},
\end{equation}
and the remainder term $\boR_{c^*(t)} \eps^*$ is given by
$$\boR_{c^*(t)} \eps^* := E'(Q_{c^*(t)} + \eps^*) - E'(Q_{c^*(t)}) - E''(Q_{c^*(t)})(\eps^*).$$
We rely on the strategy developed by Martel and Merle in \cite{MartMer6} (see also \cite{Martel2}), and then applied by Béthuel, Gravejat and Smets in \cite{BetGrSm2} to the Gross-Pitaevskii equation. We define the pair
\begin{equation}
\label{def:u*}
u^*(\cdot, t) := S \boH_{c^*(t)}(\eps^*(\cdot, t)).
\end{equation}
Since $S \boH_{c^*(t)}(\partial_x Q_{c^*(t)}) = 0$, we deduce from \eqref{eq:pourepsbis} that
\begin{equation}
\label{eq:pouru*}
\begin{split}
\partial_t u^* = & \ S \boH_{c^*(t)} \big( J S u^* \big) + S \boH_{c^*(t)} \big( J \boR_{c^*(t)} \eps^* \big) - (c^*)'(t) S \boH_{c^*(t)}(\partial_c Q_{c^*(t)})\\
& \ + (c^*)'(t) S \partial_c \boH_{c^*(t)}(\eps^*) + \big( (a^*)'(t) - c^*(t) \big) S \boH_{c^*(t)}(\partial_x \eps^*).
\end{split}
\end{equation}

Decreasing further the value of $\beta_\gc$ if necessary, we have

\begin{prop}
\label{prop:monou1}
There exist two positive numbers $A_*$ and $R_*$, depending only on $\gc$, such that we have
\footnote{In \eqref{eq:petitplateau}, we use the notation
$$\big\| (f, g) \big\|_{X(\Omega)}^2 : = \int_\Omega \Big( (\partial_x f)^2 + f^2 + g^2 \Big),$$
in which $\Omega$ denotes a measurable subset of $\R$.}
\begin{equation}
\label{eq:petitplateau}
\frac{d}{dt} \bigg( \int_\R x u^*_1(x, t) u^*_2(x, t) \, dx \bigg) \geq \frac{1 - \gc^2}{16} \big\| u^*(\cdot, t) \big\|_{X(\R)}^2 - A_* \| u^*(\cdot, t)\|_{X(B(0, R_*))}^2,
\end{equation}
for any $t \in \R$. 
\end{prop}

We give a second monotonicity type formula to dispose of the non-positive local term \\ $\| u^*(\cdot, t) \|_{X(B(0, R_*))}^2$ in the right-hand side of \eqref{eq:petitplateau}. If $M$ is a smooth, bounded, two-by-two \\
symmetric matrix-valued function, then 
\begin{equation}
\label{eq:pignon}
\frac{d}{dt} \big\langle M u^*, u^* \big\rangle_{L^2(\R)^2} = 2 \big\langle S M u^*,\boH_{c^*}(-2 u^*) \big\rangle_{L^2(\R)^2} + \text{``super-quadratic terms'',}
\end{equation}
where $S$ is the matrix
$$S := \begin{pmatrix} 0 & 1 \\ 1 & 0 \end{pmatrix}.$$
For $c \in (- 1, 1) \setminus \{ 0 \}$, let $M_c$ be given by
\begin{equation}
\label{def:Mc}
M_c := \begin{pmatrix} - \frac{2c v_c \partial_x v_c}{ (1 - v_c)^2} & - \frac{\partial_x v_c}{v_c} \\ - \frac{\partial_x v_c}{v_c} & 0 \end{pmatrix}.
\end{equation}

We have the following lemma.
\begin{lem}
\label{lem:loc-coer}
Let $c \in (- 1, 1) \setminus \{ 0 \}$ and $u \in X^3(\R)$. Then,
\begin{equation}
\label{eq:loc-virial}
\begin{split}
G_c(u) := & 2 \big\langle S M_c u, \boH_c(-2\partial_x u) \big\rangle_{L^2(\R)^2}\\
= & 2 \int_\R \mu_c \Big( u_2 - \frac{c v^2_c}{\mu_c} u_1 - \frac{2c v_c \partial_x v_c}{\mu_c (1-v^2_c) } \partial_x u_1 \Big)^2 + 3 \int_\R \frac{v_c^4}{\mu_c} \Big( \partial_x u_1 - \frac{\partial_x v_c}{v_c} u_1 \Big)^2,
\end{split}
\end{equation}
where 
\begin{equation}
\label{def:mu_c}
\mu_c = 2 (\partial_x v_c)^2 + v_c^2(1-v_c^2) > 0.
\end{equation}
\end{lem}

The functional $G_c$ is a non-negative quadratic form, and
\begin{equation}
\label{ker-Gc}
\Ker(G_c) = \Span(Q_c).
\end{equation}
We have indeed chosen the matrix $M_c$ such that $M_c Q_c = \partial_x Q_c$ to obtain \eqref{ker-Gc}. Since $Q_c$ does not vanish, we deduce from standard Sturm-Liouville theory, that $G_c$ is non-negative, which is confirmed by the computation in Lemma \ref{lem:loc-coer}.

By the second orthogonality condition in \eqref{eq:orthobis} and the fact that $\boH_{c^*}(\chi_{c^*}) = -\tilde{\lambda}_{c^*} \chi_{c^*}$, we have
\begin{equation}
\label{eq:ortho-u*}
0 = \langle \boH_{c^*}(\chi_{c^*}), \eps^* \rangle_{L^2(\R)^2} = \langle \boH_{c^*}(\eps^*),\chi_{c^*} \rangle_{L^2(\R)^2} = \langle u^*, S \chi_{c^*} \rangle_{L^2(\R)^2}.
\end{equation}
On the other hand, we know that
\begin{equation}
\label{eq:angle}
\big\langle Q_{c^*} , S \chi_{c^*} \big\rangle =  P'\big(Q_{c^*}\big)\big(\chi_{c^*}\big) \neq 0,
\end{equation}
so that the pair $u^*$ is not proportional to $Q_{c^*}$ under the orthogonality condition in \eqref{eq:ortho-u*}. We claim the following coercivity property of $G_c$ under this orthogonality condition.
\begin{prop}
\label{prop:coer-Gc}
Let $c \in (- 1, 1) \setminus \{ 0 \}$. There exists a positive number $\Lambda_c$, depending only and continuously on $c$, such that
\begin{equation}
\label{eq:coer-Gc}
G_c(u) \geq \Lambda_c \int_\R \big[ (\partial_x u_1)^2 + (u_1)^2 + (u_2)^2 \big] (x) e^{- 2 |x|} \, dx,
\end{equation}
for any pair $u \in X(\R)$ verifying
\begin{equation}
\label{eq:ortho-u}
\langle u, S \chi_c \rangle_{L^2(\R)^2} = 0.
\end{equation}
\end{prop}

Coming back to \eqref{eq:pignon}, we can prove

\begin{prop}
\label{prop:monou2}
There exists a positive number $B_*$, depending only on $\gc$, such that
\begin{equation}
\label{eq:moyenplateau}
\begin{split}
\frac{d}{dt} \Big( \big\langle M_{c^*(t)} u^*(\cdot, t), u^*(\cdot, t) \big\rangle_{L^2(\R)^2} \Big) \geq & \frac{1}{B_*} \int_\R \big[ (\partial_x u_1^*)^2 + (u_1^*)^2 + (u_2^*)^2 \big] (x, t) e^{- 2 |x|} \, dx\\
& - B_* \big\| \eps^*(., t) \big\|_{X(\R)}^\frac{1}{2} \big\| u^*(\cdot, t) \big\|_{X(\R)}^2,
\end{split}
\end{equation}
for any $t \in \R$.
\end{prop}

Using Propositions \ref{prop:monou1} and \ref{prop:monou2}, we claim 

\begin{cor}
\label{cor:bomono}
Set
$$N(t) := \frac{1}{2} \begin{pmatrix} 0 & x \\ x & 0 \end{pmatrix} + A_* B_* e^{2 R_*} M_{c^*(t)}.$$
There exists a positive constant $\boA_\gc$ such that we have 
\begin{equation}
\label{eq:grandplateau}
\frac{d}{dt} \Big( \langle N(t) u^*(\cdot, t), u^*(\cdot, t) \rangle_{L^2(\R)^2} \Big) \geq \boA_\gc \big\| u^*(\cdot, t) \big\|_{X(\R)}^2,
\end{equation}
for any $t \in \R$. Since
\begin{equation}
\label{eq:petitpignon}
\int_{- \infty}^{+ \infty} \big\| u^*(\cdot, t) \big\|_{X(\R)}^2 \, dt < + \infty,
\end{equation}
there exists a sequence $(t_k^*)_{k \in \N}$ such that
\begin{equation}
\label{eq:derailleur}
\lim_{k \to + \infty} \big\| u^*(\cdot, t_k^*) \big\|_{X(\R)}^2 = 0.
\end{equation}
\end{cor}

In view of \eqref{eq:orthobis}, \eqref{def:u*} and the bound for $\boH _{c^*}$ in \eqref{inv:Hc}, we have 
\begin{equation}
\label{eps:u}
\big\| \eps^*(\cdot, t) \|_{X(\R)} \leq A_\gc \big\| u^*(\cdot, t) \big\|_{X(\R)},
\end{equation}
Hence, we can apply \eqref{eq:derailleur} and \eqref{eps:u} in order to obtain
\begin{equation}
\label{eq:derailleur2}
\lim_{k \to + \infty} \big\| \eps^*(\cdot, t_k^*) \big\|_{X(\R)}^2 = 0.
\end{equation}
By \eqref{eq:derailleur2} and the orbital stability in Theorem \ref{thm:orbistab},
this yields

\begin{cor}
\label{cor:cayest}
We have
$$\eps_0^* \equiv 0.$$
\end{cor}
At this stage we obtain \eqref{conv-stab-asymp} for some subsequence. We should extend this result for any sequence. The proof is exactly the same as the one done by Béthuel, Gravejat and Smets in \cite{BetGrSm2} (see Subsection 1.3.4 in \cite{BetGrSm2} for the details).
\subsection{Proof of Theorem \ref{thm:stabasympt_m}}
We choose a positive number $\delta_\gc$ such that $\| (v_0, w_0) -Q_\gc \|_{X(\R)} \leq \beta_\gc$, whenever $d_\boE(m^0, u_\gc) \leq \delta_\gc$. We next apply Theorem \ref{thm:stabasympt} to the solution $(v, w) \in \boC^0(\R, \boN \boV(\R))$ to \eqref{HLL} corresponding to the solution $m$ to \eqref{LL}. This yields the existence of a speed $\gc^*$ and a position function $b$ such that the convergences in Theorem \ref{thm:stabasympt} hold. In particular, since the weak convergence for $m_3$ is satisfied by Theorem \ref{thm:stabasympt}, it is sufficient to show the existence of a phase function $\theta$ such that $\exp(i\theta(t)) \partial_x\check{m}(\cdot+b(t),t)$ is weakly convergent to $\partial_x \check{u}_{\gc^*}$ in $L^2(\R)$ as $t \to \infty$. The locally uniform convergence of $\exp(i\theta(t)) \check{m}(\cdot+b(t),t)$ towards $\check{u}_{\gc^*}$ then follows from the Sobolev embedding theorem. We begin by constructing this phase function.

We fix a non-zero function $\chi \in \boC_c^\infty(\R, [0, 1])$ such that $\chi$ is even. Using the explicit formula of $\check{u}_{\gc^*}$, we have
\begin{equation}
\label{int_check_u_gc_chi_c}
\int_\R \check{u}_{\gc^*}(x) \chi(x) \, dx = 2 \gc^* \int_\R \frac{\chi(x)}{\cosh\big(\sqrt{1- (\gc^*)^2} x \big)} \, dx \neq 0.
\end{equation}
Decreasing the value of $\beta_\gc$ if needed, we deduce from the orbital stability in \cite{DeLGr} that
\begin{equation}
\label{est:int_m^check_chi}
\bigg|\int_\R \check{m}(x + b(t), t) \chi(x) \, dx \bigg| \geq |\gc^*| \int_\R \frac{\chi(x)}{\cosh\big(\sqrt{1- (\gc^*)^2} x \big)} \, dx \neq  0,
\end{equation}
for any $t \in \R$.

Let $\Upsilon : \R ^2 \longrightarrow \R$ be the $\boC ^1$ function defined by
$$\Upsilon(t,\theta):= \Im\Big( e^{-i\theta} \int_\R \check{m}(x + b(t), t) \chi(x) \, dx \Big).$$
From \eqref{est:int_m^check_chi} we can find a number $\theta_0$ such that $\Upsilon(0,\theta_0)=0$ and $\partial_\theta \Upsilon(0,\theta_0) > 0$. Then, using the implicit function theorem, there exists a $\boC^1$ function $\theta : \R \to \R$ such that $\Upsilon(t,\theta(t))=0.$ In addition, using \eqref{est:int_m^check_chi} another time, we can fix the choice of $\theta$ so that there exists a positive constant $A_{\gc^*}$ such that
\begin{equation}
\label{partial-theta-Upsilon}
\partial_\theta \Upsilon(t,\theta(t)) = \Re\Big( e^{-i\theta(t)} \int_\R \check{m}(x + b(t), t) \chi(x) \, dx \Big) \geq A_{\gc^*} >0.
\end{equation}
This implies, differentiating the identity $\Upsilon(t,\theta(t))=0$ with respect to $t$, that
\begin{equation}
\label{eq:theta'}
|\theta'(t)| = \Big| \frac{\partial_t \Upsilon(t,\theta(t))}{\partial_\theta\Upsilon(t,\theta(t))}\Big| \leq \frac{1}{A_{\gc^*}} \big|\partial_t \Upsilon(t,\theta(t)) \big|,
\end{equation}
for all $t \in \R$. Now, we differentiate the function $\Upsilon$ with respect to $t$, and we use the equation of $\check{m}$ to obtain 
\begin{equation}
\label{derv_t-Upsilon}
\begin{split}
\partial_t \Upsilon(t,\theta(t)) =& \Im\Big( e^{-i\theta} \int_\R \chi(x) \big( \partial_x \check{m}(x + b(t), t) b'(t) -i m_3(x + b(t), t)   \partial_{xx}\check{m}(x + b(t), t) \\
                     & + i \check{m}(x + b(t), t)   \partial_{xx}m_3(x + b(t), t) -i m_3(x + b(t), t)  \check{m}(x + b(t), t) \big) \, dx \Big).
\end{split}
\end{equation}
Since $b \in \boC_b^1(\R, \R)$, and since both $\partial_x \check{m}$ and $\partial_t \check{m}$ belong to $\boC_b^0(\R, H^{- 1}(\R))$, it follows that the derivative $\theta '$ is bounded on $\R$.

We denote by $\varphi$ the phase function defined by
$$\varphi(x+b(t),t):= \varphi(b(t),t) + \int_0^x w(y+b(t),t) \, dy,$$
with $\varphi(b(t),t) \in [0,2\pi],$ which is associated to the function $\check{m}(x+b(t),t)$ for any $(x,t) \in \R^2$  in the way that
$$\check{m}(x+b(t),t)= \big(1-m^2_3(x+b(t),t)\big)^\frac{1}{2} \exp\big(i\varphi(x+b(t),t)\big).$$
It is sufficient to prove that
\begin{equation}
\label{conv:theta0}
\exp\Big( i \big(\varphi(b(t),t)-\theta(t)\big)\Big) \longrightarrow 1,
\end{equation}
as $t \to \infty$ to obtain 
$$\exp\Big( i \big(\varphi(\cdot+b(t),t) - \theta(t)\big)\Big) \longrightarrow \exp \big(i\varphi_{\gc^*}(\cdot)\big):= \exp\Big( i\int_0^\cdot w_{\gc^*}(y) \, dy\Big) \quad {\rm in} \quad L^\infty_{loc}(\R),$$
as $t \to \infty$. This implies, using Theorem \ref{thm:stabasympt} once again, and the Sobolev embedding theorem, that
\begin{equation}
\label{eq:conv:m-check}
\begin{array}{lcll}
e^{- i \theta(t)} \partial_x \check{m}(\cdot + b(t), t) & \rightharpoonup & \partial_x \check{u}_{\gc^*} & {\rm in} \ L^2(\R),\\
e^{-i \theta(t)} \check{m}(\cdot + b(t), t) & \to & \check{u}_{\gc^*} & {\rm in} \ L_{\rm loc}^\infty(\R),
\end{array}
\end{equation}
as $t\to  \infty$.
Now, let us prove \eqref{conv:theta0}. We have
\begin{equation*}
\begin{split}
& e^{-i\theta(t)} \int_\R \check{m}(x + b(t), t) \chi(x) \, dx \\
& = \exp\big(i[\varphi(b(t),t)-\theta(t)]\big) \int_\R \big(1-m^2_3(x+b(t),t)\big)^\frac{1}{2} \exp\Big(i\int_0^x w(y+b(t),t) \, dy\Big)  \chi(x) \, dx.
\end{split}
\end{equation*}
We use the fact that $\Upsilon(t,\theta(t))=0$ to obtain
\begin{equation*}
\begin{split}
& \cos\big(\varphi(b(t),t)-\theta(t)\big) \Im\Big( \int_\R \Big(1-m^2_3(x+b(t),t)\big)^\frac{1}{2} \exp\big(i\int_0^x w(y+b(t),t) \, dy\Big)  \chi(x) \, dx \Big)\\
& + \sin\big(\varphi(b(t),t)-\theta(t)\big) \Re\Big( \int_\R \big(1-m^2_3(x+b(t),t)\big)^\frac{1}{2} \exp\Big(i\int_0^x w(y+b(t),t) \, dy\Big)  \chi(x) \, dx \Big)\\
&= 0.
\end{split}
\end{equation*}
On the other hand, by \eqref{partial-theta-Upsilon}, we have
\begin{equation*}
\begin{split}
& \cos\big(\varphi(b(t),t)-\theta(t)\big) \Re\Big( \int_\R \Big(1-m^2_3(x+b(t),t)\big)^\frac{1}{2} \exp\big(i\int_0^x w(y+b(t),t) \, dy\Big)  \chi(x) \, dx \Big)\\
& - \sin\big(\varphi(b(t),t)-\theta(t)\big) \Im\Big( \int_\R \big(1-m^2_3(x+b(t),t)\big)^\frac{1}{2} \exp\Big(i\int_0^x w(y+b(t),t) \, dy\Big)  \chi(x) \, dx \Big)\\
& > 0.
\end{split}
\end{equation*}
We derive from Theorem \ref{thm:stabasympt} and \eqref{int_check_u_gc_chi_c} that
$$ \Im\Big( \int_\R \big(1-m^2_3(x+b(t),t)\big)^\frac{1}{2} \exp\Big(i\int_0^x w(y+b(t),t) \, dy\Big)  \chi(x) \, dx \Big) \to \Im\Big( \int_\R \check{u}_{\gc^*} (x)  \chi(x) \, dx \Big) =0,$$
and
$$\Re\Big( \int_\R \big(1-m^2_3(x+b(t),t)\big)^\frac{1}{2} \exp\Big(i\int_0^x w(y+b(t),t) \, dy\Big)  \chi(x) \, dx \Big) \to \Re\Big( \int_\R \check{u}_{\gc^*} (x)  \chi(x) \, dx \Big) > 0.$$
This is enough to derive \eqref{conv:theta0}.

 Finally, we claim that $\theta'(t) \longrightarrow 0 $ as $t \to \infty$. Indeed, we can introduce \eqref{eq:conv:m-check} into \eqref{derv_t-Upsilon}, and we then obtain, using the equation satisfied by $\check{u}_{\gc^*},$ that
 $$\partial_t \Upsilon(t,\theta(t)) \longrightarrow 0,$$
as $t \to \infty$. By \eqref{eq:theta'}, this yields $\theta'(t) \longrightarrow 0 $ as $t \to \infty$, which finishes the proof of Theorem \ref{thm:stabasympt_m}.

\qed

\numberwithin{equation}{section}
\section{Proof of the orbital stability}
First, we recall the orbital stability theorem, which was established in \cite{DeLGr} (see Corollary 2, Propositions 2 and 4 in \cite{DeLGr}).
\begin{thm}
\label{thm:orbistab1}
Let $c \in (- 1, 1) \setminus \{ 0 \}$ and $(v_0,w_0) \in X(\R)$ satisfying \eqref{cond:alpha}. There exist a unique global solution $(v,w) \in \boC^0(\R, \boN\boV(\R))$ to \eqref{HLL} with initial datum $(v_0,w_0)$, and two maps $c_1 \in \boC^1(\R, (- 1, 1) \setminus \{ 0 \})$ and $a_1 \in \boC^1(\R, \R)$ such that the function $\eps_1$, defined by \eqref{def:eps}, satisfies the orthogonality conditions
\begin{equation}
\label{eq:ortho1}
\langle \eps_1(\cdot, t), \partial_x Q_{c_1(t)} \rangle_{L^2(\R)^2} = P'(Q_{c_1(t)})(\eps_1(\cdot, t)) = 0,
\end{equation}
for any $t \in \R$. Moreover, $\eps_1(\cdot,t)$, $c_1(t)$ and $a_1(t)$ satisfy \eqref{eq:max-v}, \eqref{eq:modul0} and \eqref{eq:modul1} for any $t\in \R$.
\end{thm}

With Theorem \ref{thm:orbistab1} at hand, we can provide the proof of Theorem \ref{thm:orbistab}.
\begin{proof}
We consider the following map
\begin{equation*}
\Xi((v,w), \sigma, \gb) := \Big( \langle \partial_x Q_{\sigma, \gb}, \eps \rangle_{L^2 \times L^2},  \langle \chi_{\sigma, \gb}, \eps \rangle_{L^2 \times L^2} \Big),
\end{equation*}
where we have set $\eps = (v,w) - Q_{\sigma, \gb},$ and $\chi_{\sigma, \gb}=\chi_{\sigma}(\cdot-\gb)$ (we recall that $\chi_{\sigma}$ is the eigenfunction associated to the unique negative eigenvalue $-\tilde{\lambda}_\sigma $ of the operator $\boH_{\sigma}$). The map $\Xi$ is well-defined for, and depends smoothly on, $(v,w) \in H^1(\R) \times L^2(\R)$, $\sigma \in (-1, 1)\setminus  \{0\}$, and $\gb \in \R$. 

We fix $t\in\R$. In order to simplify the notation, we substitute $(c_1(t),a_1(t))$ by $(c_1,a_1)$. We check that
$$\Xi(Q_{c_1, a_1}, c_1, a_1) = 0,$$
and we compute
$$\left\{ \begin{array}{l} \partial_{\sigma} \Xi_1(Q_{c_1, a_1}, c_1, a_1) = 0,\\
\partial_{\sigma} \Xi_2(Q_{c_1,a_1}, c_1, a_1) = -\langle \chi_{c_1, a_1},\partial_\sigma Q_{c_1,a_1}  \rangle_{L^2 \times L^2}. \end{array} \right.$$
Let $c \in (-1,1) \setminus \{0\}$ and suppose by contradiction that 
$$\langle \chi_c, \partial_c Q_c \rangle_{L^2 \times L^2} = 0.$$
Using the fact that $ \boH_c\big(\partial_c Q_c \big) = P' (Q_c),$ it comes
$$0=\langle \chi_c, \partial_c Q_c \rangle_{L^2 \times L^2} = - \frac{1}{\tilde{\lambda}_c} \langle \chi_c, \boH_c\big(\partial_c Q_c \big) \rangle_{L^2 \times L^2} = - \frac{1}{\tilde{\lambda}_c} \langle \chi_c, P' (Q_c) \rangle_{L^2 \times L^2}.$$ Since $\boH_c$ is self-adjoint, we also have
$$ \langle \chi_c, \partial_x Q_c \rangle_{L^2 \times L^2} = 0.$$ By Proposition 1 in \cite{DeLGr}, we infer that
$$0 > - \tilde{\lambda}_c \| \chi_c \|_{L^2 \times L^2}^2  = \langle \chi_c, \boH_c (\chi_c) \rangle_{L^2 \times L^2} \geq \Lambda_c \| \chi_c \|_{L^2 \times L^2}^2 >0,$$
which provides the contradiction and shows that
\begin{equation}
\label{orth:dcqc-chi_c}
\langle \chi_c, \partial_c Q_c \rangle_{L^2 \times L^2} \neq 0,
\end{equation}
for all $c \in (-1,1) \setminus \{0\}.$ In addition, we have 
$$\left\{ \begin{array}{l} \partial_{b} \Xi_1(Q_{c_1, a_1}, c_1, a_1) = \big\| \partial_x Q_{c_1} \big\|_{L^2}^2 = 2 (1 - c_1^2)^\frac{1}{2} > 0,\\
\partial_{b} \Xi_2(Q_{c_1, a_1}, c_1, a_1) = 0. \end{array} \right.$$
Therefore, the matrix 
$$d_{\sigma, b} \Xi(Q_{c_1, a_1}, c_1, a_1) = \begin{pmatrix}  0 & \langle \chi_{c_1, a_1},\partial_\sigma Q_{c_1,a_1}  \rangle_{L^2 \times L^2} \\ 2 (1 - c_1^2)^\frac{1}{2} & 0 \end{pmatrix}$$ 
is an isomorphism from $\R^{2}$ to $\R^{2}$.

Then, we can apply the version of the implicit function theorem in \cite{BetGrSm1} in order to find a neighbourhood $\boV$ of $ Q_{c_1,a_1}$, a neighbourhood $\boU$ of $(c_1,a_1)$, and a map $\gamma_{c_1,a_1} : \boU \rightarrow \boV$ such that 
$$\Xi((v,w),\sigma,\gb) = 0 \Leftrightarrow (c(v,w),a(v,w)) := (\sigma,\gb) = \gamma_{c,a}(v,w) \quad  \forall (v , w) \in \boV, \ \ \forall (\sigma,\gb) \in \boU.$$ 
In addition, there exists a positive constant $\Lambda$, depending only on $c_1$ such that
\begin{equation}
\label{est:ca}
\|\eps(t)\|_{X} + |c(t)-c_1(t)| + |a(t)-a_1(t)| \leq \Lambda \|\eps_1(t)\|_{X} \leq \Lambda_{c_1} A_\gc \alpha_0,
\end{equation}
where $c(t):=c(v(t),w(t))$, $a(t):=a(v(t),w(t))$ and $\eps(t) := (v(t),w(t)) - Q_{c(t),a(t)},$ for any fixed $t\in \R$. Using the fact that $(v(t),w(t))$ stays into a neighbourhood of $Q_{c_1(t),a_1(t)}$ for all $t\in \R$ by Theorem \ref{thm:orbistab1}, and also the fact that $c_1$ satisfies \eqref{eq:modul0}, we are led to the following lemma.
\begin{lem}
\label{lem:cond-orth-t-2}
Under the assumptions of Theorem \ref{thm:orbistab1}, there exists a unique pair of functions $(a,c) \in \boC^0 \big( \R , \R^2 \big)$ such that
$$\eps(t) := (v(t),w(t))- Q_{c(t),a(t)},$$
verifies the two following orthogonality conditions
\begin{equation}
\label{cond:orth:tr}
\big\langle \eps(t), \partial_x Q_{c(t),a(t)}\big\rangle_{L^2 \times L^2} = \langle \chi_{c(t), a(t)}, \eps(t) \rangle_{L^2 \times L^2} = 0.
\end{equation}
Moreover, we have \eqref{eq:modul0}.
\end{lem}
This completes the proof of orbital stability.
Now, let us prove the continuous differentiability of the functions $a$ and $c$, as well as the inequality
\begin{equation}
\label{guardi}
\big| c'(t) \big| + \big| a'(t) - c(t) \big| \leq A_\gc \big\| \eps(\cdot, t) \big\|_{X(\R)},
\end{equation}
for all $t \in \R$. The $\boC^1$ nature of $a$ and $c$ can be derived from a standard density argument as in \cite{DeLGr}. Concerning \eqref{guardi}, we can write the equation verified by $\eps$, namely
\begin{equation}
\label{eq:eps1}
\partial_t \eps_v = \Big( \big( a'(t) - c(t) \big) \partial_x v_{c,a} - c'(t) \partial_c v_{c,a} \Big) + \partial_x \bigg( \big( (v_{c,a} + \eps_v)^2 - 1 \big) (v_{c,a} + \eps_w) - \big( v_{c,a}^2 - 1 \big) w_{c,a} \bigg),
\end{equation}
and
\begin{equation}
\label{eq:eps2}
\begin{split}
 \partial_t \eps_w & = \big( a'(t) - c(t) \big) \partial_x w_{c,a} - c'(t) \partial_c w_{c,a} \\ 
                   & + \partial_{x} \bigg( \frac{\partial_{xx} v_{c,a} + \partial_{xx} \eps_v}{1 - (v_{c,a} + \eps_v)^2} + \big( v_{c,a} + \eps_v \big) \frac{\big( \partial_x v_{c,a} +\partial_x \eps_v \big)^2}{\big( 1 - (v_{c,a} + \eps_v)^2 \big)^2} \\
                   & - \frac{\partial_{xx} v_{c,a}}{1 - v_{c,a}^2} - v_{c,a} \frac{\big( \partial_x v_{c,a} \big)^2}{\big( 1 - v_{c,a}^2 \big)^2} \bigg) \\
                   & + \partial_x \bigg( ( v_{c,a} + \eps_v) \big( (w_{c,a} + \eps_w)^2 - 1 \big) - v_{c,a} \big(w_{c,a}^2 - 1 \big) \bigg).
\end{split}
\end{equation}
We differentiate with respect to time the orthogonality conditions in \eqref{eq:ortho} and we invoke equations \eqref{eq:eps1} and \eqref{eq:eps2} to write the identity
\begin{equation}
\label{grece}
M \begin{pmatrix} c' \\ a' - c \end{pmatrix} = \begin{pmatrix} Y \\ Z \end{pmatrix}.
\end{equation}
Here, $M$ refers to the matrix of size $2$ given by
\begin{align*}
& M_{1,1} = \langle \partial_c Q_c, \chi_{c} \rangle_{L^2 \times L^2} + \langle \partial_c \chi_{c,a}, \eps \rangle_{L^2 \times L^2},\\
& M_{1, 2} = \langle \chi_c , \partial_x Q_c \rangle_{L^2 \times L^2} - \langle \partial_x \chi_{c,a}, \eps \rangle_{L^2 \times L^2},\\
& M_{2, 1} = - \langle \partial_x Q_c, \partial_c Q_c \rangle_{L^2 \times L^2} + \langle \partial_c \partial_x Q_{c,a}, \eps \rangle_{L^2 \times L^2},\\
& M_{2, 2} = \big\| \partial_x Q_c \big\|^2 _{L^2 \times L^2} - \langle \partial_{xx} Q_{c,a}, \eps \rangle_{L^2 \times L^2}.
\end{align*}
The vectors $Y$ and $Z$ are defined by
\begin{align*}
Y = & \Big\langle \partial_x w_{c,a}, \big( (v_{c,a} + \eps_v)^2 - 1 \big) (w_{c,a} + \eps_w) - \big( v_{c,a}^2 - 1 \big) w_{c,a} \Big\rangle_{L^2}\\
& + \Big\langle \partial_x v_{c,a}, \big( (w_{c,a} + \eps_w)^2 - 1 \big) (v_{c,a} + \eps_v) - \big( w_{c,a}^2 - 1 \big) v_{c,a} \Big\rangle_{L^2}\\
& - \Big\langle \partial_{xx} v_{c,a}, \frac{\partial_{xx} v_{c,a} + \partial_{xx} \eps_v}{1 - (v_{c,a} + \eps_v)^2} - \frac{\partial_{xx} v_{c,a}}{1 - v_{c,a}^2} \Big\rangle_{L^2} + c \big\langle \partial_x \chi_{c,a}, \eps \big\rangle_{L^2 \times L^2},
\end{align*}
and
\begin{align*}
Z = & \Big\langle \partial_{xx} v_{c,a}, \big( (v_{c,a} + \eps_v)^2 - 1 \big) (w_{c,a} + \eps_w) - \big( v_{c,a}^2 - 1 \big) w_{c,a} \Big\rangle_{L^2}\\
& + \Big\langle \partial_{xx} w_{c,a}, \big( (w_{c,a} + \eps_w)^2 - 1 \big) (v_{c,a} + \eps_v) - \big( w^2_{c,a} - 1 \big) v_{c,a} \Big\rangle_{L^2}\\
& - \Big\langle \partial_{xxx} w_{c,a}, \frac{\partial_{xx} v_{c,a} + \partial_{xx} \eps_v}{1 - (v_{c,a} + \eps_v)^2} - \frac{\partial_{xx} v_{c,a}}{1 - v_{c,a}^2} \Big\rangle_{L^2} + c \langle \partial_{xx} Q_{c,a}, \eps \rangle_{L^2 \times L^2}.
\end{align*}

We next decompose the matrix $M$ as $M = D + H$, where $D$ is the diagonal matrix of size $2$ with diagonal coefficients
$$D_{1, 1} = \langle \partial_c Q_c, \chi_{c} \rangle_{L^2 \times L^2} \neq 0,$$
by \eqref{orth:dcqc-chi_c}, and
$$D_{2, 2} = \| \partial_x Q_{c(t)} \|_{L^2}^2 = 2 (1 - c(t)^2)^\frac{1}{2},$$
so that $D$ is invertible. Concerning the matrix $H$, we check that
$$\langle P'(Q_c), \partial_x Q_c \rangle_{L^2 \times L^2} = \langle \partial_x Q_c, \partial_c Q_c \rangle_{L^2 \times L^2} = 0.$$
Then,
$$H = \begin{pmatrix}  \langle \partial_c \chi_{c,a}, \eps \rangle_{L^2 \times L^2} & - \langle \partial_x \chi_{c,a}, \eps \rangle_{L^2 \times L^2} \\ \langle \partial_c \partial_x Q_{c,a}, \eps \rangle_{L^2 \times L^2} & - \langle \partial_{xx} Q_{c,a}, \eps \rangle_{L^2 \times L^2} \end{pmatrix}.$$
It follows from the exponential decay of $Q_{c,a}$ and its derivatives that 
$$|H| \leq A_\gc \|\eps\|_{L^2 \times L^2 }.$$
We can make a further choice of the positive number $\alpha_c$, such that the operator norm of the matrix $D^{- 1} H$ is less than $1/2$. In this case, the matrix $M$ is invertible and the operator norm of its inverse is uniformly bounded with respect to $t$. Coming back to \eqref{grece}, we are led to the estimate
\begin{equation}
\label{eubee}
 \big| c'(t) \big| + \big| a'(t) - c(t) \big| \leq A_\gc \bigg( \big| Y(t) \big| + \big| Z(t) \big| \bigg).
\end{equation}
It remains to estimate the quantities $Y$ and $Z$. We write
\begin{align*}
& \Big| \Big\langle \partial_x w_{c,a}, \big( (v_{c,a} + \eps_v)^2 - 1 \big) (w_{c,a} + \eps_w) - \big( v_{c,a}^2 - 1 \big) w_{c,a} \Big\rangle_{L^2} \Big| \\
& = \Big| \Big\langle \partial_x w_{c,a}, (\eps^2_v + 2 v_{c,a} \eps_v ) w_{c,a} + \eps_w \big((\eps_v + v_{c,a})^2 -1\big) \Big\rangle_{L^2} \Big| \leq A_\gc \|\eps\|_{L^2 \times L^2} .
\end{align*}
Arguing in the same way for the other terms in $Y$ and $Z$, we obtain
$$|Y| + |Z| = \boO \big( \| \eps \|_{L^2 \times L^2} \big),$$
which is enough to deduce \eqref{guardi} from \eqref{eubee}.\\
To achieve the proof, we show \eqref{eq:max-v}. Using the Sobolev embedding theorem of $H^1(\R)$ into $\boC^0(\R)$, we can write
$$\max_{x \in \R} v(x, t) \leq \big\| v_{c(t)} \big\|_{L^\infty(\R)} + \big\|v(\cdot, t) - v_{c(t),a(t)} \big\|_{L^\infty(\R)} \leq \big\| v_{c(t)} \big\|_{L^\infty(\R)} + \|\eps(t)\|_{X(\R)}.$$
By \eqref{form:vc}, $\big\| v_{c} \big\|_{L^\infty(\R)}<1$, so that by \eqref{eq:modul0} there exists a small positive number $\gamma_c$ such that \\ $\big\| v_{c(t)} \big\|_{L^\infty(\R)} \leq 1 - \gamma_c .$ We obtain
$$\max_{x \in \R} v(x, t) \leq 1 - \gamma_c + \|\eps(t)\|_{X(\R)} \leq 1 - \gamma_c + \alpha_c.$$
For $\alpha_c$ small enough, estimate \eqref{eq:max-v} follows, with $\sigma_c := -\alpha_c + \gamma_c$.
\end{proof}

\section{Proofs of localization and smoothness of the limit profile}

\subsection{Proof of Proposition \ref{prop:mono}}
The proof relies on the conservation law for the density of momentum $vw$. Let $R$ and $t$ be two real numbers, and recall that
$$I_R(t) \equiv I_R^{(v, w)}(t) := \frac{1}{2} \int_\R \big[ v w \big](x + a(t), t) \Phi(x - R) \, dx,$$
where $\Phi$ is the function defined on $\R$ by 
$$\Phi(x) := \frac{1}{2} \Big( 1 + \th \big( \nu_\gc x \big) \Big),$$
with $\nu_\gc := \sqrt{1 - \gc^2}/8$.
First, we deduce from the conservation law for $vw$  (see Lemma 3.1 in \cite{DeLGr} for more details) the identity
\begin{equation}
\label{consint}
\begin{split}
\frac{d}{dt} \big[ I_{R + \sigma t}(t) \Big] = & - (a'(t) + \sigma) \int_\R \big[  v w \big](x + a(t), t) \Phi'(x - R - \sigma t) \, dx\\
& + \int_\R \Big[ v^2 + w^2 -3 v^2 w^2 + \frac{3-v^2}{(1-v^2)^2} (\partial_x v)^2 \Big](x + a(t), t) \Phi'(x - R - \sigma t) \, dx\\
& + \int_\R \big[ \ln(1 - v^2) \big](x + a(t), t) \Phi'''(x - R - \sigma t) \, dx.
\end{split}
\end{equation}

Our goal is to provide a lower bound for the integrand in the right-hand side of \eqref{consint}.

 Notice that the function $\Phi$ satisfies the inequality
\begin{equation}
\label{eq:dior0}
|\Phi'''| \leq 4 \nu_\gc^2 \Phi'. 
\end{equation}
In view of the bound \eqref{eq:lvmh1} on $a'(t)$ and the definition of $\sigma_\gc$, we obtain that
\begin{equation}
\label{eq:dior3}
\big| a'(t) + \sigma \big|^2 \leq \frac{9+7\gc^2}{8}.
\end{equation}
Hence, we deduce
\begin{equation}
\label{consint1}
\begin{split}
\frac{d}{dt} \big[ I_{R + \sigma t}(t) \big] \geq &\int_\R \Big[ 4 \nu _{\gc} ^2 \ln(1 - v^2) + v^2 + w^2 -3 v^2 w^2 \\
    & + (\partial_x v)^2- \sqrt{\frac{9 + 7\gc ^2}{8}} \big| vw \big| \Big](x + a(t), t) \Phi'(x - R - \sigma t) \, dx := J_1 + J_2.
\end{split}
\end{equation}
At this step, we decompose the real line into two domains, $[- R_0, R_0]$ and its complement, where $R_0$ is to be defined below and we denote $J_1$ and $J_2$ the value of the integral in the right-hand side of \eqref{consint1} on each region. On $\R \setminus [- R_0, R_0]$, we bound the integrand pointwise from below by a positive quadratic form in $(v, w)$. Exponentially small error terms arise from integration on $[- R_0, R_0]$.

For $|x| \geq R_0$, using Theorem \ref{thm:orbistab}, the Sobolev embedding theorem, and choosing $\alpha_0$ small enough and $R_0$ large enough, we obtain
\begin{equation}
\label{eq:dior1}
\big| v(x + a(t), t) \big| \leq |\eps_v(x,t)| + |v_{c(t)}(x)| \leq A_\gc (\alpha_0 + \exp(-\sqrt{1-\gc ^2}R_0)) \leq \frac{1}{12},
\end{equation}
for any $t \in \R$. Using the fact that $\ln(1 - s) \geq -2s$ for all $s \in [0,\frac{1}{2}]$ and introducing \eqref{eq:dior1} in \eqref{consint1}, we obtain
\begin{equation}
\label{consint2}
\begin{split}
J_1 \geq \frac{1-\gc ^2}{8} \int_{|x| \geq R_0} \Big[ v^2 + w^2 + (\partial_x v)^2 \Big](x + a(t), t) \Phi'(x - R - \sigma t) \, dx.
\end{split}
\end{equation}
We next consider the case $x \in [- R_0, R_0]$. In that region, we have
$$ |x-R-\sigma t| \geq - R_0 + |R+\sigma t|.$$
Hence,
\begin{equation}
\label{eq:dior4}
\Phi'(x - R - \sigma t) \leq 2 \nu_\gc e^{ 2\nu_\gc R_0} e^{ -2\nu_\gc |R + \sigma t|}.
\end{equation}
Since the function $|\ln|$ is decreasing on $(0,1]$, in view of \eqref{eq:max-v} and \eqref{consint1},
$$ \Big| J_2 \Big| \leq A_c \int_{|x| \leq R_0} \Big[ v^2 + w^2 + (\partial_x v)^2 \Big](x + a(t), t) \Phi'(x - R - \sigma t) \, dx. $$
Then, by \eqref{eq:dior4} and the control on the norm of $(v,w)$ in $X(\R)$ provided by the conservation of the energy, we obtain 
$$ \Big| J_2 \Big| \leq B_c e^{ -2\nu_\gc |R + \sigma t|}.$$
This finishes the proof of \eqref{eq:mono}.
It remains to prove \eqref{eq:monobis}. For that, we distinguish two cases. If $R \geq 0$, we integrate \eqref{eq:mono} from $t = t_0$ to $t = (t_0 + t_1)/2$,  choosing $\sigma = \sigma_\gc$ and $R = R - \sigma_\gc t_0$, and then from $t = (t_0 + t_1)/2$ to $t = t_1$ choosing $\sigma = - \sigma_\gc$ and $R = R + \sigma_\gc t_1$. If $R \leq 0$, we use the same arguments for the reverse choices $\sigma = - \sigma_\gc$ and $\sigma = \sigma_\gc$. This implies \eqref{eq:monobis}, and finishes the proof of Proposition \ref{prop:mono}.
\qed

\subsection{Proof of Proposition \ref{prop:smooth}}
Let $\Psi^*$ and $v^*$ be the solutions of \eqref{eq:Psi}-\eqref{eq:v-Psi} expressed in terms of the hydrodynamical variables $(v^*,w^*)$ as in \eqref{def:Psi}. We split the proof into five steps.
\setcounter{step}{0}
\begin{step}
\label{R0}
There exists a positive number $A_{\gc}$, depending only on $\gc$, such that
\begin{equation}
\label{smoothpsi*}
\int_t^{t + 1} \int_\R \big| \partial_x \Psi^*(x + a^*(t), s) \big|^2 e^{ \nu_\gc |x|} \, dx \, ds \leq A_{ \gc},
\end{equation}
for any $t\in \R$.
\end{step}
By \eqref{eq:lvmh2bis} and \eqref{def:Psi}, 
\begin{equation}
\label{estpsi*}
|\Psi^*| \leq A_\gc \big( |\partial_x v^* | + |w^*| \big).
\end{equation}
In view of Proposition \ref{prop:local} and the fact that $|a^*(t) - a^*(s)|$ is uniformly bounded for $s \in [t - 1, t + 2]$ by \eqref{eq:lvmh1bis}, this yields
\begin{equation}
\label{estL2psi*}
\int_{t-1}^{t + 2} \int_\R \big| \Psi^*(x + a^*(t), s) \big|^2 e^{2 \nu_\gc |x|} \, dx \, ds \leq A_\gc.
\end{equation}
We denote
$$ F^* := - \frac{1}{2} (v^*)^2 \Psi^* + \Re \Big( \Psi^* \big( 1 - 2 F(v^*, \overline{\Psi}^*) \big) \Big) \big( 1 - 2 F(v^*, \Psi^*) \big).$$
We recall that $\|v^*\|_{L^\infty(\R \times \R)} <1 -\sigma_\gc$ by \eqref{eq:lvmh2bis}. Using the Cauchy-Schwarz inequality, the Sobolev embedding theorem and the control of the norm in $X(\R)$ provided by the conservation of energy, we have $F(v^*, \Psi^*) \in L^\infty ( \R \times \R ).$ Hence, 
\begin{equation}
\label{estF*}
 |F^*| \leq A_\gc |\Psi^*|,
\end{equation}
where $A_\gc$ is a positive number depending only on $\gc$.
Then, by \eqref{estL2psi*},
\begin{equation}
\label{smoothF*}
\int_{t-1}^{t + 2} \int_\R \big| F^*(x + a^*(t), s) \big|^2 e^{2 \nu_\gc |x|} \, dx \, ds \leq A_\gc,
\end{equation}
for any $t\in \R$. 
Next, by Proposition \ref{prop:cont-Cauchy}, we have
\begin{equation}
\label{estimL4Linf}
\big\| \Psi^*\big\|_{L^4([t-1, t+2], L^\infty)} \leq A_\gc.
\end{equation}
Indeed, we fix $t \in \R$ and we denote $(\Psi^0_{1},v^0_1) := (\Psi^*(\cdot + a^*(t-1), t-1),v^*(\cdot + a^*(t-1), t-1))$ and $(\Psi_{1}(s),v_1(s)) := (\Psi^*(\cdot + a^*(t-1), t-1+s),v^*(\cdot + a^*(t-1), t-1+s))$ the corresponding solution to \eqref{eq:Psi}-\eqref{eq:v-Psi}. Denote also $(\Psi^0_{2},v^0_2) := (\Psi_{c^*(t-1)},v_{c^*(t-1)})$ and $(\Psi_{2}(s),v_2(s)) := (\Psi_{c^*(t-1)}(x-c^*(t-1)s),v_{c^*(t-1)}(x-c^*(t-1)s))$ the corresponding solution to \eqref{eq:Psi}-\eqref{eq:v-Psi}, where $\Psi_{c^*(t)}$ is the solution to \eqref{eq:Psi} associated to the soliton $Q_{c^*(t)}$. We have, by \eqref{mali}, 
$$\big\|\Psi_{1}(s) - \Psi_{2}(s)\big\|_{L^4([0, \tau_\gc], L^\infty)} \leq A \Big( \big\| v^0_1 - v^0_2 \big\|_{L^2} + \big\| \Psi^0_{1} - \Psi^0_{2} \big\|_{L^2} \big).$$
Using \eqref{eq:modul0bis}, we obtain
$$\big\|\Psi_{1}(s) - \Psi_{2}(s)\big\|_{L^4([0, \tau_\gc], L^\infty)} \leq A_\gc,$$
where $\tau_\gc = \tau_\gc (\| v^0_1 \|_{L^2},\| v^0_2 \|_{L^2},\| \Psi^0_1 \|_{L^2},\| \Psi^0_2 \|_{L^2})$ depend only on $\gc$. Since $[0,3] \subseteq \underset{0\leq k \leq 3/\tau_\gc}{\bigcup} [k\tau_\gc ,(k+1) \tau_\gc],$ we can infer \eqref{estimL4Linf} inductively .

In addition, by \eqref{estpsi*}, we have
\begin{equation}
\label{estmPsi*L_infty_L_2}
\big\| \Psi^*(\cdot + a^*(t), \cdot) \big\|_{L^\infty([t-1,t+2],L^2)} \leq A_\gc.
\end{equation}
Hence, applying the Cauchy-Schwarz inequality to the integral with respect to the time variable, \eqref{estL2psi*}, \eqref{estimL4Linf} and \eqref{estmPsi*L_infty_L_2},
\begin{equation}
\label{estL4psi*}
\begin{split}
&\int_{t-1}^{t + 2} \int_\R \big| \Psi^*(x + a^*(t), s) \big|^4 e^{ \nu_\gc |x|} \, dx \, ds \\
 & \leq \int_{t-1}^{t + 2} \int_\R \big| \Psi^*(x + a^*(t), s) \big|^2 e^{ \nu_\gc |x|} \, dx \| \Psi^*(s)\|^2_{L^\infty(\R)} \, ds\\
      & \leq \big\| \Psi^*(\cdot + a^*(t), \cdot) e^{\frac{\nu_\gc}{2}  |\cdot|} \big\|_{L^4([t-1,t+2],L^2(\R))}^2 \, \big\| \Psi^*(\cdot + a^*(t), \cdot) \big\|_{L^4([t-1,t+2],L^\infty(\R))}^2 \\
      & \leq \big\| \Psi^*(\cdot + a^*(t), \cdot) e^{ \nu_\gc |\cdot|} \big\|_{L^2([t-1,t+2],L^2(\R))} \,  \big\| \Psi^*(\cdot + a^*(t), \cdot) \big\|_{L^\infty([t-1,t+2],L^2(\R))} \,\\
      & \ \ \ \ \ \big\| \Psi^*(\cdot + a^*(t), \cdot) \big\|_{L^4([t-1,t+2],L^\infty(\R))}^2 \\
      & \leq A_\gc.
\end{split}
\end{equation}
In order to use Proposition \ref{prop:smoothing} on $\Psi^*$, it is sufficient to verify
\begin{equation}
\label{estLinfL2psi*}
\sup_{s \in [t-1,t + 2]} \int_\R \big| \Psi^*(x + a^*(t), s) \big|^2 e^{2 \nu_\gc |x|} \, dx \, ds \leq A_\gc.
\end{equation}
 
Indeed, using \eqref{estLinfL2psi*} and \eqref{estimL4Linf}, we can write
\begin{equation}
\label{estL6psi*}
\begin{split}
&\int_{t-1}^{t + 2} \int_\R \big| \Psi^*(x + a^*(t), s) \big|^6 e^{ 2\nu_\gc |x|} \, dx \, ds \\
&\leq \big\| \Psi^*(\cdot + a^*(t), \cdot) e^{ \nu_\gc |\cdot|} \big\|^2_{L^\infty([t-1,t+2],L^2(\R))} \, \big\| \Psi^*(\cdot + a^*(t), \cdot) \big\|_{L^4([t-1,t+2],L^\infty(\R))}^4 \\
&\leq A_\gc,
\end{split}
\end{equation}
which proves that $\Psi^*$ satisfies the assumptions of Proposition \ref{prop:smoothing}. Then, we apply Proposition \ref{prop:smoothing} with $u := \Psi^*(\cdot + a^*(t), \cdot + (t + 1/2))$, $T := 1/2$, $F := |u|^2u + F^*(\cdot,t+1/2)$ and successively $\lambda := \pm  \nu_\gc$ and we use \eqref{estL2psi*} and \eqref{smoothF*} to obtain \eqref{smoothpsi*}.

Now let us prove \eqref{estLinfL2psi*}. First, we recall the next lemma stated by Kenig, Ponce and Vega \cite{KenPoVe9}.
\begin{lem}
\label{Lemma 2.1}

Let $a \in [-2,-1]$ and $b \in [2,3]$. Assume that $u\in C^0([a,b]:L^2(\R))$ is a solution of the inhomogeneous Schrödinger equation
\begin{equation}
\label{2.2}
i\partial_t u + \partial_{xx} u = H,
\end{equation}
with $H\in L^1([a,b]:L^2(e^{\beta x}dx))$, for some $\beta\in\R$, and 
\begin{equation}
\label{2.3}
u_a \equiv u(\cdot,a),\,u_b\equiv u(\cdot,b)\in L^2(e^{\beta x}dx).
\end{equation}
There exist a positive number $K$ such that
\begin{equation}
\label{2.4}
\begin{aligned}
&\sup_{a\leq t\leq b}\|u(\cdot,t)\|_{L^2(e^{\beta x}dx)}\\
\\
&\leq K (\|u_a\|_{L^2(e^{\beta x}dx)} + \|u_b\|_{L^2(e^{\beta x}dx)}
+\|H\|_{L^1([a,b],L^2(e^{\beta x}dx))}).
\end{aligned}
\end{equation}

\end{lem}
In order to apply the lemma, we need to verify the existence of numbers $a$ and $b$ such that \eqref{2.3} holds for $u := \Psi^*(\cdot + a^*(t), \cdot + t )$ and such that $H:=|u|^2u + F^*(\cdot,\cdot+t) \in L^1([a,b],L^2(e^{\beta x}dx))$, for $\beta = \pm \nu_\gc$ respectively and any $t\in \R$.
Our first claim is a consequence of \eqref{estL2psi*} and the Markov inequality. Indeed, there exist $s_0 \in [-2,-1]$ and $s_1 \in [2,3]$ such that
$$\int_\R \big| \Psi^*(x + a^*(t), s_j+t) \big|^2 e^{2 \nu_\gc |x|} \, dx \leq A_\gc \quad {\rm for} \, j=0,1.$$
For the second claim, by \eqref{smoothF*} and the Cauchy-Schwarz estimate, it is sufficient to show that $|u|^2u \in L^1([-2,3],L^2(e^{\nu_\gc |x|}dx))$. To prove this we use the Cauchy-Schwarz inequality for the time variable, \eqref{estL2psi*} and \eqref{estimL4Linf},
\begin{align*}
& \int_{-2}^{ 3} \Big( \int_\R \big| \Psi^*(x + a^*(t), s+t) \big|^6 e^{ 2\nu_\gc |x|} \, dx\Big)^\frac{1}{2} \, ds\\
& \leq \big\| \Psi^*(\cdot + a^*(t), \cdot+t) e^{ \nu_\gc |\cdot|} \big\|_{L^2([-2,3],L^2)} \, \big\| \Psi^*(\cdot + a^*(t), \cdot+t) \big\|_{L^4([-2,3],L^\infty)}^2\\
& \leq A_\gc.
\end{align*}
Now, we are allowed to apply Lemma \ref{Lemma 2.1} with $a=s_0$ and $b=s_1$ to deduce \eqref{estLinfL2psi*}.
This finishes the proof of this first step.

In the next step, we prove that \eqref{smoothpsi*} remains true for all the derivatives of $\Psi^*$ and $v^*$.
\begin{step}
\label{R1}
Let $k \geq 1$. There exists a positive number $A_{k, \gc}$, depending only on $k$ and $\gc$, such that
\begin{equation}
\label{eq:Rk2}
\int_t^{t + 1} \int_\R \big| \partial_x^k \Psi^*(x + a^*(t), s) \big|^2 e^{ \nu_\gc |x|} \, dx \, ds \leq A_{k, \gc},
\end{equation}
and
\begin{equation}
\label{estmL2tL2v*exp}
\int_t^{t + 1} \int_\R \big| \partial_x^k v^*(x + a^*(t), s) \big|^2 e^{ \nu_\gc |x|} \, dx \leq A_{k, \gc},
\end{equation}
for any $t\in \R$.
\end{step}

The proof of Step \ref{R1} is by induction on $k \geq 1$. We are going to differentiate \eqref{eq:Psi} $k$ times with respect to the space variable and write the resulting equation as
\begin{equation}
\label{eq:psikder}
i \partial_t \big( \partial_x^k \Psi^* \big) + \partial_{xx} \big( \partial_x^k \Psi^* \big) = R_k(v^*,\Psi^*).
\end{equation}
where $R_k(v^*,\Psi^*) = \partial^k_x \big( |\Psi ^* |^2 \Psi ^*\big) + \partial^k_x F^*.$
 We are going to prove by induction that \eqref{eq:Rk2}, \eqref{estmL2tL2v*exp} and
\begin{equation}
\label{eq:Rk}
\int_t^{t + 1} \int_\R \big| R_k(v^*,\Psi^*)(x + a^*(t), s) \big|^2 e^{ \nu_\gc |x|} \, dx \, ds \leq A_{k, \gc},
\end{equation}
hold simultaneously for any $t \in \R$. Notice that \eqref{eq:Rk2} implies that $\partial_x^k \Psi^* \in L^2_{\rm loc}(\R, L^2(\R))$, while \eqref{eq:Rk} implies that $R_k(v^*,\Psi^*) \in L^2_{\rm loc}(\R, L^2(\R))$. Therefore, if \eqref{eq:Rk2}, \eqref{estmL2tL2v*exp} and \eqref{eq:Rk} are established for some $k \geq 1$, then applying Proposition \ref{prop:smoothing} to $\partial_x^k \Psi^*$ can be justified by a standard approximation procedure.

For $k=1$, \eqref{eq:Rk2} is exactly \eqref{smoothpsi*}. \eqref{estmL2tL2v*exp} holds from Proposition \ref{prop:local} and the fact that $|a^*(t) - a^*(s)|$ is uniformly bounded for $s \in [t - 1, t + 2]$. Next, we write
\begin{align*}
R_1(v^*,\Psi^*) &= - v^* \partial_x v^* \Psi^* - \frac{1}{2} (v^*)^2 \partial_x \Psi^* + \Re \Big( \partial_x \Psi^* \big( 1 - 2 F(v^*, \overline{\Psi}^*) \big) \Big) \big( 1 - 2 F(v^*, \Psi^*) \big)\\
          & \ \ - 2 v^* |\Psi^*|^2 \big( 1 - 2 F(v^*, \Psi^*) \big) -2 v^* \Psi^* \Re \Big( \Psi^* \big( 1 - 2 F(v^*, \overline{\Psi}^*) \big) -2 \partial_x \big(\Psi^*|\Psi^*|^2\big) \Big).
\end{align*}
We will show that 
\begin{equation}
\label{est:Psi*_L_inf_L_inf}
\Psi^* \in L^\infty([t-1,t+2], L^\infty(\R)),
\end{equation}
in order to control the derivative of the cubic non-linearity by $|\partial_x \Psi^*|$ and then we will use the fact that $F(v^*,\Psi^*) \in L^\infty(\R \times \R)$, $\|v^*\|_{L^\infty(\R \times \R)}<1$ and the second equation in \eqref{eq:v-Psi} to get
\begin{equation}
\label{estimR1}
R_1(v^*,\Psi^*) \leq K \big( |\partial_x \Psi^*| + |\partial_x v^*| |\Psi^*| + |\Psi^*|^2 \big).
\end{equation}
Let us prove \eqref{est:Psi*_L_inf_L_inf}. We define the function $H$ on $\R$ by
$$H(s) := \frac{1}{2} \int_\R ( |\partial_x \Psi^*(x, s)|^2 - |\Psi^*(x, s)|^4) dx.$$
We differentiate it with respect to $s$,  integrate by part and use \eqref{eq:Psi} to obtain
\begin{equation}
\label{est:H}
\begin{split}
H'(s) & =  - \Re \Big( \int_\R \partial_s \Psi^*(x, s) \big[ \overline{\partial_{xx} \Psi^* + 2 \Psi^*|\Psi^*|^2} \big](x, s) dx\Big)\\
      & = \Re \Big( \int_\R \partial_s \Psi^*(x, s) F^*(x, s) dx\Big)\\
      & \leq \|\partial_s \Psi^*(s)\|_{H^{-1}(\R)} \|F^*(s)\|_{H^{1}(\R)}.
\end{split}
\end{equation}
We have
$$|\partial_x F^*| \leq K \big( |\partial_x \Psi^*| + |\partial_x v^*| |\Psi^*| + |\Psi^*|^2 \big),$$
using the fact that $F(v^*,\Psi^*) \in L^\infty(\R \times \R)$, $\|v^*\|_{L^\infty(\R \times \R)}<1$ and the second equation in \eqref{eq:v-Psi}.

Hence, by \eqref{smoothpsi*}, \eqref{estL2psi*}, \eqref{estL4psi*}, and the fact that $ | \partial_x v^* | \leq | \Psi^*| $ on $\R \times \R $, we obtain
\begin{equation}
\label{est:dxF^*}
\|\partial_x F^*\|_{L^2([t-1,t+2],L^2(\R))} \leq A_\gc .
\end{equation}
On the other hand, we infer from \eqref{eq:Psi}, \eqref{smoothpsi*}, \eqref{smoothF*} and the fact that \\ $\Psi^* \in L^4([t-1,t+2],L^\infty(\R))\bigcap L^8([t-1,t+2],L^4(\R)),$ that
\begin{equation}
\label{ds:Psi}
\|\partial_s \Psi^*\|_{L^2([t-1,t+2],H^{-1}(\R))} \leq A_\gc.
\end{equation}
Next, we integrate \eqref{est:H} between $t-1$ and $t+2$ and we apply the Cauchy-Schwarz inequality to obtain $H \in W^{1,1}([t-1,t+2]) $ for all $t\in \R$ using \eqref{est:dxF^*} and \eqref{ds:Psi}. Notice that all these computations can be justified by a standard approximation procedure. This yields, by the Sobolev embedding theorem, that $H \in L^\infty([t-1,t+2]).$ We conclude that the derivative $\partial_x \Psi^* \in  L^\infty([t-1,t+2],L^2(\R)).$ Indeed, we can use the Gagliardo-Nirenberg inequality and the fact that $\Psi^*$ is uniformly bounded in $L^2(\R)$ by a positive number to write
\begin{equation*}
\begin{split}
H(s) & \geq \frac{1}{2} \int_\R  |\partial_x \Psi^*(x, s)|^2 dx - A \| \Psi^*(s)\|^3_{L^2(\R)} \|\partial_x \Psi^*(\cdot)\|_{L^2(\R)} \\
     & \geq \frac{1}{2} \int_\R |\partial_x \Psi^*(x, s)|^2 dx - A K^3 \|\partial_x \Psi^*(\cdot)\|_{L^2(\R)}.
\end{split}
\end{equation*}
The function $x\mapsto \frac{1}{2} x^2 -A M^3 x$ diverges to $+\infty$ when $x$ goes to $+\infty$. Since $H$ is bounded, we infer that $\|\partial_x \Psi^*(\cdot)\|_{L^2(\R)}$ is uniformly bounded on $[t-1,t+2]$ for all $t \in \R$. This finishes the proof of \eqref{est:Psi*_L_inf_L_inf} by the Sobolev embedding theorem.
Then, by \eqref{estimR1}, \eqref{eq:Rk} for $k=1$ is a consequence of \eqref{smoothpsi*}, \eqref{estL4psi*} and the fact that $ | \partial_x v^* | \leq | \Psi^*| $ on $\R \times \R .$

Assume now that \eqref{eq:Rk2}, \eqref{estmL2tL2v*exp} and \eqref{eq:Rk} are satisfied for any integer $1 \leq k \leq k_0$ and any $t \in \R$. Let us prove these three estimates for $k=k_0+1$.
We apply Proposition \ref{prop:smoothing} with $u := \partial_x^{k_0} \Psi^*(\cdot + a^*(t), \cdot + (t + 1/2))$, $T := 1/2$ and successively $\lambda := \pm  \nu_\gc$. In view of \eqref{eq:Rk2}, \eqref{eq:psikder}, \eqref{eq:Rk}, and the fact that $|a^*(t) - a^*(s)|$ is uniformly bounded for $s \in [t - 1, t + 2]$, this yields
\begin{equation}
\label{eq:Rk3}
\int_t^{t + 1} \int_\R |\partial_x^{k_0 + 1} \Psi^*(x + a^*(t), s)|^2 e^{ \nu_\gc |x|} \, dx \, ds \leq A_\gc ,
\end{equation}
so that \eqref{eq:Rk2} is satisfied for $k = k_0 + 1$.

Let $k \in \{1,...,k_0\}$. We use the induction hypothesis and \eqref{eq:Rk3} to infer that
$$ \partial^{k-1}_x \Psi^* \in L^2([t,t+1], H^{2}(\R)).$$
Also, we have 
$$ \partial^{k-1}_x \Psi^* \in H^1([t,t+1], L^{2}(\R))$$
using \eqref{eq:psikder} and \eqref{eq:Rk}. This yields, by interpolation,
$$\partial^{k-1}_x \Psi^* \in H^{\frac{2}{3}}([t,t+1], H^{\frac{2}{3}}(\R)).$$
Hence, using the Sobolev embedding theorem, we obtain
\begin{equation}
\label{estm:derpsik}
\partial^{k-1}_x \Psi^* \in L^\infty([t,t+1], L^\infty(\R)) \quad {\rm for \ all} \quad  t \in \R.
\end{equation}
On the other hand, since
$ |\partial_x v^*| \leq | \Psi^*|,$ we have, by \eqref{est:Psi*_L_inf_L_inf}, $\partial_x v^* \in L^\infty([t,t+1], L^\infty(\R)).$ For $k \in \{2,...,k_0 \}$, we differentiate the second equation in \eqref{eq:v-Psi} $k$ times and we use \eqref{estm:derpsik} to obtain
\begin{equation}
\label{estm:derv*}
|\partial_x^k v^*| \leq K \Big( \sum_{j=1}^{k-1} |\partial_x^j \Psi^*| + \sum_{j=0}^{k-2} |\partial_x^j v^*|\Big),
\end{equation}
where $K$ is a positive constant. We infer from \eqref{estm:derpsik} by induction that
\begin{equation}
\label{estm:derv*2}
\partial^{k}_x v^* \in L^\infty([t,t+1], L^\infty(\R)) \quad {\rm for \ all} \quad  t \in \R,
\end{equation}
for all $k \in \{2,...,k_0\}.$ 
Then, we just compute explicitly $R_{k_0+1}(v^*,\Psi^*)$ and we use \eqref{estm:derpsik} and \eqref{estm:derv*2} to obtain
$$\big| R_{k_0+1}(v^*,\Psi^*) \big| \leq A_{k_0+1,\gc,K} \Big( \sum_{j=0}^{k_0+1}  |\partial_x^{j} \Psi^*| + \sum_{j=1}^{k_0} |\partial_x^{j} v^*|\Big).$$
Hence, by \eqref{eq:Rk2} for all $k \leq k_0$, \eqref{estmL2tL2v*exp} and \eqref{eq:Rk3}, we obtain \eqref{eq:Rk} for $k=k_0+1$.
Finally, we introduce \eqref{eq:Rk2} for all $k\leq k_0+1$ and \eqref{estmL2tL2v*exp} for all $k\leq k_0$ into \eqref{estm:derv*} to deduce \eqref{estmL2tL2v*exp} for $k=k_0+1$. This finishes the proof of this step.

In order to finish the proof of Proposition \ref{prop:smooth}, we now turn these $L_{\rm loc}^2$ in time estimates into $L^\infty$ in time estimates, and then into uniform estimates.

\begin{step}
\label{R2}
Let $k \geq 0$. There exists a positive number $A_{k, \gc}$, depending only on $k$ and $\gc$, such that
\begin{equation}
\label{estmLinftL2Psi*exp}
\int_\R \big| \partial_x^k \Psi^*(x + a^*(t), t) \big|^2 e^{ \nu_\gc |x|} \, dx \leq A_{k, \gc},
\end{equation}
for any $t\in \R$. In particular, we have
\begin{equation}
\label{estmLinftLinftPsi*exp}
\big\| \partial_x^k \Psi^*(\cdot + a^*(t), t) e^{ \frac{\nu_\gc}{2}  |\cdot|} \big\|_{L^\infty(\R)} \leq A_{k, \gc},
\end{equation}
for any $t \in \R$, and a further positive constant $A_{k, \gc}$, depending only on $k$ and $\gc$.
\end{step}

Here, we use the Sobolev embedding theorem in time and \eqref{eq:psikder} for the proof. By the Sobolev embedding theorem, we have
\begin{align*}
\big\| \partial_x^k \Psi^*(\cdot + a^*(t), t) e^{\frac{\nu_\gc}{2} |\cdot|} \big\|_{L^2(\R)}^2 \leq & K \Big( \big\| \partial_s \big( \partial_x^k \Psi^*(\cdot + a^*(t), s) e^{\frac{\nu_\gc}{2} |\cdot|} \big) \big\|_{L^2([t - 1, t + 1], L^2(\R))}^2\\
& + \big\| \partial_x^k \Psi^*(\cdot + a^*(t), s) e^{\frac{\nu_\gc}{2} |\cdot|} \big\|_{L^2([t - 1, t + 1],L^2(\R))}^2 \Big),
\end{align*}
while, by \eqref{eq:psikder},
\begin{align*}
\big\| \partial_s \big( \partial_x^k \Psi^*(\cdot + a^*(t), s) e^{\frac{\nu_\gc}{2} |\cdot|} \big) \big\|_{L^2([t - 1, t + 1], L^2(\R))}^2 \leq & 2 \Big( \big\| \partial_x^{k + 2} \Psi^*(\cdot + a^*(t), s) e^{\frac{\nu_\gc}{2} |\cdot|} \big\|_{L^2([t - 1, t + 1], L^2(\R))}^2\\
& + \big\| R_k(\Psi^*)(\cdot + a^*(t), s) e^{\frac{\nu_\gc}{2} |\cdot|} \big\|_{L^2([t - 1, t + 1], L^2(\R))}^2 \Big),
\end{align*}
so that we finally deduce \eqref{estmLinftL2Psi*exp} from \eqref{eq:Rk2} and \eqref{eq:Rk}. Estimate \eqref{estmLinftLinftPsi*exp} follows from applying the Sobolev embedding theorem to \eqref{estmLinftL2Psi*exp}.

Similarly, the function $v^*$ satisfies

\begin{step}
\label{R3}
Let $k \in \N$. There exists a positive number $A_{k, \gc}$, depending only on $k$ and $\gc$, such that
\begin{equation}
\label{estmLinftL2v*exp}
\int_\R \big( \partial_x^k v^*(x + a^*(t), t) \big)^2 e^{ \nu_\gc |x|} \, dx \leq A_{k, \gc},
\end{equation}
and
\begin{equation}
\label{estmLinftLinftv*exp}
\big\| \partial_x^k v^*(\cdot + a^*(t), t) e^{ \frac{\nu_\gc}{2} |\cdot|} \big\|_{L^\infty(\R)} \leq A_{k, \gc},
\end{equation}
for any $t \in \R$.
\end{step}
The proof is similar to the proof of Step \ref{R2} using the first equation in \eqref{eq:v-Psi} instead of \eqref{eq:Psi}. We use the Sobolev embedding theorem to write
\begin{align*}
\big\| \partial_x^k v^*(\cdot + a^*(t), t) e^{\nu_\gc |\cdot|} \big\|_{L^2(\R)}^2 \leq & K \Big( \big\| \partial_s \big(\partial_x^k v^*(\cdot + a^*(t), s) e^{\nu_\gc |\cdot|} \big) \big\|_{L^2([t - 1, t + 1], L^2(\R))}^2\\
& + \big\| \partial_x^k v^*(\cdot + a^*(t), s) e^{\nu_\gc |\cdot|} \big\|_{L^2([t - 1, t + 1],L^2(\R))}^2 \Big).
\end{align*}
By the first equation in \eqref{eq:v-Psi}, \eqref{eq:Rk2}, \eqref{eq:psikder} and \eqref{estm:derv*2}, we have

$$\big\| \partial_s \big(\partial_x^k v^*(\cdot + a^*(t), s) e^{\nu_\gc |\cdot|} \big) \big\|_{L^2([t - 1, t + 1], L^2(\R))}^2 \leq A_\gc.$$
This leads to \eqref{estmLinftL2v*exp}. The uniform bound in \eqref{estmLinftLinftv*exp} is then a consequence of the Sobolev embedding theorem.

Finally, we provide the estimates for the function $w^*$.

\begin{step}
\label{R4}
Let $k \in \N$. There exists a positive number $A_{k, \gc}$, depending only on $k$ and $\gc$, such that
\begin{equation}
\label{estmL2w*exp}
\int_\R \big| \partial_x^k w^*(x + a^*(t), t) \big| ^2 e^{ \nu_\gc |x|} \, dx \leq A_{k, \gc},
\end{equation}
and
\begin{equation}
\label{estmLinftyw*exp}
\big\| \partial_x^k w^*(\cdot + a^*(t), t) e^{ \frac{\nu_\gc}{2} |\cdot|} \big\|_{L^\infty(\R)} \leq A_{k, \gc},
\end{equation}
for any $t \in \R$.
\end{step}
The proof relies on the last two steps. First, we write
$$v^* \Psi^* = - \frac{1}{2} \partial_x \Big( (1 - (v^*)^2)^\frac{1}{2} \exp i \theta^* \Big).$$
Since $(1 - v^*(x, t)^2)^{1/2} \exp i \theta^*(x, t) \to 1$ as $x \to - \infty$ for any $t \in \R$, we obtain the formula
\begin{equation}
\label{eq:theta}
2 F(v^*, \Psi^*) = 1 - (1 - (v^*)^2)^\frac{1}{2} \exp i \theta^*.
\end{equation}
Hence, using \eqref{def:Psi}, we have
\begin{equation}
\label{eq:w^*}
 w^* = 2 \Im \Big( \frac{\Psi^* \big( 1- 2 F(v^*,\Psi ^*)\big)}{1-(v^*)^2} \Big).
\end{equation}
Combining \eqref{eq:max-v} and \eqref{eq:theta}, we recall that
\begin{equation}
\label{estm:v*}
\frac{|1- 2 F(v^*,\Psi ^*)|}{1- (v^*)^2} \leq A_\gc .
\end{equation}
Hence, we obtain
$$|w^*| \leq A_\gc |\Psi^*|.$$
Then, \eqref{estmL2w*exp} and \eqref{estmLinftyw*exp} follow from \eqref{estmLinftL2Psi*exp} and \eqref{estmLinftLinftPsi*exp} for $k=0$. 
For $k \geq 1$, we differentiate \eqref{eq:w^*} $k$ times with respect to the space variable, and using \eqref{estmLinftLinftPsi*exp}, \eqref{estmLinftLinftv*exp} and \eqref{estm:v*}, we are led to
$$|\partial_x^{k} w^* | \leq A_{k,\gc} \Big( \sum_{j=0}^{k} |\partial_x^{j} \Psi^*| + \sum_{j=1}^{k-1} |\partial_x^{j} v^*| \Big).$$
We finish the proof of this step using Steps \ref{R2} and \ref{R3}.
This achieves the proof of Proposition \ref{prop:smooth}.
\qed
\section{Proof of the Liouville theorem}

\subsection{Proof of Proposition \ref{prop:monou1}}
\label{sec:prop:monou1}

First, by \eqref{eq:smooth} and the explicit formula for $v_c$ and $w_c$ in \eqref{form:vc}, there exists a positive number $A_{k, \gc}$ such that
\begin{equation}
\label{smooth:psi*}
\int_\R \Big( \big( \partial_x^k \eps_v^*(x, t) \big)^2 + \big( \partial_x^k \eps_w^*(x, t) \big)^2 \Big) e^{ \nu_\gc |x|} \, dx \leq A_{k, \gc},
\end{equation}
for any $k \in \N$ and any $t \in \R$. In view of the formulae of $\boH_c$ in \eqref{def:boHc} and for $u^*$ in \eqref{def:u*}, a similar estimate holds for $u^*$, for a further choice of the constant $A_{k, \gc}$. As a consequence, we are allowed to differentiate with respect to the time variable the quantity
$$\boI^*(t) := \int_\R x u^*_1(x, t) u^*_2(x, t) \, dx,$$
in the left-hand side of \eqref{eq:petitplateau}. Moreover, we can compute
\begin{equation}
\label{dt_I*}
\begin{split}
\frac{d}{dt} \Big( \boI^* \Big) = & - 2 \int_\R \mu \big\langle \boH_{c^*}(\partial_x u^*), u^* \big\rangle_{\R^2} + \int_\R \mu \big\langle \boH_{c^*} \big( J \boR_{c^*} \eps^* \big), u^* \big\rangle_{\R^2} \\
                                  & - \big( c^* \big)' \int_\R \mu \big\langle \boH_{c^*}(\partial_c Q_{c^*}), u^* \big\rangle_{\R^2}+ \big( c^* \big)' \int_\R \mu \big\langle \partial_c \boH_{c^*}(\eps^*), u^* \big\rangle_{\R^2} \\
                                  & + \big( (a^*)' - c^* \big) \int_\R \mu \big\langle \boH_{c^*}(\partial_x \eps^*), u^* \big\rangle_{\R^2},
\end{split}
\end{equation}
where we have set $\mu(x) = x$ for any $x \in \R$.

At this stage, we split the proof into five steps. The proof of these steps is similar to the proof of Proposition 7 in \cite{BetGrSm2}. We first show
\setcounter{step}{0}
\begin{step}
\label{I1}
There exist two positive numbers $A_1$ and $R_1$, depending only on $\gc$, such that
\begin{equation}
\label{boI_1*}
\boI_1^*(t) := - 2 \int_\R \mu \big\langle \boH_{c^*}(\partial_x u^*), u^* \big\rangle_{\R^2} \geq \frac{1 - \gc^2}{8} \big\| u^*(\cdot, t) \big\|_{X(\R)}^2 - A_1 \big\| u^*(\cdot, t) \big\|_{X(B(0, R_1))}^2,
\end{equation}
for any $t \in \R$.
\end{step}

We introduce the explicit formulae of the operator $\boH_{c^*}$ in the definition of $\boI_1^*(t)$ to obtain
\begin{align*}
\boI_1^*(t) = & 2 \int_\R \mu \partial_x \Big( \frac{ \partial_{xx} u^*_1 }{1 - v^2_{c^*}} \Big) u^*_1 - 2 \int_\R \mu \big( 1-(c^*)^2 - (5 + (c^*)^2) v^2_{c^*} + 2 v^4_{c^*} \big) \frac{\partial_x u^*_1}{(1 - v^2_{c^*})^2} u_1^* \\
            & + 2 \int_\R \mu c^* \frac{1+v^2_{c^*}}{1-v^2_{c^*}} (\partial_x u_2^* ) u_1^* - 2 \int_\R \mu (c^*)^2 \frac{(1+v^2_{c^*})^2}{(1-v^2_{c^*})^3} (\partial_x u_1^*) u_1^* \\
            & + 2 \int_\R \mu c^* \frac{1+v^2_{c^*}}{1-v^2_{c^*}} (\partial_x u_1^*) u_2^* - 2 \int_\R \mu \big(1-v^2_{c^*}\big) (\partial_x u_2^*) u_2^*.
\end{align*}

Integrating by parts each term, we obtain
%
$$\boI_1^*(t) = \int_\R \iota_1^*(x, t) \, dx,$$
with
\begin{align*}
\iota_1^* = & \Big( \frac{ 2 }{1 - v^2_{c^*}} + 2 x \frac{ \partial_x v_{c^*} \ v_{c^*}}{1 - v^2_{c^*}} \Big) \big(\partial_x u^*_1 \big)^2 - 2 c^* \Big( \frac{1+v^2_{c^*}}{1-v^2_{c^*}} + \frac{4 x \partial_x v_{c^*} \ v_{c^*}}{\big( 1-v^2_{c^*} \big)^2} \Big) u_2^* u_1^* \\
            & + \big( 1-v^2_{c^*} - 2 x \partial_x v_{c^*} \ v_{c^*} \big) \big( u_2^* \big) ^2 +  \frac{1+2\big( (c^*)^2 -3 \big) v^2_{c^*} + \big( 2 (c^*)^2 -3 \big) v^4_{c^*} - 2 v^6_{c^*}}{\big( 1-v^2_{c^*} \big)^3} \big( u_1^* \big) ^2 \\
            &+ 4 x \partial_x v_{c^*} \ v_{c^*} \frac{\big( (c^*)^2 -3 \big) + \big( 2 (c^*)^2 -3 \big) v^2_{c^*} - 3 v^4_{c^*}}{\big( 1-v^2_{c^*} \big)^4} \big( u_1^* \big) ^2 .
\end{align*}
Let $\delta$ be a small positive number. We next use the exponential decay of the function $v_c$ and its derivatives to guarantee the existence of a radius $R$, depending only on $\gc$ and $\delta$ (in view of the bound on $c^* - \gc$ in \eqref{eq:modul0bis}), such that
$$\iota_1^*(x, t) \geq \big( 2 - \delta\big) \big( \partial_x u_1^* \big)^2(x,t) + \big( \frac{1-\gc ^2}{4} - \delta\big) \big( (u_1^*)^2(x,t) + (u_2^*)^2(x,t) \big),$$
when $|x| \geq R$.

Then, we choose $\delta$ small enough and fix the number $R_1$ according to the value of the corresponding $R$, to obtain
\begin{equation}
\label{int-iota_1*}
\int_{|x| \geq R_1} \iota_1^*(x, t) \, dx \geq \frac{1 - \gc^2}{8} \int_{|x| \geq R_1} \big( (\partial_x u_1^*(x, t))^2 + u_1^*(x, t)^2 + u_2^*(x, t)^2 \big) \, dx.
\end{equation}
On the other hand, it follows from \eqref{form:vc}, and again \eqref{eq:modul0}, that
$$\int_{|x| \leq R_1} \iota_1^*(x, t) \, dx \geq \Big( \frac{1 - \gc^2}{8} - A_1 \Big) \int_{|x| \leq R_1} \big( (\partial_x u_1^*(x, t))^2 + u_1^*(x, t)^2 + u_2^*(x, t)^2 \big) \, dx,$$
for a positive number $A_1$ depending only on $\gc$. Combining with \eqref{int-iota_1*}, we obtain \eqref{boI_1*}.

\begin{step}
\label{I2}
There exist two positive numbers $A_2$ and $R_2$, depending only on $\gc$, such that
\begin{equation}
\label{boI_2*}
\big| \boI_2^*(t) \big| := \bigg| \int_\R \mu \big\langle \boH_{c^*} \big( J \boR_{c^*} \eps^* \big), u^* \big\rangle_{\R^2} \bigg| \leq \frac{1 - \gc^2}{64} \big\| u^*(\cdot, t) \big\|_{X(\R)}^2 + A_2 \big\| u^*(\cdot, t) \big\|_{X(B(0, R_2))}^2,
\end{equation}
for any $t \in \R$.
\end{step}
We refer to the proof of Step 2 in the proof of Proposition 7 in \cite{BetGrSm2} for mare details.

Next, we infer from \eqref{eq:modul1}, \eqref{eps:u}, the explicit formula of $\boH_{c^*}$ in \eqref{def:boHc} and the exponential decay of the function $\partial_c Q_{c^*}$ and its derivatives, that
\begin{step}
\label{I4}
There exist two positive numbers $A_3$ and $R_3$, depending only on $\gc$, such that
\begin{equation}
\label{boI_4*}
\big| \boI_4^*(t) \big| := \bigg| (c^*)' \int_\R \mu \big\langle \boH_{c^*}(\partial_c Q_{c^*}), u^* \big\rangle_{\R^2} \bigg| \leq \frac{1 - \gc^2}{64} \big\| u^*(\cdot, t) \big\|_{X(\R)}^2 + A_3 \big\| u^*(\cdot, t) \big\|_{X(B(0, R_3))}^2,
\end{equation}
for any $t \in \R$.
\end{step}
We decompose the real line into two regions $[-R,R]$ and its complement for any $R >0$. We use the fact that $|x|\leq e^\frac{\nu_\gc |x|}{4}$ for all $|x| \geq R$, to write
\begin{equation*}
\begin{split}
\big| \boI_4^*(t) \big| \leq & R |(c^*)'(t)| \int_{|x| \leq R} \big| \boH_{c^*(t)} (\partial_c Q_{c^*(t)})(x) \big| \big| u^*(x, t) \big| \, dx\\
& + \delta |(c^*)'(t)| \int_{|x| \geq R} \big| \boH_{c^*(t)} (\partial_c Q_{c^*(t)})(x) \big| \big| u^*(x, t) \big| e^\frac{\nu_\gc |x|}{4} \, dx,
\end{split}
\end{equation*}
for any $t \in \R$. We deduce from \eqref{eq:modul1}, the explicit formula of $\boH_{c^*}$ in \eqref{def:boHc} and the exponential decay of the function $\partial_c Q_{c^*}$ and its derivatives that
$$\big| \boI_4^*(t) \big| \leq A_\gc \Big( R \big\| u^*(\cdot, t) \big\|_{X(B(0, R))} + \delta \big\| u^*(\cdot, t) \big\|_{X(\R)} \Big) \big\| \eps^*(\cdot, t) \big\|_{L^2(\R)^2},$$
for any $t \in \R$. Hence, by \eqref{eps:u},
$$\big| \boI_4^*(t) \big| \leq A_\gc \Big( \frac{R^2}{\delta} \big\| u^*(\cdot, t) \big\|_{X(B(0, R))}^2 + 2 \delta \big\| u^*(\cdot, t) \big\|_{X(\R)}^2 \Big).$$
We choose $\delta$ so that $2 A_\gc \delta \leq (1 - \gc^2)/64$, and we denote by $R_4$ the corresponding number $R$, we obtain \eqref{boI_4*}, with $A_4 = A_\gc R_4^2/\delta$.

Similarly, we use \eqref{eq:modul1}, \eqref{eq:modul0bis} and \eqref{eps:u} to obtain
\begin{step}
\label{I3}
There exists two positive numbers $A_4$ and $R_4$, depending only on $\gc$, such that
\begin{equation}
\label{boI_3*}
\big| \boI_3^*(t) \big| := \bigg| (c^*)' \int_\R \mu \big\langle \partial_c \boH_{c^*}(\eps^*), u^* \big\rangle_{\R^2} \bigg| \leq \frac{1 - \gc^2}{64} \big\| u^*(\cdot, t) \big\|_{X(\R)}^2 + A_4 \big\| u^*(\cdot, t) \big\|_{X(B(0, R_4))}^2,
\end{equation}
for any $t \in \R$.
\end{step}

We argue as in Steps \ref{I4} to show

\begin{step}
\label{I5}
There exist two positive numbers $A_5$ and $R_5$, depending only on $\gc$, such that
\begin{equation}
\label{boI_5*}
\begin{split}
\big| \boI_5^*(t) \big| &:= \bigg| \big( (a^*)' - c^* \big) \int_\R \mu \big\langle \boH_{c^*}(\partial_x \eps^*), u^* \big\rangle_{\R^2} \bigg| \\
                        & \leq \frac{1 - \gc^2}{64} \big\| u^*(\cdot, t) \big\|_{X(\R)}^2 + A_5 \big\| u^*(\cdot, t) \big\|_{X(B(0, R_5))}^2,
\end{split}
\end{equation}
for any $t \in \R$.
\end{step}

Finally, combining the estimates in Steps \ref{I1} to \ref{I5} with the identity \eqref{dt_I*}, we obtain
$$\frac{d}{dt} \Big( \boI^*(t) \Big) \geq \frac{1 - \gc^2}{16} \big\| u^*(\cdot, t) \big\|_{X(\R)}^2 - \big( A_1+A_2+ A_3 + A_4+A_5 \big) \big\| u^*(\cdot, t) \big\|_{X(B(0, R_*))}^2,$$
this allow us to conclude the proof of \eqref{eq:petitplateau}
with $R_* = \max \{ R_1, R_2,R_3,R_4,R_5 \}$ and $A_* = A_1 + A_2 + A_3 + A_4 +A_5$.

\subsection{Proof of Lemma \ref{lem:loc-coer}}

When $u \in H^3(\R)\times H^1(\R)$, the function $\partial_x u$ is in the space $H^2(\R)\times L^2(\R)$ which is the domain of $\boH_c$. The scalar product in the right-hand side of \eqref{eq:loc-virial} is well-defined in view of \eqref{def:Mc}. Next, we use the formula for $\boH_c$ in \eqref{def:boHc} to express $G_c(u)$ as
\begin{equation}
\label{int=int0}
\begin{split}
\big\langle & S M_c u, \boH_c(-2 \partial_x u) \big\rangle_{L^2(\R)^2} \\
= & 2 \int_\R \frac{\partial_x v_c}{v_c} \bigg( \frac{ 1 - c^2 - (5 + c^2) v_c^2 + 2 v_c^4 }{(1 - v_c^2)^2} + c^2 \frac{(1 + v_c^2)^2}{(1 - v_c^2)^3} - 2 c^2 \frac{v^2_c (1 + v_c^2)}{(1 - v_c^2)^3} \bigg) u_1 \partial_x u_1\\
& - 2 \int_\R \frac{\partial_x v_c}{v_c} \partial_x \Big( \frac{\partial_{xx} u_1}{1 - v^2_c} \Big) + 2 \int_\R \frac{\partial_x v_c (1 - v^2_c)}{v_c} u_2 \partial_x u_2\\
& + 2 c \int_\R \bigg( 2 \frac{ v_c \partial_x v_c}{1 - v^2_c} u_1 \partial_x u_2 - \frac{\partial_x v_c (1+v_c^2)}{v_c (1 - v^2_c)} \partial_x \big( u_1 u_2 \big) \bigg).
\end{split}
\end{equation}
We recall that $v_c$ solves the equation
\begin{equation}
\label{eq:vc}
\partial_{xx} v_c = (1 - c^2 -2 v_c^2) v_c,
\end{equation}
which leads to
\begin{equation}
\label{part_vc}
(\partial_x v_c)^2 = (1 - c^2 - v_c^2) v_c^2 , \quad {\rm and} \quad \partial_x \Big( \frac{\partial_x v_c}{v_c} \Big) = - v^2_c.
\end{equation}
Then, the third integral in the right-hand side of \eqref{int=int0} can be written as
\begin{equation}
\label{int=int}
2 \int_\R \frac{\partial_x v_c (1 - v^2_c)}{v_c} u_2 \partial_x u_2 = \int_\R \mu_c u_2^2,
\end{equation}
with $\mu_c := 2 (\partial_x v_c)^2 + (1 - v_c^2) v_c^2$. Similarly, the last integral is given by
\begin{equation}
\label{int=int1}
\int_\R \bigg( 2 \frac{ v_c \partial_x v_c}{1 - v^2_c} u_1 \partial_x u_2 - \frac{\partial_x v_c (1+v_c^2)}{v_c (1 - v^2_c)} \partial_x \big( u_1 u_2 \big) \bigg) = - \int_\R \Big( v^2_c u_1 u_2 + 2 \frac{v_c \partial_x v_c}{1 - v_c^2} u_2 \partial_x u_1 \Big).
\end{equation}
Combining \eqref{int=int} and \eqref{int=int1} with \eqref{int=int0}, we obtain the identity
$$\big\langle S M_c u, \boH_c(-2 \partial_x u) \big\rangle_{L^2(\R)^2} = I + \int_\R \mu_c \Big( u_2 - \frac{c v^2_c}{\mu_c} u_1 - \frac{2c v_c \partial_x v_c}{ \mu_c (1 - v^2_c)} \partial_x u_1 \Big)^2,$$
where
\begin{align*}
I = & \int_\R 2 \bigg( \frac{\partial_x v_c}{v_c} \Big( \frac{ 1 - c^2 - (5 + c^2) v_c^2 + 2 v_c^4 }{(1 - v_c^2)^2} + c^2 \frac{1 + v_c^2}{(1 - v_c^2)^2} \Big) -2 c^2 \frac{v^3_c \partial_x v_c}{\mu_c (1 - v^2_c)} \bigg) u_1 \partial_x u_1\\
& - \int_\R \frac{\partial_x v_c}{v_c} u_1 \partial_x \Big( \frac{\partial_{xx} u_1}{1 - v^2_c} \Big) - c^2 \int_\R \frac{v_c^4}{\mu_c} u_1^2 - 4 c^2 \int_\R \frac{(\partial_x v_c)^2 v^2_c}{\mu_c (1 - v^2_c)^2} (\partial_x u_1)^2.
\end{align*}
Using \eqref{eq:vc} and \eqref{part_vc}, we finally deduce that
$$I = \frac{3}{2} \int_\R \frac{v_c^4}{\mu_c} \Big( \partial_x u_1 - \frac{\partial_x v_c}{v_c} u_1 \Big)^2,$$
which finishes the proof of \eqref{eq:loc-virial}.

\subsection{Proof of Proposition \ref{prop:coer-Gc}}

We first rely on \eqref{form:vc} and \eqref{eq:loc-virial} to check that the quadratic form $G_c$ is well-defined and continuous on $X(\R)$. Next, setting 
\begin{equation}
\label{def:v}
v = (v_c u_1, v_c u_2), 
\end{equation}
and using \eqref{eq:vc}, we can express it as
\begin{equation}
\label{G_c=K_c}
G_c(u) = K_c(v) := \int_\R \frac{v^2_c}{\mu_c} \Big( \partial_x v_1 - \frac{2 \partial_x v_c}{ v_c} v_1 \Big)^2 + \int_\R \frac{\mu_c}{v^2_c} \Big( v_2 + \frac{c \lambda_c}{ \mu_c (1 - v^2_c)} v_1 - 2\frac{c v_c \partial_x v_c}{ \mu_c (1 - v^2_c)} \partial_x v_1 \Big)^2,
\end{equation}
where we have set $\lambda_c := -\mu_c + 4 (\partial_{x} v_c)^2$. From \eqref{ker-Gc} and \eqref{def:v} we deduce that 
\begin{equation}
\label{ker:Kc}
\Ker(K_c)= \Span(v_cQ_c).
\end{equation}
Let $w$ be the pair defined in the following way
\begin{equation*}
w = \Big( v_1, v_2 - 2\frac{c v_c \partial_x v_c}{ \mu_c (1 - v^2_c)} \partial_x v_1 \Big).
\end{equation*}
We compute
\begin{equation}
\label{K_c=boT_c-w}
K_c(v) = \big\langle \boT_c(w), w \big\rangle_{L^2(\R)^2},
\end{equation}
with
\begin{equation}
\label{def:boT_c}
\begin{split}
&\boT_c(w) = \\
&\begin{pmatrix} - 3\partial_x \Big( \frac{v^2_c}{ \mu_c} \partial_x w_1 \Big) + \Big( \frac{8 v^4_c (\partial_x v_c)^2 - 2 v^6_c (1-v^2_c)}{ \mu^2_c } + \frac{4 (\partial_x v_c)^2}{ \mu_c } + \frac{c^2 (2c^2-1+v^2_c)^2 v^2_c}{ \mu_c (1 - v^2_c)^2} \Big) w_1 - \frac{c (2c^2-1+v^2_c)}{(1 - v^2_c)} w_2 \\ - \frac{c (2c^2-1+v^2_c)}{(1 - v^2_c)} w_1 + \frac{\mu_c}{v^2_c} w_2 \end{pmatrix}.
\end{split}
\end{equation}
The operator $\boT_c$ in \eqref{def:boT_c} is self-adjoint on $L^2(\R)^2$, with domain $\Dom(\boT_c) = H^2(\R) \times L^2(\R)$. In addition, combining \eqref{G_c=K_c} with \eqref{K_c=boT_c-w} we deduce that $\boT_c$ is non-negative, with a kernel equal to
$$\Ker(\boT_c) = \Span \Big\{ \Big( v^2_c, \frac{2c v^2_c (\partial_x v_c)^2}{\mu_c (1 - v^2_c)} \Big)\Big\}.$$
At this stage, we divide the proof into three steps.
\setcounter{step}{0}
\begin{step}
\label{G1}
Let $c \in (- 1, 1) \setminus \{ 0 \}$. There exists a positive number $\Lambda_1$, depending continuously on $c$, such that
\begin{equation}
\label{postiv:boT_c}
\langle \boT_c(w), w \rangle_{L^2(\R)^2} \geq \Lambda_1 \int_\R \Big( w_1^2 + w_2^2 \Big),
\end{equation}
for any pair $w \in X^1(\R)$ such that
\begin{equation}
\label{cond:ortho:w}
\Big\langle w, \Big( v^2_c, \frac{2c v^2_c (\partial_x v_c)^2}{\mu_c (1 - v^2_c)} \Big) \Big\rangle_{L^2(\R)^2} = 0.
\end{equation}
\end{step}

We claim that the essential spectrum of $\boT_c$ is given by
\begin{equation}
\label{spec:ess:boT}
\sigma_{\rm ess}(\boT_c) = \big[ \tau_c, + \infty \big),
\end{equation}
with
\begin{equation}
\label{def:tau_c}
\tau_c = \tau_{1,c} - \frac{1}{2} \tau_{2,c}^\frac{1}{2} > 0.
\end{equation}
Here, we have set 
$$\tau_{1,c} = \frac{4(1 - c^2)+c^2(2c^2 - 1)^2}{2(3-2 c^2)}+ \frac{3-2 c^2}{2}$$
and
$$\tau_{2,c} = \Big( \frac{4(1 - c^2)+c^2(2c^2 - 1)^2}{3-2 c^2} - (3-2 c^2)\Big)^2 + 4 c^2(2c^2 - 1)^2.$$
In particular, $0$ is an isolated eigenvalue in the spectrum of $\boT_c$. Inequality \eqref{postiv:boT_c} follows with $\Lambda_1$ either equal to $\tau_c$, or to the smallest positive eigenvalue of $\boT_c$. In view of the analytic dependence on $c$ of the operator $\boT_c$, $\Lambda_1$ depends continuously on $c$ .

Now, let us prove \eqref{spec:ess:boT}. We rely on the Weyl criterion. It follows from \eqref{def:mu_c} and \eqref{eq:vc} that
$$\frac{\mu_c(x)}{v^2_c(x)} \to 3 - 2c^2, \quad {\rm and} \quad \frac{(\partial_x v_c)^2(x)}{\mu_c(x)} \to \frac{1-c^2}{3-2c^2},$$
as $x \to \pm \infty$. Coming back to \eqref{def:boT_c}, we introduce the operator $\boT_\infty$ given by
$$\boT_\infty(w) = \begin{pmatrix} - \frac{3}{3 - 2c^2} \partial_{xx} w_1 + \frac{4(1-c^2)+ c^2 (2c^2-1)^2}{3-2c^2} w_1 - c (2c^2-1) w_2 \\ - c (2c^2-1) w_1 +  (3 - 2 c^2) w_2 \end{pmatrix}.$$
By the Weyl criterion, the essential spectrum of $\boT_c$ is equal to the spectrum of $\boT_\infty$.

We next apply again the Weyl criterion to establish that a real number $\lambda$ belongs to the spectrum of $\boT_\infty$ if and only if there exists a complex number $\xi$ such that
$$\lambda^2 - \Big( \frac{3}{3 - 2c^2} |\xi|^2 + \frac{4(1-c^2)+ c^2 (2c^2-1)^2}{3-2c^2} + 3 - 2 c^2 \Big) \lambda + 3 |\xi|^2 + 4 (1-c^2) = 0.$$
This is the case if and only if
\begin{align*}
\lambda = &\frac{4(1 - c^2)+c^2(2c^2 - 1)^2 +3 |\xi|^2}{2(3-2 c^2)} + \frac{3-2 c^2}{2} \\ 
          & \pm \frac{1}{2}\Big( \Big( \frac{4(1 - c^2)+c^2(2c^2 - 1)^2+ 3|\xi|^2}{3-2 c^2} - (3-2 c^2)\Big)^2 + 4 c^2(2c^2 - 1)^2 \Big)^\frac{1}{2}.
\end{align*} 
This leads to
$$\sigma_{\rm ess}(\boT_c) = \sigma(\boT_\infty) = \big[ \tau_c, + \infty \big),$$
with $\tau_c$ as in \eqref{def:tau_c}. This completes the proof of Step \ref{G1}.

\begin{step}
\label{G2}
There exists a positive number $\Lambda_2$, depending continuously on $c$, such that
\begin{equation}
\label{pos:K-c}
K_c(v) \geq \Lambda_2 \int_\R \big( (\partial_x v_1)^2 + v_1^2 + v_2^2\big) ,
\end{equation}
for any pair $v \in X^1(\R)$ such that
\begin{equation}
\label{orth:v:v_c-1:chi_c}
\langle v, v^{-1}_c S \chi_c \rangle_{L^2(\R)^2} = 0.
\end{equation}
\end{step}

We start by improving the estimate in \eqref{postiv:boT_c}. Given a pair $w \in X^1(\R)$, we observe that
$$\bigg| \langle \boT_c(w), w \rangle_{L^2(\R)^2} - 3 \int_\R \frac{v^2_c}{\mu_c} (\partial_x w_1)^2 \bigg| \leq A_c \int_\R (w_1^2 + w_2^2).$$
Here and in the sequel, $A_c$ refers to a positive number, depending continuously on $c$. For $0 < \tau < 1$, we have
$$\langle \boT_c(w), w \rangle_{L^2(\R)^2} \geq \big( 1 - \tau \big) \langle \boT_c(w), w \rangle_{L^2(\R)^2} + 3 \tau \int_\R \frac{v^2_c}{\mu_c} (\partial_x w_1)^2 - A_c \tau \int_\R (w_1^2 + w_2^2).$$
Since $v^2_c/\mu_c \geq 1/(3 - 2c^2)$, this yields
$$\langle \boT_c(w), w \rangle_{L^2(\R)^2} \geq \Big( \big( 1 - \tau \big) \Lambda_1 - A_c \tau \Big) \int_\R (w_1^2 + w_2^2) + \frac{3 \tau}{3 - 2c^2} \int_\R (\partial_x w_1)^2,$$
under condition \eqref{cond:ortho:w}. For $\tau$ small enough, this leads to
\begin{equation}
\label{postiv:boT_c-w,w}
\langle \boT_c(w), w \rangle_{L^2(\R)^2} \geq A_c \int_\R \big( (\partial_x w_1)^2 + w_1^2 + w_2^2 \big),
\end{equation}
when $w$ satisfies condition \eqref{cond:ortho:w}.

Since the pair $w$ depends on the pair $v$, we can write \eqref{postiv:boT_c-w,w} in terms of $v$. By \eqref{K_c=boT_c-w}, $K_c(v)$ is equal to the left-hand side of \eqref{postiv:boT_c-w,w}. We deduce that \eqref{postiv:boT_c-w,w} may be expressed as
$$K_c(v) \geq A_c \int_\R \Big( (\partial_x v_1)^2 + v_1^2 \Big) + A_c \int_\R \Big( v_2 - \frac{2c v_c (\partial_x v_c)}{\mu_c (1 - v^2_c)} \partial_x v_1 \Big)^2.$$
We recall that, given two vectors $a$ and $b$ in a Hilbert space $H$, we have
$$\big\| a - b \big\|_H^2 \geq \tau \big\| a \big\|_H^2 - \frac{\tau}{1 - \tau} \big\| b \big\|_H^2,$$
for any $0 < \tau < 1$. Then, we deduce that
$$K_c(v) \geq A_c \int_\R \Big( (\partial_x v_1)^2 + v_1^2 + \tau v_2^2 \Big) - \frac{\tau A_c}{1 - \tau} \int_\R \Big( \frac{v_c (\partial_x v_c)}{\mu_c (1 - v^2_c)} \partial_x v_1 \Big)^2.$$
We choose $\tau$ small enough so that we can infer from \eqref{form:vc} that
\begin{equation}
\label{postiv:K_c-v}
K_c(v) \geq A_c \int_\R \big( (\partial_x v_1)^2 + v_1^2 + v_2^2 \big),
\end{equation}
when $w$ satisfies condition \eqref{cond:ortho:w}, i.e. when $v$ is orthogonal to the pair
\begin{equation}
\label{def:gv_c}
\gv_c = \Big( v^2_c - \partial_x\Big( \frac{2c v_c^2 (\partial_x v_c)^2}{\mu_c (1 - v^2_c)} \Big), \frac{2c v_c^2 (\partial_x v_c)^2}{\mu_c (1 - v^2_c)} \Big).
\end{equation}

Next, we verify that \eqref{postiv:K_c-v} remains true, decreasing possibly the value of $A_c$, when we replace this orthogonality condition by
\begin{equation}
\label{orth:v,v_cQ_c}
\langle v, v_c Q_c \rangle_{L^2(\R)^2} = 0.
\end{equation}
We remark that
$$\langle \gv_c, v_c Q_c \rangle_{L^2(\R)^2} \neq 0.$$ Indeed, we would deduce from \eqref{postiv:K_c-v} that
$$0 = K_c(v_c Q_c) \geq A_c \int_\R \big( (\partial_x v^2_c)^2 + v_c^4 + (v_c w_c)^2 \big) > 0,$$
which is impossible. In addition, the number $\langle \gv_c, v_c Q_c \rangle_{L^2(\R)^2}$ depends continuously on $c$ in view of \eqref{def:gv_c}. Given a pair $\tilde{v}$ satisfying \eqref{orth:v,v_cQ_c}, we denote by $\lambda$ the real number such that $\gv = \lambda v_c Q_c + \tilde{v}$ is orthogonal to $\gv_c$. Since $v_c Q_c$ belongs to the kernel of $K_c$, we obtain using \eqref{postiv:K_c-v},
\begin{equation}
\label{positiv:K_c:tilde-v}
K_c(\tilde{v}) = K_c(\gv) \geq A_c \int_\R \big( (\partial_x \gv_1)^2 + \gv_1^2 + \gv_2^2 \big).
\end{equation}
On the other hand, since $\tilde{v}$ satisfies \eqref{orth:v,v_cQ_c}, we have
$$\lambda = \frac{\langle \gv, v_c Q_c \rangle_{L^2(\R)^2}}{\| v_c Q_c \|_{L^2(\R)^2}^2}.$$
Using the Cauchy-Schwarz inequality, this yields
$$\lambda^2 \leq A_c \bigg( \int_\R \big( v^4_c + (v_c w_c)^2 \big) \bigg) \bigg( \int_\R \big(\gv_1^2 + \gv_2^2 \big) \bigg),$$
hence, by \eqref{form:vc} and \eqref{positiv:K_c:tilde-v},
$$\lambda^2 \leq A_c K_c(\gv) = A_c K_c(\tilde{v}).$$
Using \eqref{positiv:K_c:tilde-v}, this leads to
\begin{align*}
\int_\R \big( (\partial_x \tilde{v}_1)^2 + \tilde{v}_1^2 + \tilde{v}_2^2 \big)\leq & 2 \bigg( \lambda^2 \int_\R v^2_c \big( (\partial_x v_c)^2 + v_c^2 + w_c^2 \big) + \int_\R \big( (\partial_x \gv_1)^2 + \gv_1^2 + \gv_2^2 \big) \bigg)\\
                                                 & \leq A_c K_c(\tilde{v}),
\end{align*}
We finish the proof of this step applying again the same argument. We write $v = \lambda v_c S Q_c + \tilde{v}$, with $\langle \tilde{v}, v_c Q_c \rangle_{L^2(\R)^2} = 0$. Since $v_c Q_c$ belongs to the kernel of $K_c$, we infer from the same argument that
\begin{equation}
\label{positiv:K_c:v}
K_c(v) = K_c(\tilde{v}) \geq \Lambda_2 \int_\R (\partial_x \tilde{v}_1)^2 + \tilde{v}_1^2 + \tilde{v}_2^2.
\end{equation}
Using the orthogonality condition in \eqref{orth:v:v_c-1:chi_c}, we obtain
$$\lambda = - \frac{\langle \tilde{v}, v^{-1}_c S \chi_c \rangle_{L^2(\R)^2}}{\langle Q_c, S \chi_c \rangle_{L^2(\R)^2}}.$$
By the Cauchy-Schwarz inequality, we are led to
$$\lambda^2 \leq A_c \|v^{-1}_c S \chi_c\|^2_{L^2 \times L^2} \int_\R \big(\tilde{v}_1^2 + \tilde{v}_2^2 \big).$$
Invoking the exponential decay of $\chi_c$ in \eqref{dec:chi-c}, we deduce
$$ \|v^{-1}_c S \chi_c\|^2_{L^2 \times L^2} \leq A_c.$$
As a consequence, we can derive from \eqref{positiv:K_c:v} that
$$\lambda^2 \leq A_c K_c(\tilde{v}) = A_c K_c(v).$$
Combining again with \eqref{positiv:K_c:v}, we are led to
\begin{align*}
\int_\R \big( (\partial_x v_1)^2 + v_1^2 + v_2^2\big)  \leq & 2 \bigg( \lambda^2 \int_\R v^2_c \big( (\partial_x v_c)^2 + v_c^2 + w_c^2 \big) + \int_\R \big( (\partial_x \tilde{v}_1)^2 + \tilde{v}_1^2 + \tilde{v}_2^2\big) \bigg)\\
                                            \leq & A_c K_c(v).
\end{align*}
which completes the proof of Step \ref{G2}.

\begin{step}
\label{G3}
End of the proof.
\end{step}
Since the pair $v$ depends on the pair $u$ as in \eqref{def:v}, we can write \eqref{pos:K-c} in terms of $u$. The left hand side of \eqref{pos:K-c} is equal to $G_c(u)$ by \eqref{G_c=K_c}. Moreover, for the right-hand side, we have
$$\int_\R \big( (\partial_x v_1)^2 + v_1^2 + v_2^2\big) = \int_\R v_c^2 \big( (\partial_x u_1)^2 + (2v_c^2 + c^2)u_1^2 + u_2^2 \big).$$
We deduce that \eqref{pos:K-c} may be written as
\begin{equation}
\label{pos:G-c}
G_c(u) \geq A_c \int_\R v_c^2 \big( (\partial_x u_1)^2 + u_1^2 + u_2^2 \big),
\end{equation} 
when $v_cu$ verifies the orthogonality condition \eqref{orth:v:v_c-1:chi_c}, which means that $u$ verifies the orthogonality condition \eqref{eq:ortho-u}. We recall that
$$v_c(x) \geq A_c e^{- |x|},$$
by \eqref{form:vc}, which is sufficient to obtain \eqref{eq:coer-Gc}. This completes the proof of Proposition \ref{prop:coer-Gc}.

\subsection{Proof of Proposition \ref{prop:monou2}}

First we check that we are allowed to differentiate the quantity
$$\boJ^*(t) := \big\langle M_{c^*(t)} u^*(\cdot, t), u^*(\cdot, t) \big\rangle_{L^2(\R)^2}.$$
Indeed, by \eqref{def:u*}, \eqref{smooth:psi*}, and \eqref{def:boHc}, there exists a positive number $A_{k, \gc}$ such that
\begin{equation}
\label{int-der-u*-exp}
\int_\R \Big( \big( \partial_x^k u_1^*(x, t) \big)^2 + \big( \partial_x^k u_2^*(x, t) \big)^2 \Big) e^{ \nu_\gc |x|} \, dx \leq A_{k, \gc}.
\end{equation}
Next, using \eqref{eq:pouru*} and \eqref{def:Mc}, we obtain
\begin{equation}
\label{dt-boJ}
\begin{split}
\frac{d}{dt} \big( \boJ^* \big) = & 2 \big\langle S M_{c^*} u^*, \boH_{c^*}(J S u^*) \big\rangle_{L^2(\R)^2} + 2 \big\langle S M_{c^*} u^*, \boH_{c^*}(J \boR_{c^*} \eps^*) \big\rangle_{L^2(\R)^2}\\
                                  & + 2 \big( (a^*)' - c^* \big) \big\langle S M_{c^*} u^*, \boH_{c^*}(\partial_x \eps^*) \big\rangle_{L^2(\R)^2} \\ 
                                  & - 2 \big( c^* \big)' \big\langle S M_{c^*} u^*, \boH_{c^*}(\partial_c Q_{c^*}) \big\rangle_{L^2(\R)^2}\\ 
                                  & + \big( c^* \big)' \big\langle \partial_c M_{c^*} u^*, u^* \big\rangle_{L^2(\R)^2} + 2 \big( c^* \big)' \big\langle M_{c^*} u^*, S \partial_c \boH_{c^*}(\eps^*) \big\rangle_{L^2(\R)^2}.
\end{split}
\end{equation}

The proof of \eqref{eq:moyenplateau} is the same as in \cite{BetGrSm2}. We will give only the main ideas of the proof. We will estimate all the terms in the right-hand side of \eqref{dt-boJ} except the fourth term  which vanishes.

For the first one, we infer from Proposition \ref{prop:coer-Gc} the following estimate.
\setcounter{step}{0}
\begin{step}
\label{L1}
There exists a positive number $B_1$, depending only on $\gc$, such that
\begin{equation*}
\boJ_1^*(t) := 2 \big\langle S M_{c^*} u^*, \boH_{c^*}(J S u^*) \big\rangle_{L^2(\R)^2}\\
\geq B_1 \int_\R \big[ (\partial_x u_1^*)^2 + (u_1^*)^2 + (u_2^*)^2 \big](x, t) e^{- 2 |x|} \, dx,
\end{equation*}
for any $t \in \R$.
\end{step}

For the second term, by \eqref{eq:modul0bis}, \eqref{eps:u} and \eqref{smooth:psi*}, we have

\begin{step}
\label{L2}
There exists a positive number $B_2$, depending only on $\gc$, such that
\begin{equation*}
\big| \boJ_2^*(t) \big| := 2 \Big| \big\langle S M_{c^*} u^*, \boH_{c^*}(J \boR_{c^*} \eps^*) \big\rangle_{L^2(\R)^2} \Big| \leq B_2 \big\| \eps^*(\cdot, t) \big\|_{X(\R)}^\frac{1}{2} \big\| u^*(\cdot, t) \big\|_{X(\R)}^2,
\end{equation*}
for any $t \in \R$.
\end{step}
For the third one, we use \eqref{eq:modul0bis} to obtain
\begin{step}
\label{L3}
There exists a positive number $B_3$, depending only on $\gc$, such that
\begin{equation*}
\big| \boJ_3^*(t) \big| := 2 \big| (a^*)' - c^* \big| \, \Big| \big\langle S M_{c^*} u^*, \boH_{c^*}(\partial_x \eps^*) \big\rangle_{L^2(\R)^2} \Big| \leq B_3 \big\| \eps^*(\cdot, t) \big\|_{X(\R)}^\frac{1}{2} \big\| u^*(\cdot, t) \big\|_{X(\R)}^2,
\end{equation*}
for any $t \in \R$.
\end{step}

We now prove the following statement for the fourth term.
\begin{step}
\label{L4}
We have
\begin{equation*}
 \boJ_4^*(t) := 2 (c^*)'  \big\langle S M_{c^*} u^*, \boH_{c^*}(\partial_c Q_{c^*}) \big\rangle_{L^2(\R)^2} = 0,
\end{equation*}
for any $t \in \R$.
\end{step}

Since $\boH_{c^*}(\partial_c Q_{c^*}) = P'(Q_{c^*}) = SQ_{c^*} $ and $M_{c^*}Q_{c^*} = S \partial_x Q_{c^*},$ we have
\begin{align*}
 \big\langle S M_{c^*} u^*, \boH_{c^*}(\partial_c Q_{c^*}) \big\rangle_{L^2(\R)^2} &= \big\langle  M_{c^*} u^*, Q_{c^*} \big\rangle_{L^2(\R)^2} = \big\langle u^*, S \partial_x Q_{c^*} \big\rangle_{L^2(\R)^2} \\
 & = \big\langle \eps^*,  \boH_{c^*}(\partial_x Q_{c^*}) \big\rangle_{L^2(\R)^2} = 0.
\end{align*}
This is the reason why we do not need to establish a quadratic dependence of $(c^*)'(t)$ on $\eps^*$.

Next, we use \eqref{form:vc}, \eqref{eq:modul1}, \eqref{eq:modul0bis} and \eqref{def:Mc} to bound the fifth term in the following way.
\begin{step}
\label{L5}
There exists a positive number $B_5$, depending only on $\gc$, such that
\begin{equation*}
\big| \boJ_5^*(t) \big| := \big| (c^*)' \big| \Big| \big\langle \partial_c M_{c^*} u^*, u^* \big\rangle_{L^2(\R)^2} \Big| \leq B_5 \big\| \eps^*(\cdot, t) \big\|_{X(\R)}^\frac{1}{2} \big\| u^*(\cdot, t) \big\|_{X(\R)}^2,
\end{equation*}
for any $t \in \R$.
\end{step}

Finally, we have in the same way.

\begin{step}
\label{L6}
There exists a positive number $B_6$, depending only on $\gc$, such that
$$\big| \boJ_6^*(t) \big| := \big| (c^*)' \big| \Big| \big\langle M_{c^*} u^*, S \partial_c \boH_{c^*}(\eps^*) \big\rangle_{L^2(\R)^2} \Big| \leq B_6 \big\| \eps^*(\cdot, t) \big\|_{X(\R)}^\frac{1}{2} \big\| u^*(\cdot, t) \big\|_{X(\R)}^2,$$
for any $t \in \R$.
\end{step}

We conclude the proof of Proposition \ref{prop:monou2} by combining the six previous steps to obtain \eqref{eq:moyenplateau}, with $B_* := \max \big\{ 1/B_1, B_2 + B_3 + B_5 + B_6 \}$. \qed

\subsection{Proof of Corollary \ref{cor:bomono}}

Corollary \ref{cor:bomono} is a consequence of Propositions \ref{prop:monou1} and \ref{prop:monou2}. We combine the two estimates \eqref{eq:petitplateau} and \eqref{eq:moyenplateau} with the definition of $N(t)$ to obtain
$$\frac{d}{dt} \Big( \langle N(t) u^*(\cdot, t), u^*(\cdot, t) \rangle_{L^2(\R)^2} \Big) \geq \Big( \frac{1 - \gc^2}{16} - A_* B^2_* e^{2 R_*} \big\| \eps^*(\cdot, t) \big\|_{X(\R)}^\frac{1}{2} \Big) \big\| u^*(\cdot, t) \big\|_{X(\R)}^2,$$
for any $t \in \R$. In view of \eqref{eq:modul0bis}, we fix the parameter $\beta_\gc$ such that
$$\big\| \eps^*(\cdot, t) \big\|_{X(\R)}^\frac{1}{2} \leq \frac{1 - \gc^2}{32 A_* B^2_* e^{2 R_*}},$$
for any $t \in \R$, to obtain \eqref{eq:grandplateau}.
In view of \eqref{form:vc}, \eqref{eq:modul0bis} and \eqref{def:Mc}, we notice that there exists a positive number $A_\gc$, depending only on $\gc$, such that
\begin{equation}
\label{cleveland}
\big\| M_{c^*(t)} \big\|_{L^\infty(\R)} \leq A_\gc,
\end{equation}
for any $t \in \R$. Moreover, since the map $t \mapsto \langle N(t) u^*(\cdot, t), u^*(\cdot, t) \rangle_{L^2(\R)^2}$ is uniformly bounded by \eqref{int-der-u*-exp} and \eqref{cleveland}, estimate \eqref{eq:petitpignon} follows by integrating \eqref{eq:grandplateau} from $t = - \infty$ to $t = + \infty$. Finally, statement \eqref{eq:derailleur} is a direct consequence of \eqref{eq:petitpignon}.

\appendix
\section{Appendix}
\subsection{Weak continuity of the hydrodynamical flow}
In this section, we prove the weak continuity of the hydrodynamical flow which is stated in the following proposition.
\begin{prop}
\label{prop:w-cont-Q}
We consider a sequence $(v_{n, 0}, w_{n, 0})_{n \in \N} \in \boN\boV(\R)^\N$, and a pair $(v_0, w_0) \in \boN\boV(\R)$ such that
\begin{equation}
\label{hyp:w-cont-Q}
v_{n, 0} \rightharpoonup v_0 \quad {\rm in} \ H^1(\R), \quad {\rm and} \quad w_{n, 0} \rightharpoonup w_0 \quad {\rm in} \ L^2(\R),
\end{equation}
as $n \to + \infty$. We denote by $(v_n, w_n)$ the unique solution to \eqref{HLL} with initial datum $(v_{n, 0}, w_{n, 0})$ and we assume that there exists a positive number $T_n$ such that the solutions $(v_n, w_n)$ are defined on $(- T_n, T_n)$, and satisfy the condition
\begin{equation}
\label{sup_n-sup_t:vn}
\sup_{n \in \N} \, \sup_{t \in (- T_n, T_n)} \, \max_{x \in \R} v_n(x, t) \leq 1 - \sigma,
\end{equation}
for a given positive number $\sigma$. Then, the unique solution $(v, w)$ to \eqref{HLL} with initial datum $(v_0, w_0)$ is defined on $(- T_{{\rm max}}, T_{{\rm max}})$, with\footnote{See Theorem 1 in \cite{DeLGr} for more details.} 
$$T_{\rm max} = \liminf_{n \to + \infty} T_n ,$$ and for any $t \in (- T_{\rm max}, T_{\rm max})$, we have
\begin{equation}
\label{eq:w-cont-Q}
v_n(t) \rightharpoonup v(t) \quad {\rm in} \ H^1(\R), \quad {\rm and} \quad w_n(t) \rightharpoonup w(t) \quad {\rm in} \ L^2(\R),
\end{equation}
as $n \to + \infty$. 
\end{prop}
First we prove a weak continuity property of the flow of equations \eqref{eq:Psi}--\eqref{eq:v-Psi}. Next, we deduce the weak convergence of $w_n$ from \eqref{eq:w^*}.
 
More precisely, we consider now a sequence of initial conditions $(\Psi_{n, 0},v_{n, 0}) \in L^2(\R) \times H^1(\R),$ such that the norms $\|\Psi_{n, 0}\|_{L^2}$ and $\|v_{n, 0}\|_{L^2}$ are uniformly bounded with respect to $n$ and we assume that
\begin{equation}
\label{sup_n:v_n,0}
\sup_{n \in \N} \|v_{n, 0}\|_{L^\infty(\R)} <1.
\end{equation}
Then, there exist two functions $\Psi_0 \in L^2(\R)$ and $v_0 \in H^1(\R)$ such that, going possibly to a subsequence,
\begin{equation}
\label{hyp:Psi-cont1}
\Psi_{n, 0} \rightharpoonup \Psi_0 \quad {\rm in} \ L^2(\R),
\end{equation}
\begin{equation}
\label{hyp:v-cont1}
v_{n, 0} \rightharpoonup  v_0 \quad {\rm in} \ H^1(\R), 
\end{equation}
and, for any compact subset $K$ of $\R$,
\begin{equation}
\label{hyp:v-cont2}
v_{n, 0} \to v_0 \quad {\rm in} \ L^\infty(K), 
\end{equation}
as $n \to + \infty$. We claim that this convergence is conserved along the flow corresponding to equations \eqref{eq:Psi}-\eqref{eq:v-Psi}\footnote{We only consider here positive time but the proof remains available for negative time.}.
\begin{prop}
\label{prop:w-cont-Psi}
We consider two sequences $(\Psi_{n, 0})_{n \in \N} \in L^2(\R)^\N$ and $(v_{n, 0})_{n \in \N} \in H^1(\R)^\N$, and two functions $\Psi_0 \in L^2(\R)$ and $v_0 \in H^1(\R)$, such that assumptions \eqref{sup_n:v_n,0}--\eqref{hyp:v-cont2}  are satisfied, and we denote by $(\Psi_n,v_n)$, respectively $(\Psi,v)$, the unique global solutions to \eqref{eq:Psi}-\eqref{eq:v-Psi} with initial datum $(\Psi_{n, 0},v_{n, 0})$, respectively $(\Psi_0,v_0)$, which we assume to be defined on $[0,T]$ for a positive number $T$. For any fixed $t \in [0,T]$, we have
\begin{equation}
\label{eq:Psi-conv}
\Psi_n(\cdot, t) \rightharpoonup \Psi(\cdot, t) \quad {\rm in} \ L^2(\R),
\end{equation}
and
\begin{equation}
\label{eq:v-conv}
v_n(\cdot, t) \rightharpoonup v(\cdot, t) \quad {\rm in} \ H^1(\R),
\end{equation}
when $n \to + \infty$.
\end{prop}

\begin{proof}
We split the proof into four steps.
\setcounter{step}{0}
\begin{step}
\label{T1}
There exist three functions $\Phi \in L^2( [0,T], L^2(\R))$ and $ \gv \in L^2 ([0,T], H^1(\R))$ such that, up to a further subsequence,
\begin{equation}
\label{conv-Psin}
\Psi_n(t) \rightharpoonup \Phi(t) \quad {\rm in} \quad L^2(\R),
\end{equation} 

\begin{equation}
\label{conv-vn}
v_n(\cdot, t) \rightharpoonup \gv(\cdot, t) \quad {\rm in} \quad H^1(\R),
\end{equation}
\begin{equation}
\label{convfortloc-vn}
v_n(\cdot, t) \longrightarrow \gv(\cdot, t) \quad {\rm in} \quad L^\infty_{loc}(\R),
\end{equation}
for all $t \in [0,T]$, and

\begin{equation}
\label{conv-Psin3}
|\Psi_n|^2 \Psi_n \rightharpoonup |\Phi|^2 \Phi \quad {\rm in} \quad L^2( [0,T], L^2(\R)),
\end{equation}
when $n \to + \infty$.
\end{step}
\begin{proof}
We recall that there exists a constant $M$ such that
$$ \| \Psi_{n,0}\|_{L^2} \leq M \quad {\rm and} \quad \| v_{n,0}\|_{H^1} \leq M ,$$
uniformly on $n$. Applying Proposition \ref{prop:cont-Cauchy} to the pairs $(\Psi_n,v_n)$ and $(0,0)$, we obtain 
$$\| \Psi_n\|_{\boC^0_T L^2_x} + \| v_n\|_{\boC^0_T H^1_x} + \| \Psi_n\|_{L^4_T L^\infty_x} \leq A \Big( \| \Psi_{n,0}\|_{L^2} + \| v_{n,0}\|_{H^1} \Big).$$
This leads to
\begin{equation}
\label{bornePsinvn}
\| \Psi_n\|_{L^4_T L^\infty_x} \leq 2 A M, \quad \| \Psi_n\|_{L^\infty_T L^2_x} \leq 2 A M, \quad {\rm and} \quad \| v_n\|_{L^\infty_T H^1_x} \leq 2 A M.
\end{equation}
Hence, there exist two functions $\Phi \in L^\infty([0,T],L^2(\R))\bigcap L^4([0,T],L^\infty(\R))$ and \\
$ \gv \in L^\infty ([0,T], H^1(\R))$ such that 
$$\Psi_n \overset{*}{\rightharpoonup} \Phi \quad {\rm in} \ L^\infty([0, T],L^2(\R))$$
and
$$v_n \overset{*}{\rightharpoonup} \gv \quad {\rm in} \ L^\infty([0, T],H^1(\R))$$
Let us prove \eqref{conv-Psin} and \eqref{conv-vn}. We argue as in \cite{BetGrSm2} and we introduce a cut-off function $\chi \in \boC_c^\infty(\R)$ in the way that $\chi \equiv 1$ on $[- 1, 1]$ and $\chi \equiv 0$ on $(- \infty, 2] \cup [2, + \infty)$. Denote $\chi_p(\cdot) := \chi(\cdot/p)$ for any integer $p \in \N^*$. By \eqref{bornePsinvn}, the sequences $(\chi_p \Psi_n)_{ n \in \N}$ and $(\chi_p v_n)_{ n \in \N}$ are bounded in $\boC^0([0, T], L^2(\R))$ and $\boC^0([0, T], H^1(\R))$ respectively. In view of the Rellich-Kondrachov theorem, the sets $\{ \chi_p \Psi_n(\cdot, t), n \in \N \}$ and $\{ \chi_p v_n(\cdot, t), n \in \N \}$ are relatively compact in $H^{-2}(\R)$ and $H^{-1}(\R)$ respectively, for any fixed $t \in [0, T]$. In addition, since the couple $(\Psi_n,v_n)$ is solution to \eqref{eq:Psi}-\eqref{eq:v-Psi}, we have $(\partial_t \Psi_n,\partial_t v_n)$ belongs to $\boC^0([0, T], H^{- 2}(\R) \times H^{- 1}(\R))$ and satisfies
$$\| \partial_t \Psi_n(\cdot, t) \|_{H^{- 2}(\R)}  \leq K_M \quad {\rm and} \quad \| \partial_t v_n(\cdot, t) \|_{H^{- 1}(\R)}  \leq K_M .$$
This leads to the fact that the couple $(\chi_p\Psi_n,\chi_p v_n)$ is equicontinuous in $\boC^0([0, T], H^{- 2}(\R) \times H^{- 1}(\R))$. Then, we apply the Arzela-Ascoli theorem and the Cantor diagonal argument, to find a further subsequence (independent of $p$), such that, for each $p \in \N^*$,
\begin{equation}
\label{convfortloc:C0psi_n}
\chi_p \Psi_n \to \chi_p \Phi \quad {\rm in} \ \boC^0([0, T], H^{- 2}(\R)),
\end{equation}
and
\begin{equation}
\label{convfortloc:C0v_n}
\chi_p v_n \to \chi_p \gv \quad {\rm in} \ \boC^0([0, T], H^{- 1}(\R)),
\end{equation}
as $n \to + \infty$. Combining this with \eqref{bornePsinvn} we infer that \eqref{conv-Psin} and \eqref{conv-vn} hold. By the Sobolev embedding theorem, \eqref{convfortloc-vn} is a consequence of \eqref{conv-vn}. 

Now, let us prove \eqref{conv-Psin3}. Using the Hölder inequality, we infer that
$$ \int_0^T \int_\R |\Psi_n(x,t)|^6 dx dt \leq \|\Psi_n\|_{L^\infty L^2_x}^2  \|\Psi_n\|_{L^4_T L^\infty_x}^4.$$
By \eqref{bornePsinvn}, we conclude that
\begin{equation}
\label{borne:Psin2Psi}
\||\Psi_n|^2 \Psi_n \|_{L^2_T L^2_x} \leq M.
\end{equation}
Then, there exists a function $\Phi_1 \in L^2( \R \times [0,T])$ such that up to a further subsequence,
$$ |\Psi_n|^2\Psi_n \rightharpoonup \Phi_1 \quad {\rm in} \quad L^2(\R \times [0,T]).$$
Let us prove that $\Phi_1 \equiv |\Phi|^2\Phi$. To obtain this it is sufficient to prove that, up to a subsequence, 
\begin{equation}
\label{convfortloc-Psin}
\Psi_n \longrightarrow \Phi \quad {\rm in} \ L^2( [0,T], L^2([-R,R])),
\end{equation}
for any $R>0$, i.e the sequence $(\Psi_n)$ is relatively compact in $L^2([-R,R] \times [0,T]).$ Indeed, using the Holder inequality, we obtain
\begin{equation}
\label{estim:L6/5}
\begin{array}{ll}
\||\Psi_n|^2 \Psi_n - |\Phi|^2 \Phi\|_{L^{\frac{6}{5}}_{T,R}} & =  \| (\Psi_n - \Phi ) ( |\Psi_n|^2 + |\Phi|^2 )+ \Psi_n\Phi(\overline{\Psi}_n - \overline{\Phi}) \|_{L^{\frac{6}{5}}_{T,R}} \\
                   & \leq 2 \| |\Psi_n - \Phi | ( |\Psi_n|^2 + |\Phi|^2 ) \|_{L^{\frac{6}{5}}_{T,R}} \\
                   &  \leq 2 \| \Psi_n - \Phi \|_{L^{2}_{T,R}} \big( \| \Psi_n \|^2_{L^{6}_{T,R}} + \| \Phi \|^2_{L^{6}_{T,R}} \big),
\end{array}
\end{equation}
for any $R>0$. By \eqref{borne:Psin2Psi}, $(\Psi_n)$ is uniformly bounded in $ L^{6}(\R \times [0,T])$ and $\Phi \in L^{6}(\R \times [0,T])$. Then
$$ |\Psi_n|^2\Psi_n \longrightarrow |\Phi|^2 \Phi \quad {\rm in} \quad L^{\frac{6}{5}}([-R,R] \times [0,T]).$$ So that $\Phi_1 \equiv |\Phi|^2\Phi$. Now, let us prove that the sequence $(\Psi_n)$ is relatively compact in $L^2([-R,R] \times [0,T]).$ The main point of the proof is the following claim.
\begin{claim}
\label{claim:kato}
Let $\Psi$ be a solution of \eqref{eq:Psi} in $\boC ^0 ([0,T], L^2(\R))\bigcap L^4([0,T],L^\infty(\R)).$ Then, \\ $\Psi \in L^2 ([0,T],H^{\frac{1}{2}}_{loc}(\R)).$
\end{claim}
\begin{proof}
The proof relies on the Kato smoothing effect for the linear Schrödinger group (see \cite{LinaPon0}). Denote $S(t)=e^{it\partial_{xx} },$ and
\begin{equation}
\label{def:Fpsiv}
\boF(\Psi,v) := \frac{1}{2} v^2 \Psi - \Re \Big( \Psi \big( 1 - 2 F(v, \overline{\Psi}) \big) \Big) \big( 1 - 2 F(v, \Psi) \big).
\end{equation}
We recall that there exists a positive constant $M$ such that
\begin{equation}
\label{estim:D1/2L2}
\sup_{x \in \R} \int^{+\infty}_{-\infty} |D^{\frac{1}{2}}_x S(t) f(x) |^2 dt \leq M \|f\|^2_{L^2},
\end{equation}
and
\begin{equation}
\label{estmSD12h}
\Big\| \int_\R S(-t') D^{\frac{1}{2}}_x h(\cdot ,t')dt' \Big\| _{L^2} \leq M \| h\|_{L^1_x L^2_t},
\end{equation}
when $f \in L^2(\R)$ and $h \in L^1(\R , L^2(\R))$ (see \cite{LinaPon0} for more details). We prove that there exists a positive constant $M$ such that
\begin{equation}
\label{estimD12Linf}
\|D^{\frac{1}{2}}_x \Psi\|_{L^\infty_x L^2_T} \leq M \|\Psi_0\|_{L^2} + M  \|\Psi\|_{L^2_{T,x}} \Big( \|\Psi\|^2_{L^6_{T,x}} + T^\frac{1}{2}\big( \|v\|^2_{L^\infty_{T,x}} + \|1 - 2 F(v, \Psi)\|^2_{L^\infty_{T,x}}\big) \Big).
\end{equation}
The claim is a consequence of this estimate, so that it is sufficient to prove \eqref{estimD12Linf}. 

We write
$$ \Psi(x,t) = S(t) \Psi_0(x) + i \int^t_0 S(t-t') \Big( 2 (|\Psi|^2 \Psi)(x,t') + \boF(\Psi,v)(x,t') \Big) dt',$$
for all $(x,t) \in \R$. First, using \eqref{estim:D1/2L2}, we obtain
$$\sup_{x \in \R} \int^{+\infty}_{-\infty} |D^{\frac{1}{2}}_x S(t) \Psi_0(x) |^2 dt \leq M \|\Psi_0\|^2_{L^2}.$$
For the nonlinear term, we can argue as in \cite{Goubet96regularityof} to prove that

\begin{equation}
\label{estmD12g}
\Big\| \int_0^t S(t-t') D^{\frac{1}{2}}_x g(\cdot ,t')dt' \Big\| _{L^\infty_x L^2_T} \leq M \| g\|_{L^1_{T}L^2_x}.
\end{equation}
Using a duality argument, it is equivalent to prove that for any smooth function $h$ that satisfies $\| h\|_{L^1_x L^2_t} \leq 1,$ we have
\begin{equation}
\label{estmD12gdual}
\Big| \int_{\R \times[0,T]^2} S(t-t') D^{\frac{1}{2}}_x g(x ,t') \overline{h}(x,t) dt' dx dt \Big| \leq M \| g\|_{L^1_{T}L^2_x}.
\end{equation}
The left-hand side can be written, using the Cauchy-Schwarz and Stricharz estimates, and \eqref{estmSD12h}, as
\begin{align*}
&\Big| \int_\R \Big( \int_0^T S(-t') D^{\frac{1}{2}}_x g(x ,t') dt' \Big) \Big( \int_0^T \overline{ S(-t) h(x,t)}dt \Big) dx \Big| \\
& = \Big| \int_\R \Big( \int_0^T S(-t')  g(x ,t') dt' \Big) \Big( \int_0^T \overline{ S(-t) D^{\frac{1}{2}}_x h(x,t)}dt \Big) dx \Big| \\
& \leq M \Big\| \int_0^T S(-t') g(x ,t') dt' \Big\|_{L^2} \leq M \| g\|_{L^1_TL^2_x}.
\end{align*} 
This achieves the proof of \eqref{estmD12g}. 
Similarly, we have
\begin{equation}
\label{estmD12gL5/6}
\Big\| \int_0^t S(t-t') D^{\frac{1}{2}}_x g(\cdot ,t')dt' \Big\| _{L^\infty_x L^2_T} \leq M \| g\|_{L^{\frac{5}{6}}_{T,x}}.
\end{equation}
We next apply \eqref{estmD12g} and \eqref{estmD12gL5/6} on the nonlinear terms to obtain, using the Cauchy-Schwarz and Hölder estimates,
$$\Big\| \int^t_0 D^{\frac{1}{2}}_x S(t-t') (|\Psi|^2 \Psi)(\cdot,t') dt'\Big\|_{L^\infty_x L^2_T} \leq M \| \Psi^3\|_{L^{\frac{6}{5}}_{T,x}} \leq M \| \Psi \|_{L^2_{T,x}} \| \Psi \|^2_{L^6_{T,x}},$$
and
\begin{align*}
&\Big\| \int^t_0 D^{\frac{1}{2}}_x S(t-t') \boF(\Psi,v)(\cdot,t') dt'\Big\|_{L^\infty_x L^2_T} \leq M \| \boF(\Psi,v)\|_{L^1_TL^2_x} \\
& \ \ \leq M  \| \Psi \|_{L^1_T L^2_x} \big( \| v \|^2_{L^\infty_{T,x}} + \| 1 - 2 F(v, \Psi) \|^2_{L^\infty _{T,x}} \big) \\
& \ \ \leq M T^\frac{1}{2} \| \Psi \|_{L^2_{T,x}} \big( \| v \|^2_{L^\infty _{T,x}} + \| 1 - 2 F(v, \Psi) \|^2_{L^\infty_{T,x}} \big).
\end{align*}
Since $ v \in L^\infty ( [0,T], H^1(\R))$ and $\Psi \in L^\infty([0,T],L^2(\R))$, we know that $\Psi \in L^\infty([0,T],L^2(\R))$ and $ F(\Psi,v) \in L^\infty(\R \times [0,T])$. 
Using the fact that $\Psi \in L^{6}(\R \times [0,T]),$ we finish the proof of this claim.
\end{proof}
Applying this claim to the sequence $(\Psi_n)$ yields that $(\Psi_n)$ is uniformly bounded in the space $L^2 ([0,T],H^{\frac{1}{2}}_{loc}(\R)).$ On the other hand, we have $\boF(\Psi_n,v_n) \in  L^\infty([0,T],L^2(\R)),$ since\\ $ v_n \in L^\infty ( [0,T], H^1(\R))$, $\Psi_n \in L^\infty([0,T],L^2(\R))$ and $ F(\Psi_n,v_n) \in L^\infty(\R \times [0,T])$. Then, using \eqref{eq:Psi} and \eqref{borne:Psin2Psi}, we obtain that $(\Psi_n)$ is uniformly bounded in $H^1 ([0,T],H^{-2}(\R)).$ Hence, by interpolation $(\Psi_n)\in H^{\frac{1}{10}} ([0,T],H_{\rm loc}^{\frac{1}{4}}(\R))$ so that it converges in $L^2([-R,R] \times [0,T])$ for any $R>0$. This finishes the proofs of \eqref{convfortloc-Psin} and of Step \ref{T1}.
\end{proof}

\begin{step}
\label{T2}
We have
\begin{equation}
\label{convfaible:bofn}
\boF(\Psi_n,v_n) \rightharpoonup \boF(\Phi,\gv) \quad {\rm in} \quad L^2(\R),
\end{equation}
for any $t \in [0,T]$, and
\begin{equation}
\label{conv:bofn}
\boF(\Psi_n,v_n) \longrightarrow \boF(\Phi,\gv) \quad {\rm in} \quad L^1([0,T],L_{loc}^2(\R)).
\end{equation}
\end{step}
\begin{proof}
Let $\phi \in L^2(\R)$. We compute
\begin{equation}
\label{int:v2npsin-gvphi}
\begin{array}{ll}
&\int_{\R} \big( v_n^2(x,t) \Psi_n(x,t) - \gv ^2(x,t) \Phi(x,t) \big) \phi(x) dx\\
&  = \int_{\R} \big( v_n^2(x,t) - \gv ^2(x,t) \big) \Psi_n(x,t) \phi(x) dx + \int_{\R} \big( \Psi_n(x,t) - \Phi(x,t) \big) \gv ^2(x,t) \phi(x) dx.
\end{array}
\end{equation}
The second term in the right-hand side goes to $0$ when $n$ goes to $+\infty$, since $\gv^2(t) \phi \in L^2(\R)$ for all $t$ on one hand and using \eqref{conv-Psin} on the other hand. For the first term in the right-hand side, we consider a cut-off function $\chi$ with support into $[-1,1]$ and denote $\chi_R(x) = \chi(\frac{x}{R})$ for all $(x,R) \in \R \times (0,+\infty) $. We set
$$I_n(t) := \int_{\R} \big( v_n^2(x,t) - \gv ^2(x,t) \big) \Psi_n(x,t) \phi(x) dx,$$
$$I_n^{(1)}(t):= \int_{\R} \big( v_n^2(x,t) - \gv ^2(x,t) \big) \Psi_n(x,t) \chi_R(x) \phi(x)  dx,$$
and
$$I_n ^{(2)}(t):= \int_{\R} \big( v_n^2(x,t) - \gv ^2(x,t) \big) \Psi_n(x,t) \big( 1 - \chi_R(x) \big) \phi(x)  dx,$$
so that $ I_n(t) = I_n^{(1)}(t) + I_n^{(2)}(t).$ By the Cauchy-Schwarz inequality, we have
\begin{equation}
\label{estim:In1}
| I^{(1)}_n(t) | \leq \| \Psi_n(t)\|_{L^2(\R)}  \| \phi\|_{L^2(\R)} \| v^2_n(t) - \gv^2(t) \|_{L^\infty([-R,R])}.
\end{equation}
Using \eqref{convfortloc-vn} and \eqref{bornePsinvn}, we infer that 
\begin{equation}
\label{conv:In1LinftyL2}
I^{(1)}_n(t) \rightarrow 0 \quad {\rm for \ any } \quad t \in [0,T],
\end{equation}
as $n$ goes to $+\infty.$
Next, we write
$$ |I^{(2)}_n(t)| \leq \big( \|v_n(t)\|^2_{L^\infty(\R)} + \|\gv(t)\|^2_{L^\infty(\R)} \big) \|\Psi_n(t)\|_{L^2(\R)} \| ( 1 - \chi_R ) \phi\|_{L^2(\R)}.$$
Since $\phi \in L^2(\R)$, we have
$$\lim_{R \rightarrow \infty}  \| ( 1 - \chi_R ) \phi\|_{L^2(\R)} =0.$$
In view of \eqref{bornePsinvn}, this is sufficient to prove that
\begin{equation}
\label{conv-In}
I_n(t) \rightarrow 0 
\end{equation}
as n goes to $+\infty$, for all $t \in [0,T].$ This yields
\begin{equation}
\label{conv-vn2psin}
(v^2_n \Psi_n)(t) \rightharpoonup (\gv ^2 \Phi)(t) \quad {\rm in} \quad L^2(\R),
\end{equation}
for any $t\in[0,T]$. Now, we prove
\begin{equation}
\label{convfortloc-vn2psin}
v^2_n \Psi_n \longrightarrow \gv ^2 \Phi \quad {\rm in} \quad L^1([0,T],L^2_{loc}(\R)).
\end{equation}
We write as in \eqref{int:v2npsin-gvphi},
$$\| v_n^2 \Psi_n - \gv ^2 \Phi  \|_{L^1_TL^2_R} \leq \| ( v_n^2 - \gv ^2 ) \Psi_n\|_{L^1_TL^2_R} + \| ( \Psi_n - \Phi \big) \gv ^2\|_{L^1_TL^2_R}.$$
For the first term in the right-hand side, we infer from the Cauchy-Schwarz inequality, that 
\begin{equation*}
\begin{split}
\| ( v_n^2 - \gv ^2 ) \Psi_n\|_{L^1_TL^2_R}& \leq \| v_n^2 - \gv ^2 \|_{L^2_TL^2_R} \| \Psi_n\|_{L^2_TL^\infty_R} \\ 
             & \leq \| v_n - \gv \|_{L^4_TL^4_R} \big( \| v_n \|_{L^4_TL^4_R} + \| \gv \|_{L^4_TL^4_R} \big) T^\frac{1}{2} \| \Psi_n\|_{L^4_TL^\infty_R}.
\end{split}
\end{equation*}
On the other hand, by \eqref{bornePsinvn}, $v_n$ is uniformly bounded on $L^2([0,T],H^1(\R))$. By the first equation of \eqref{eq:v-Psi} and \eqref{bornePsinvn}, $v_n$ is uniformly bounded in $H^1([0,T],H^{-1}(\R))$. We deduce that $v_n$ is uniformly bounded in $H^{\frac{1}{3}}([0,T],H^{\frac{1}{3}}(\R))$ and so that $v_n $ converges to $\gv$ in $L^4([0,T],L^{4}([-R,R]))$ when $n$ goes to $+\infty$. Hence, using \eqref{bornePsinvn} once again,  we obtain
$$ \| ( v_n^2 - \gv ^2 ) \Psi_n\|_{L^1_TL^2_R} \rightarrow 0,$$
as $n$ goes to $+\infty$. For the second term we have by the Cauchy-Schwarz inequality and the Sobolev embedding theorem,
$$ \| ( \Psi_n - \Phi ) \gv ^2 \|_{L^1_TL^2_R} \leq \| \Psi_n - \Phi \|_{L^2_TL^2_R} \| \gv ^2 \|_{L^2_TL^\infty _R} \leq M^2 T^{1/2} \| \Psi_n - \Phi \|_{L^2_TL^2_R}.$$
This yields using \eqref{convfortloc-Psin}, $$\| ( \Psi_n - \Phi ) \gv ^2 \|_{L^1_TL^2_R} \longrightarrow 0,$$
as $n$ goes to $+\infty$, which proves \eqref{convfortloc-vn2psin}. Next, we set
$$\boG(v_n,\Psi_n) = \Psi_n  \big( 1- F(v_n,\overline{\Psi}_n)\big) \big( 1- F(v_n,\Psi_n)\big).$$
We have by \eqref{def:F1},
$$ \partial_x  F(v_n,\Psi_n) = v_n \Psi_n \quad {\rm and} \quad \partial_x  F(\gv,\Phi) = \gv \Phi .$$
Using the same arguments as in the proof of \eqref{conv-In}, we obtain
$$ \partial_x  F(v_n,\Psi_n)  \rightharpoonup \partial_x  F(\gv,\Phi) \quad {\rm in} \quad L^2(\R)),$$
for any $t \in [0,T].$ Hence, 
\begin{equation}
\label{convfortlocFn}
F(v_n,\Psi_n)  \longrightarrow F(\gv,\Phi) \quad {\rm in} \quad L^\infty_{loc}(\R),
\end{equation} 
for any $t \in [0,T].$ Using \eqref{conv-Psin}, \eqref{convfortlocFn} and the same arguments as in the proof of \eqref{conv-vn2psin}, we conclude that
\begin{equation}
\label{conv:boG}
\boG(v_n,\Psi_n) \rightharpoonup  \boG(\gv,\Phi) \quad {\rm in} \quad L^2(\R),
\end{equation}
for any $t \in [0,T].$ Next, we use \eqref{convfortloc-Psin} and \eqref{convfortlocFn} to prove that
\begin{equation}
\label{convfortloc:boG}
\boG(v_n,\Psi_n) \longrightarrow  \boG(\gv,\Phi) \quad {\rm in} \quad L^1([0,T],L^2_{loc}(\R)).
\end{equation}
This finishes the proof of this step.
\end{proof}

\begin{step}
\label{T3}
$(\Phi,\gv)$ is a weak solution of \eqref{eq:Psi}--\eqref{eq:v-Psi}.
\end{step} 
\begin{proof}
By \eqref{convfortloc-Psin}, we have
$$i \partial_t \Psi_n \to i \partial_t \Phi \quad {\rm in} \ \boD'(\R \times [0, T]), \quad {\rm and} \quad \partial_{xx}^2 \Psi_n \to \partial_{xx}^2 \Phi \quad {\rm in} \ \boD'(\R \times [0, T]),$$
as $n \to + \infty$. It remains to invoke \eqref{conv-Psin3} and \eqref{convfortlocFn} and to take the limit $n \to + \infty$ in the expression
$$\int_0^T \int_\R \big( i \partial_t \Psi_n + \partial_{xx}^2 \Psi_n + 2 |\Psi_n|^2 \Psi_n + \frac{1}{2} v_n^2 \Psi_n - \Re \big( \Psi \big( 1 - 2 F(v_n, \overline{\Psi_n}) \big) \big) \big( 1 - 2 F(v_n, \Psi_n) \big)\big) \overline{h} = 0,$$
where $h \in \boC_c^\infty(\R \times [0, T])$, in order to establish that $(\Phi , \gv)$ is solution to \eqref{eq:Psi} in the sense of distributions. In addition, using the same arguments as above and \eqref{convfortlocFn} we prove that $(\Phi,\gv)$ is solution to \eqref{eq:v-Psi} in the sense of distributions. Moreover, we infer from \eqref{hyp:Psi-cont1} that $\Phi(\cdot, 0) = \Psi_0$ and from \eqref{hyp:v-cont1} that $\gv(\cdot, 0) = v_0$.
\end{proof}

In order to prove that the function $(\Phi,\gv)$ coincides with the solution $(\Psi,v)$ in Proposition \ref{prop:w-cont-Psi}, it is sufficient, in view of the uniqueness result given by Proposition \ref{prop:cont-Cauchy}, to establish that
\begin{step}
\label{T4}
$\Phi \in \boC([0,T],L^2(\R))$ and $\gv \in \boC([0,T],H^1(\R)).$
\end{step}

\begin{proof}
First, we prove that $\Phi \in \boC([0,T],L^2(\R))$. This is a direct consequence of the identity
\begin{equation}
\label{duhamel:phi}
\Phi(x,t) = S(t) \Phi_0 + \int^t_0 S(t-t') \big( 2 (|\Phi|^2 \Phi)(\cdot,t') + \boF(\Phi,\gv)(\cdot,t') \big) dt'.
\end{equation}
Indeed, let us denote
$$ G(\Phi,\gv)(t) = \int^t_0 S(t-t') \big( 2 (|\Phi|^2 \Phi)(\cdot,t') + \boF(\Phi,\gv)(\cdot,t') \big) dt'.$$
Since $S(t)\Phi_0 \in \boC([0,T],L^2(\R))$, it is enough to show that $G(\Phi,\gv) \in \boC([0,T],L^2(\R)).$ We take $(t_1,t_2) \in [0,T]^2$ and we write
\begin{align*}
G(\Phi,\gv)(t_1) - G(\Phi,\gv)(t_2) =  & \int^{t_1}_0 \big( S(t_1-t') - S(t_2-t') \big) \big( 2 (|\Phi|^2 \Phi)(\cdot,t') + \boF(\Phi,\gv)(\cdot,t') \big) dt' \\
                                & - \int^{t_2}_{t_1} S(t-t')  \big( 2 (|\Phi|^2 \Phi)(\cdot,t') + \boF(\Phi,\gv)(\cdot,t') \big) dt'.
\end{align*}
For the second term in the right-hand side, we use the Stricharz and Cauchy-Schwarz inequalities to obtain
\begin{equation}
\begin{array}{ll}
\label{estim:L2Gphi}
& \Big\| \int^{t_2}_{t_1} S(t-t')  \big( 2 (|\Phi|^2 \Phi)(\cdot,t') + \boF(\Phi,\gv)(\cdot,t') \big) dt' \Big\|_{L^2} \\
& \leq M  \| 2|\Phi|^2 \Phi + \boF(\Phi,\gv) \|_{L^1([t_1,t_2],L^2(\R))}\\  
& \leq M |t_1-t_2|^{\frac{1}{2}} \||\Phi|^2 \Phi \|_{L^2_{T,x}} + M |t_1-t_2| \| \boF(\Phi,\gv) \|_{L^\infty_TL^2_x}.
\end{array}
\end{equation}
For the first term, we write
$$ S(t_1-t') - S(t_2-t') = S(t_1-t') \big( 1- S(t_2 -t_1) \big) .$$
Hence,
\begin{equation}
\label{egaltGphi}
\begin{split}
&\Big\| \int^{t_1}_0 \big( S(t_1-t') - S(t_2-t') \big) \big( 2 (|\Phi|^2 \Phi)(\cdot,t') + \boF(\Phi,\gv)(\cdot,t') \big) dt' \Big\|_{L^2} \\
& = \Big\| \big( 1 - S(t_2-t_1) \big)G(\Phi,\gv)(t_1) \Big\|_{L^2}.
\end{split}
\end{equation}
Taking the limit $t_2 \rightarrow t_1$ in \eqref{estim:L2Gphi} and \eqref{egaltGphi}, we obtain that $\Phi \in \boC([0,T],L^2(\R)).$

Now, let us prove \eqref{duhamel:phi}. Denote $\tilde{\Phi}$ the function given by the right-hand side of \eqref{duhamel:phi}. We will prove that 
\begin{equation}
\label{limitfaiblpsinphitild}
\Psi_n(t) \rightharpoonup \tilde{\Phi}(t) \quad {\rm in } \quad L^2(\R),
\end{equation}
for all $t \in \R$. This yields $\Phi \equiv \tilde{\Phi}$ by uniqueness of the weak limit.
Let $R>0$ and denote by $\chi_R$ the function defined in Step 2. Set

$$ G^{(1)}_n(\cdot,t) =  \int^t_0 S(t-t') \chi_R \big( 2 (|\Psi_n|^2 \Psi_n)(\cdot,t') + \boF(\Psi_n,v_n)(\cdot,t') \big) dt',$$
$$ G^{(2)}_n(\cdot,t) =  \int^t_0 S(t-t') (1 - \chi_R) \big( 2 (|\Psi_n|^2 \Psi_n)(\cdot,t') + \boF(\Psi_n,v_n)(\cdot,t') \big) dt',$$
$$ G^{(1)}(\cdot,t) =  \int^t_0 S(t-t') \chi_R \big( 2 (|\Phi|^2 \Phi)(\cdot,t') + \boF(\Phi,\gv)(\cdot,t') \big) dt',$$
and
$$ G^{(2)}(\cdot,t) =  \int^t_0 S(t-t') (1 - \chi_R) \big( 2 (|\Phi|^2 \Phi)(\cdot,t') + \boF(\Phi,\gv)(\cdot,t') \big) dt',$$
for all $t\in\R$, so that $G(\Phi,\gv)=G^{(1)} + G^{(2)}$ and $G(\Psi_n,v_n) = G^{(1)}_n + G^{(2)}_n$. Since $S(t)\Psi_{n,0} \rightharpoonup S(t) \Phi_0$ in $L^2(\R)$ as $n\rightarrow + \infty$ for all $t\in \R$, it is sufficient to show that $G(\Psi_n,v_n)(t) \rightharpoonup G(\Phi,\gv)(t)$ in $L^2(\R)$ as $n\rightarrow + \infty$ for all $t\in \R$.
We write 
\begin{align*}
\big( G(\Psi_n,v_n)(t)- G(\Phi,\gv)(t) , \varphi \big)_{L^2} & = \int_{-\infty}^{+\infty} \big[ G^{(1)}_n(x,t) - G^{(1)}(x,t) \big] \overline{\varphi(x)} dx \\ 
                    & + \int_{-\infty}^{+\infty} \big[ G^{(2)}_n(x,t) - G^{(2)}(x,t) \big] \overline{\varphi(x)} dx\\
            & = I^{R}_n(t) + J^{R}_n(t).
\end{align*}
For the first integral, using the Cauchy-Schwartz inequality, the Strichartz estimates for the admissible pairs $(6,6)$ and $(\infty,2)$, the Hölder inequality as well as \eqref{estim:L6/5}, there exists a positive constant $M$ such that for all $t \in [0,T]$ we have
\begin{align*}
|I^{R}_n(t)| & \leq \|G^{(1)}_n(t) - G^{(1)}(t) \|_{L^2} \|\varphi\|_{L^2} \\
             & \leq M \|\varphi\|_{L^2} \big( \||\Psi_n|^2 \Psi_n - |\Phi|^2 \Phi\|_{L^{\frac{6}{5}}_{T,R}} +  \| \boF(\Psi_n,v_n) - \boF(\Phi,\gv) \|_{L^1_TL^2_R} \big) \\
             & \leq M \|\varphi\|_{L^2} \big( \| \boF(\Psi_n,v_n) - \boF(\Phi,\gv) \|_{L^1_TL^2_R} + \| \Psi_n - \Phi \|_{L^{2}_{T,R}} \big( \| \Psi_n \|^2_{L^{6}_{T,R}} + \| \Phi \|^2_{L^{6}_{T,R}} \big)\big).
\end{align*}
Then, using \eqref{convfortloc-Psin} and \eqref{conv:bofn}, we obtain for all $t \in \R$
$$ |I^{R}_n(t)| \longrightarrow 0 \quad {\rm as} \quad n \rightarrow \infty .$$
Next, using the Hölder inequality we have
\begin{align*}
&|J_n^R(t)| \\
\leq & 2 \Big(\int_0^T \int_{-\infty}^\infty \Big||\Psi_n|^2 \Psi_n(x,t')-|\Phi|^2\Phi(x,t')\Big|^{\frac{6}{5}}   dxdt'\Big)^{\frac{5}{6}}\Big(\int_0^T \int_{|x|\geq R}|S(t-t')\varphi|^6 dx dt'\Big)^{\frac{1}{6}}  \\
                & + \int_0^T \Big( \int_{-\infty}^\infty \Big|\boF(\Psi_n,v_n)(x,t') - \boF(\Phi,\gv)(x,t')\Big|^2 dx \Big)^{\frac{1}{2}} dt'  \sup_{t' \in [0,T]} \Big( \int_{|x|\geq R}|S(t-t')\varphi(x)|^2 dx \Big)^{\frac{1}{2}}.
\end{align*}
The terms in the right-hand side are bounded by a constant independent of $n$. Besides, since $(6,6)$ and $(\infty,2)$ are admissible pairs, we have $\|S(t)\varphi\|_{L^6_{T,x}} \leq M \|\varphi\|_{L^2(\R)}$ and\\ $\|S(t)\varphi\|_{L^\infty_TL^2(\R)} \leq M \|\varphi\|_{L^2(\R)}$, so that, by the dominated convergence theorem and the fact that $t \mapsto S(t) $ is uniformly continuous from $[0,T]$ to $L^2(\R)$, we obtain
$$ \lim_{R\to\infty}\int_0^T\!\!\int_{|x|\geq R}|S(t)\varphi|^6 dx dt = \lim_{R\to\infty} \sup_{t \in [0,T]} \Big( \int_{|x|\geq R}|S(t)\varphi(x)|^2 dx \Big)^{\frac{1}{2}} = 0.$$
Hence,
$$ \lim_{R\to\infty} |J_n^R(t)|=0 \quad \mbox{uniformly with respect to}\;\; n\in\mathbb{N},$$
for any $ t \in [0,T].$ This completes the proof of \eqref{limitfaiblpsinphitild} and then of \eqref{duhamel:phi}. This leads to the fact that  $\Phi \in \boC^0([0,T],L^2(\R))$. 

 Now, let us prove that $\gv \in \boC^0([0,T],H^1(\R)).$ Since $(\Phi,\gv)$ verifies the first equation in \eqref{eq:v-Psi}, $\Phi \in L^\infty([0,T],L^2(\R))$ and $F(\Psi,\gv)\in L^\infty([0,T],L^\infty(\R)) $, we have $\gv \in H^1([0,T],H^{-1}(\R)).$ This yields, using the Sobolev embedding theorem, $\gv \in \boC^0([0,T],H^{-1}(\R)).$ Let $(t_1,t_2) \in [0,T]^2$. We can write that
\begin{align*}
\int_{\R} \big|\gv(t_1,x) - \gv(t_2,x) \big|^2 dx & = \Big< \gv(t_1,x) - \gv(t_2,x),\gv(t_1,x) - \gv(t_2,x) \Big>_{H^{-1},H^1}\\
                                                    & \leq \| \gv(t_1,x) - \gv(t_2,x)\|_{H^{-1}}  \| \gv(t_1,x) - \gv(t_2,x)\|_{H^1}.
\end{align*}
Since $\gv \in \boC^0([0,T],H^{-1}(\R)) \cap L^\infty([0,T],H^{1}(\R)),$ we obtain $\gv \in \boC^0([0,T],L^2(\R)).$ Next, we write
\begin{align*}
& \big\| F(\gv,\Phi)(t_1) - F(\gv,\Phi)(t_2) \big\|_{L^\infty(\R)} \\
              & \leq \|\gv(t_1) - \gv(t_2)\|_{L^2} \|\Phi(t_1)\|_{L^2} + \|\Phi(t_2) - \Phi(t_1)\|_{L^2} \|\gv(t_2)\|_{L^2}.
\end{align*}
Using the fact that $\Phi,$ $\gv \in \boC^0([0,T],L^2(\R)),$ we infer that $ F(\gv,\Phi) \in \boC^0([0,T],L^\infty(\R)).$ Then, by the second equation in \eqref{eq:v-Psi}, $\gv \in \boC^0([0,T],H^1(\R))$. This finishes the proof of this step. 
\end{proof}
This achieves the proof of Proposition \ref{prop:w-cont-Psi}.
\end{proof}
Finally, we give the proof of Proposition \ref{prop:w-cont-Q}.
\begin{proof}
In view of Proposition \ref{prop:w-cont-Psi}, it is sufficient to prove the convergence of $w_n$. The proof follows the arguments in the proof of \eqref{convfaible:bofn}. Let $\phi \in L^2(\R)$. We rely on \eqref{eq:w^*} to write
\begin{align*}
&\int_\R \big[ w^*(t,x) - w_n(t,x)\big] \phi(x) dx \\
& = 2 \int_\R \Im \Big( \frac{\Psi^*(t,x) \big( 1- 2 F(v^*,\Psi ^*)(t,x)\big)}{1-(v^*)^2(t,x)} - \frac{\Psi_n(t,x) \big( 1- 2 F(v_n,\Psi_n)(t,x)\big)}{1-(v_n)^2(t,x)} \Big) \phi(x) dx \\
& = 2 \int_\R \Im \Big( \frac{\Psi^*(t,x)}{1-(v^*)^2(t,x)} - \frac{\Psi_n(t,x)}{1-(v_n)^2(t,x)} \Big) \phi(x) dx \\
& \ \ - 4 \int_\R \Im \Big( \frac{\Psi^*(t,x) F(v^*,\Psi ^*)(t,x)}{1-(v^*)^2(t,x)} - \frac{\Psi_n(t,x) F(v_n,\Psi_n)(t,x)}{1-(v_n)^2(t,x)} \Big) \phi(x) dx, 
\end{align*}
for all $t \in [0,T].$ Then, we use the same arguments as in the proof of \eqref{convfaible:bofn} to show that the two last terms in the right-hand side go to $0$ when $n$ goes to $+\infty$. This finishes the proof of the proposition. 
\end{proof}
\subsection{Exponential decay of $\chi_c$}
In this subsection, we recall the explicit formula and some useful properties of the operator $\boH_c,$ and then study its negative eigenfunction $\chi_c$.
For $c \in (- 1, 1) \setminus \{ 0 \}$, the operator $\boH_c$ is given in explicit terms by
\begin{equation}
\label{def:boHc}
\boH_c(\eps) = \begin{pmatrix} \boL_c(\eps_v) + c^2 \frac{(1 + v_c^2)^2}{(1 - v_c^2)^3} \eps_v - c \frac{1 + v_c^2}{1 - v_c^2} \eps_w \\ - c \frac{1 + v_c^2}{1 - v_c^2} \eps_v + (1 - v_c^2) \eps_w \end{pmatrix},
\end{equation} 
where $\eps =(\eps_v,\eps_w)$, and $$\boL_c(\eps_v) = - \partial_x \bigg( \frac{\partial_x \eps_v}{1 - v_c^2} \bigg) + \Big( 1 - c^2 - (5 + c^2) v_c^2 + 2 v_c^4 \Big) \frac{\eps_v}{(1 - v_c^2)^2}.$$
In view of \eqref{def:boHc}, the operator $\boH_c$ is an isomorphism from $H^2(\R)\times L^2(\R) \cap \Span(\partial_x Q_c)^\perp$ onto $\Span(\partial_x Q_c)^\perp$. In addition, there exists a positive number $A_c$, depending continuously on $c$, such that
\begin{equation}
\label{inv:Hc}
\big\| \boH_c^{- 1}(f, g) \big\|_{H^{k + 2}(\R) \times H^k(\R)} \leq A_c \big\| (f, g) \big\|_{H^k(\R)^2},
\end{equation}
for any $(f, g) \in H^k(\R)^2 \cap \Span(\partial_x Q_c)^\perp$ and any $k \in \N$.

The following proposition establishes the coercivity of the quadratic form $H_c$ under suitable orthogonality conditions. 
\begin{prop}
\label{prop:coer-single}
Let $c \in (- 1, 1) \setminus \{0\}$. There exists a positive number $\Lambda_c$, depending only on $c$, such that
\begin{equation}
\label{eq:coer-Qc}
H_c(\eps) \geq \Lambda_c \| \eps \|_{H^1 \times L^2}^2,
\end{equation}
for any pair $\eps \in H^1(\R) \times L^2(\R)$ satisfying the two orthogonality conditions
\begin{equation}
\label{eq:ortho-Qc}
\langle \partial_x Q_c, \eps \rangle_{L^2 \times L^2} = \langle \chi_c, \eps \rangle_{L^2 \times L^2} = 0.
\end{equation}
Moreover, the map $c \mapsto \Lambda_c$ is uniformly bounded from below on any compact subset of $(- 1, 1) \setminus \{ 0 \}$.
\end{prop}
The proof relies on standard Sturm-Liouville theory (see e.g. the proof of Proposition 1 in \cite{DeLGr} for more details). 

Now, we turn to the analysis of the pair $\chi_c$.
\begin{lem}
\label{lem-chi-c}
The pair $\chi_c$ belongs to $\boC^{\infty}(\R) \times \boC^{\infty}(\R)$. In addition, there exist two positive numbers $A_c$ and $a_c$, depending continuously on $c$, such that $a_c > \sqrt{1-c^2}$ and
\begin{equation}
\label{dec:chi-c}
|\partial^k_x \chi_c| \leq A_c e^{-a_c|x|} \quad \textsl{on} \quad \R \quad \textsl{for} \quad k \in \{0,1,2\}.
\end{equation}

\end{lem}

\begin{proof}
We denote $\chi_c := (\zeta_c,\xi_c).$ Since $\boH_c(\chi_c) = - \tilde{\lambda}_c \chi_c$, we have the following system
\begin{equation}
\begin{split}
\label{syst:chi-c1}
&-\partial_x \bigg( \frac{\partial_x \zeta_c}{1 - v_c^2} \bigg) + \Big( 1 - c^2 - (5 + c^2) v_c^2 + 2 v_c^4 \Big) \frac{\zeta_c}{(1 - v_c^2)^2} + c^2 \frac{(1 + v_c^2)^2}{(1 - v_c^2)^3} \zeta_c\\
& - c \frac{1 + v_c^2}{1 - v_c^2} \xi_c = -\tilde{\lambda}_c \zeta_c, 
\end{split}
\end{equation}

\begin{equation}
\begin{split}
& c \frac{1 + v_c^2}{1 - v_c^2} \zeta_c = (1 - v_c^2 + \tilde{\lambda}_c) \xi_c.
\label{syst:chi-c2}
\end{split}
\end{equation}
It follows from standard elliptic theory that $\chi_c \in H^2(\R) \times L^2(\R)$. Since the coefficients in \eqref{syst:chi-c2} are smooth, bounded from below and above, we infer from a standard bootstrap argument that $\chi_c \in \boC^\infty(\R) \times \boC^\infty(\R)$. Notice in particular that, by the Sobolev embedding theorem, $\chi_c$ and $\partial_x \chi_c$ are bounded on $\R$. Then, we deduce from the first statement in \eqref{part_vc} that\footnote{The notation $\boO (v_c^2)$ refers to a quantity, which is bounded by $A_c v_c^2$ (pointwisely), where the positive number $A_c$ depends only on $c$.}
\begin{align}
\label{syst:1}
&-\partial_{xx} \zeta_c + (1+\tilde{\lambda}_c) \zeta_c - c \xi_c = \boO (v_c^2) , \\
& \zeta_c =  \frac{1 + \tilde{\lambda}_c}{c} \xi_c + \boO(v_c^2).
\label{syst:2}
\end{align}
Note that, we have
\begin{equation}
\label{borne:v_c}
 B_c \exp(-\sqrt{1-c^2}|x|) \leq v_c(x) \leq A_c \exp(-\sqrt{1-c^2}|x|) \quad \textsl{for all} \quad x \in \R,
\end{equation}
where $B_c$ and $A_c$ are two positive numbers.

In order to prove \eqref{dec:chi-c}, we now introduce \eqref{syst:2} into \eqref{syst:1} to obtain,
\begin{align}
\label{systbis:1}
&-\partial_{xx} \zeta_c + b^2_{c} \zeta_c = \boO(\exp(-2\sqrt{1-c^2}|x|)),\\
& \xi_c =  \frac{c}{1 + \tilde{\lambda}_c} \zeta_c + \boO(\exp(-2\sqrt{1-c^2}|x|)),
\label{systbis:2}
\end{align}
with $b^2_{c} = \frac{1-c^2+ 2 \tilde{\lambda}_c + (\tilde{\lambda}_c)^2}{1+\tilde{\lambda}_c} > 1-c^2.$ Next, we set
\begin{equation}
\label{ODE:xi_c}
g_c := -\partial_{xx} \zeta_c + b^2_{c} \zeta_c ,
\end{equation}
so that $g_c(x)=\boO(\exp(-2\sqrt{1-c^2}|x|))$ for all $x \in \R$. Using the variation of constant method, we obtain, for all $x \in \R$,
$$ \zeta_c(x) = A(x) e^{b_c x} + A_c e^{b_c x}  + B(x) e^{-b_c x} + B_c e^{-b_c x},$$
with 
$$ A(x) = \frac{-1}{2 b_c}\int_0^x e^{-b_c t} g_c(t) dt ,$$
and
$$ B(x) = \frac{-1}{2 b_c}\int_0^x e^{b_c t} g_c(t) dt.$$
Since $\zeta_c \in L^2(\R)$, this leads to
$$\zeta_c(x)= \boO\big( \exp\big(-2\sqrt{1-c^2} |x| \big) +  \exp(-b_c |x|)\big).$$
Hence, we can take $a_c= \min\{2\sqrt{1-c^2},b_c\}$ and invoke \eqref{syst:2} to obtain \eqref{dec:chi-c} for $k=0$. Using \eqref{eq:vc}, \eqref{part_vc}, \eqref{syst:chi-c1}, \eqref{syst:chi-c2} and \eqref{borne:v_c}, we extend  \eqref{dec:chi-c} to $k \in {1,2}$.

\end{proof}

\begin{merci}
I would like to thank both my supervisors R. Côte and P. Gravejat for their support throughout the time I spent to finish this manuscript, offering invaluable advices.
I am also grateful to Y. Martel for interesting and helpful discussions. I am also indebted to the anonymous referee for valuable comments and remarks. This work is supported by a PhD grant from "Région Ile-de-France" and is partially sponsored by the project ``Schr\"odinger equations and applications'' (ANR-12-JS01-0005-01) of the Agence Nationale de la Recherche.
\end{merci}
\bibliographystyle{plain}

\begin{thebibliography}{10}

\bibitem{BetGrSa2}
F.~B\'ethuel, P.~Gravejat, and J.-C. Saut.
\newblock Existence and properties of travelling waves for the
  {Gross}-{Pitaevskii} equation.
\newblock In A.~Farina and J.-C. Saut, editors, {\em Stationary and time
  dependent {Gross}-{Pitaevskii} equations}, volume 473 of {\em Contemp.
  Math.}, pages 55--104. Amer. Math. Soc., Providence, RI, 2008.

\bibitem{BetGrSm2}
F.~B\'ethuel, P.~Gravejat, and D.~Smets.
\newblock Asymptotic stability in the energy space for dark solitons of the
  {Gross}-{Pitaevskii} equation.
\newblock {\em Ann. Sci. \'Ec. Norm. Sup\'er.}, (4):in press, 2014.

\bibitem{BetGrSm1}
F.~B\'ethuel, P.~Gravejat, and D.~Smets.
\newblock Stability in the energy space for chains of solitons of the
  one-dimensional {Gross}-{Pitaevskii} equation.
\newblock {\em Ann. Inst. Four.}, 1(64):19--70, 2014.

\bibitem{BuslPer1}
V.S. Buslaev and G.~Perelman.
\newblock Scattering for the nonlinear {Schr\"odinger} equation: states close
  to a soliton.
\newblock {\em St. Petersburg Math. J.}, 4(6):1111--1142, 1993.

\bibitem{BuslPer2}
V.S. Buslaev and G.~Perelman.
\newblock On the stability of solitary waves for nonlinear {Schr\"odinger}
  equations.
\newblock In N.N. Uraltseva, editor, {\em Nonlinear evolution equations},
  volume 164 of {\em American Mathematical Society Translations-Series 2},
  pages 75--98. Amer. Math. Soc., Providence, RI, 1995.

\bibitem{BuslSul1}
V.S. Buslaev and C.~Sulem.
\newblock On asymptotic stability of solitary waves for nonlinear
  {Schr\"odinger} equations.
\newblock {\em Ann. Inst. Henri Poincar\'e, Analyse Non Lin\'eaire},
  20(3):419--475, 2003.

\bibitem{Cuccagn1}
S.~Cuccagna.
\newblock Stabilization of solutions to nonlinear {Schr\"odinger} equations.
\newblock {\em Commun. Pure Appl. Math.}, 54(9):1110--1145, 2001.

\bibitem{Cuccagn2}
S.~Cuccagna.
\newblock On asymptotic stability of ground states of {NLS}.
\newblock {\em Rev. Math. Phys.}, 15(8):877--903, 2003.

\bibitem{deLaire4}
A.~de~Laire.
\newblock Minimal energy for the traveling waves of the {Landau}-{Lifshitz}
  equation.
\newblock {\em SIAM J. Math. Anal.}, 46(1):96--132, 2014.

\bibitem{DeLGr}
A.~de~Laire and P.~Gravejat.
\newblock Stability in the energy space for chains of solitons of the
  {L}andau-{L}ifshitz equation.
\newblock {\em J. Differential Equations}, 258(1):1--80, 2015.

\bibitem{EsKePoV5}
L.~Escauriaza, C.E. Kenig, G.~Ponce, and L.~Vega.
\newblock Hardy's uncertainty principle, convexity and {Schr\"odinger}
  evolutions.
\newblock {\em J. Eur. Math. Soc.}, 10(4):883--907, 2008.

\bibitem{GangSig1}
Z.~Gang and I.M. Sigal.
\newblock Relaxation of solitons in nonlinear {Schr\"odinger} equations with
  potential.
\newblock {\em Adv. in Math.}, 216(2):443--490, 2007.

\bibitem{Goubet96regularityof}
O.~Goubet and L.~Molinet.
\newblock Regularity of the attractor for the weakly damped nonlinear
  schrödinger equations.
\newblock {\em Applicable Anal}, pages 99--119, 1996.

\bibitem{GriShSt1}
M.~Grillakis, J.~Shatah, and W.A. Strauss.
\newblock Stability theory of solitary waves in the presence of symmetry {I}.
\newblock {\em J. Funct. Anal.}, 74(1):160--197, 1987.

\bibitem{GuoDing0}
B.~Guo and S.~Ding.
\newblock {\em Landau-{L}ifshitz equations}, volume~1 of {\em Frontiers of
  Research with the Chinese Academy of Sciences}.
\newblock World Scientific, Hackensack, New Jersey, 2008.

\bibitem{HubeSch0}
A.~Hubert and R.~Sch\"afer.
\newblock {\em Magnetic domains: the analysis of magnetic microstructures}.
\newblock Springer-Verlag, Berlin-Heidelberg-New York, 1998.

\bibitem{JerrSme1}
R.~Jerrard and D.~Smets.
\newblock On {Schr\"odinger} maps from $\mathbb{T}^1$ to $\mathbb{S}^2$.
\newblock {\em Ann. Sci. Ec. Norm. Sup.}, 45(4):635--678, 2012.

\bibitem{KenPoVe9}
C.~E. Kenig, G.~Ponce, and L.~Vega.
\newblock On unique continuation for nonlinear {S}chr\"odinger equations.
\newblock {\em Comm. Pure Appl. Math.}, 56(9):1247--1262, 2003.

\bibitem{KosIvKo1}
A.M. Kosevich, B.A. Ivanov, and A.S. Kovalev.
\newblock Magnetic solitons.
\newblock {\em Phys. Rep.}, 194(3-4):117--238, 1990.

\bibitem{LandLif1}
L.D. Landau and E.M. Lifshitz.
\newblock On the theory of the dispersion of magnetic permeability in
  ferromagnetic bodies.
\newblock {\em Phys. Zeitsch. der Sow.}, 8:153--169, 1935.

\bibitem{LinaPon0}
F.~Linares and G.~Ponce.
\newblock {\em Introduction to nonlinear dispersive equations}.
\newblock Universitext. Springer-Verlag, Berlin-Heidelberg-New York, 2009.

\bibitem{Martel2}
Y.~Martel.
\newblock Linear problems related to asymptotic stability of solitons of the
  generalized {KdV} equations.
\newblock {\em SIAM J. Math. Anal.}, 38(3):759--781, 2006.

\bibitem{MartMer6}
Y.~Martel and F.~Merle.
\newblock Asymptotic stability of solitons of the {gKdV} equations with general
  nonlinearity.
\newblock {\em Math. Ann.}, 341(2):391--427, 2008.

\bibitem{MartMer5}
Y.~Martel and F.~Merle.
\newblock Refined asymptotics around solitons for the {gKdV} equations with a
  general nonlinearity.
\newblock {\em Disc. Cont. Dynam. Syst.}, 20(2):177--218, 2008.

\bibitem{Mizumac1}
T.~Mizumachi.
\newblock Large time asymptotics of solutions around solitary waves to the
  generalized {Korteweg}-de {Vries} equations.
\newblock {\em SIAM J. Math. Anal.}, 32(5):1050--1080, 2001.

\bibitem{PegoWei1}
R.L. Pego and M.I. Weinstein.
\newblock Asymptotic stability of solitary waves.
\newblock {\em Commun. Math. Phys.}, 164(2):305--349, 1994.

\bibitem{SoffWei1}
A.~Soffer and M.I. Weinstein.
\newblock Multichannel nonlinear scattering theory for nonintegrable equations.
\newblock In P.~Lochak M.~Balabane and C.~Sulem, editors, {\em Integrable
  Systems and Applications}, volume 342 of {\em Lecture Notes in Physics},
  pages 312--327. Springer-Verlag Berlin Heidelberg, 1989.

\bibitem{SoffWei2}
A.~Soffer and M.I. Weinstein.
\newblock Multichannel nonlinear scattering for nonintegrable equations.
\newblock {\em Commun. Math. Phys.}, 133(1):119--146, 1990.

\bibitem{SoffWei3}
A.~Soffer and M.I. Weinstein.
\newblock Multichannel nonlinear scattering for nonintegrable equations {II}.
  {The} case of anisotropic potentials and data.
\newblock {\em J. Diff. Eq.}, 98(2):376--390, 1992.

\end{thebibliography}

\end{document}